\documentclass[11pt,a4paper]{amsart}

\usepackage[english]{babel}  
\usepackage[utf8]{inputenc}
\usepackage{fontenc}
\usepackage[normalem]{ulem}
\usepackage{amsmath,amsthm,amsfonts,amssymb}    
\allowdisplaybreaks
\usepackage{nccmath}
\usepackage{verbatim}
\usepackage{multicol,multirow}     
\usepackage[numbers]{natbib}   
\renewcommand\labelenumi{(\roman{enumi})}
\renewcommand\theenumi\labelenumi
\usepackage{marvosym}
\usepackage{mathtools}
\mathtoolsset{showonlyrefs} 
\usepackage{verbatim}
\usepackage{dsfont}
\usepackage{float}
\usepackage{cancel}
\usepackage{todonotes}
\usepackage[pdftex,
bookmarksopen,
colorlinks,
linkcolor=black,
urlcolor=black,
citecolor=black
]{hyperref}
\hypersetup{pdfauthor={A. Lang, A. Papini, and V. Schwarz}}
\usepackage{color}
\makeatletter
\def\namedlabel#1#2{\begingroup
    #2%
    \def\@currentlabel{#2}%
    \phantomsection\label{#1}\endgroup
}
\usepackage{a4wide}
\usepackage{caption}

\usepackage{graphicx}
\usepackage{subcaption}

\usepackage{cancel}

\makeatletter
\renewcommand\p@subfigure{\thefigure.(\arabic{subfigure}\expandafter)\@gobble}
\makeatother

\theoremstyle{plain}
\newtheorem{theorem}{Theorem}[section]
\newtheorem{lemma}[theorem]{Lemma}
\newtheorem{proposition}[theorem]{Proposition}
\newtheorem{corollary}[theorem]{Corollary}

\theoremstyle{definition}
\newtheorem{assumption}[theorem]{Assumption}

\newtheorem{remark}[theorem]{Remark}
\newtheorem{example}[theorem]{Example}

\newcommand{\E}{{\mathbb{E}}}

\newcommand{\N}{{\mathbb{N}}}

\renewcommand{\P}{{\mathbb{P}}}

\newcommand{\R}{{\mathbb{R}}}

\newcommand{\diff}{\mathop{}\!\mathrm{d}}

\renewcommand{\S}{{\mathbb{S}}}

\newcommand{\Id}{{\operatorname{Id}}}

\newcommand{\realSH}{{Y}}
\newcommand{\mean}{{\widetilde m}}
\newcommand{\drift}{{b}}
\newcommand{\driftlm}{{b_{\ell,m}}}
\newcommand{\meanlm}{{\widetilde m_{\ell,m}}}
\newcommand{\hatdrift}{{\hat b}}
\newcommand{\hatmeanlm}{{\hat m_{\ell,m}}}
\newcommand{\hatmean}{{\hat m}}

\newcommand{\varlm}{v_{\ell,m}}
\newcommand{\hatvarlm}{{\hat v_{\ell,m}}}
\newcommand{\hatvar}{{\hat v}}
\newcommand{\std}{{\tilde v}}
\newcommand{\cova}{{Q}}

\newcommand{\convLlm}{{L^{\Delta_{\S^2}}_{\ell,m}}}
\newcommand{\convL}{{L^{\Delta_{\S^2}}}}

\newcommand{\convW}{{W^{\Delta_{\S^2}}}}

\newcommand{\convN}{{N^{\Delta_{\S^2}}}}
\newcommand{\discConvL}{{L^{\Delta_{\S^2}(\kappa,h)}}}
\newcommand{\discConvLlm}{{L^{\Delta_{\S^2}(\kappa,h)}_{\ell,m}}}
\newcommand{\discConvW}{{W^{\Delta_{\S^2}(\kappa,h)}}}

\newcommand{\discConvN}{{N^{\Delta_{\S^2}(\kappa,h)}}}

\newcommand{\approxOp}{{R(h{\Delta_{\S^2}}) }}
\newcommand{\proOp}{{P_\kappa }}

\usepackage[normalem]{ulem}

\begin{document}

\bibliographystyle{plainnat}

\title[Approximation Lévy-driven heat equations on the sphere]{Approximation of the Lévy-driven stochastic heat equation on the sphere}

\author[A.~Lang]{Annika Lang} 
\author[A.~Papini]{Andrea Papini} \address[Annika Lang, Andrea Papini]{Department of Mathematical Sciences, Chalmers University of Technology and University of Gothenburg, S--412 96 G\"oteborg, Sweden.} \email[]{annika.lang@chalmers.se, andreapa@chalmers.se}
\author[V.~Schwarz]{Verena Schwarz} \address[Verena Schwarz]{Mathematical Institute, University of Oxford, UK--OX2 6GG Oxford, United Kingdom.} \email[]{verena.schwarz@\text{maths.ox.ac}.uk} 

\thanks{Acknowledgment: The authors thank the anonymous referee for their careful reading and constructive comments, which led to a major improvement of the results. This research was funded by the Austrian Science Fund (FWF) [10.55776/DOC78], by the European Union (ERC, StochMan, 101088589), by the Swedish Research Council (VR) through grant no. 2020-04170, by the Wallenberg AI, Autonomous Systems and Software Program (WASP) funded by the Knut and Alice Wallenberg Foundation, and by the Chalmers AI Research Center (CHAIR). For open access purposes, the authors have applied a CC BY public copyright license to any author-accepted manuscript version arising from this submission. Views and opinions expressed are however those of the author(s) only and do not necessarily reflect those of the European Union or the European Research Council Executive Agency. Neither the European Union nor the granting authority can be held responsible for them. The first version of this article was written while V.~Schwarz was affiliated with the Department of Statistics, University of Klagenfurt, Austria.}


\begin{abstract}
The stochastic heat equation on the sphere driven by an additive square-integra\-ble L\'evy process
is approximated by a spectral method in space and forward and backward Euler--Maruyama schemes in time. New sharp regularity results are proven for its solution. The spectral approximation is based on a truncation of the series expansion with respect to the spherical harmonic functions. For a given regularity of the initial condition and two different settings of regularity for the driving noise, strong convergence rates for the spectral approximation and for the Euler--Maruyama methods are proven. Moreover, weak rates of up to twice the strong rates are shown for Wiener noise and for L\'evy noise with independent and up to a deterministic decay rate factor identically distributed components, while for general square-integrable L\'evy noise the weak rate depends explicitly on the polynomial order of the test function. Numerical simulations confirm the theoretical results.
\end{abstract}

\maketitle

\section{Introduction}

Stochastic partial differential equations (SPDEs) play a central role in modeling systems influenced by random effects across space and time. Traditionally, much of the foundational work in this area has focused on SPDEs, mostly on Euclidean spaces, driven by Wiener noise, leveraging its mathematical tractability and the rich theory developed around Gaussian processes \cite{DaPrato_Zabczyk_1992,peszat2007, PR07}. 
Numerical methods for such SPDEs have been developed and analyzed for more than thirty years by now, with references given for example in the books \cite{JentzMono,Lord_Powell_Shardlow_2014}.

However, many real-world phenomena, ranging from turbulent flows on Earth and stellar fluids to biological and financial systems \cite{fogedby1998levy_force, Woy2001, droge2003sep_sde, Ruffolo_2004, frontiers2021random_fields, cont2003financial}, exhibit abrupt non-Gaussian behavior that cannot be adequately captured by just a Gaussian noise. More so, applications on the Earth motivate to extend the theory to spheres \cite{marinucci2011}. These observations prompt the need to consider more general settings; to our knowledge, literature on surfaces is still lacking with first results in the Gaussian setting on the sphere given in \cite{lang2015, KAZASHI2019, LeGia2019, cohen2022, lang2023, ALODAT2024, AFL25+, CDGL26+}.

In this work, we take a step beyond the classical setting by studying SPDEs driven by square-integrable Lévy processes, which naturally can incorporate jumps \cite{Hausenblas2005, peszat2007, Haus12}. We consider the stochastic heat equation on the sphere~$\S^2$, on a complete filtered probability space $(\Omega, \mathcal{F}, (\mathcal{F}_t)_{t\in[0,T]},\mathbb{P})$ and a finite time interval $[0,T]$, $T<\infty$, i.e.,
\begin{equation}\label{eq:HeatEq}
    \diff X(t) = \Delta_{\S^2} X(t) \diff t + \diff L(t)
\end{equation}
with $\mathcal{F}_0$-measurable initial condition $X(0)=X^0$ driven by an infinite-dimensional L\'evy process $L$, which is independent of $X^0$ and will be introduced in more detail in Section~\ref{subse:noise_intro}.

These processes offer a more realistic \cite{reynolds2009scale_free, bott10, humphries2010environmental, Palyulin_2019, Herzog20} and flexible framework for modeling complex stochastic dynamics. The use of Lévy noise introduces a range of analytical and numerical challenges, but also opens new pathways for understanding the behavior of solutions to SPDEs under more general sources of randomness.

Our contributions follow two main directions. First, we prove strong convergence estimates for the spectral approximation in space and the Euler--Maruyama approximation in time with non-isotropic noise characterised by its Sobolev regularity. This generalises the results of \cite{lang2015,lang2023}. Second, we establish new regularity results including a sharp characterization of the spatial regularity of the solution in $L^{p}(\Omega)$ for $p\geq 2$, which for general jump noise is strictly below the second-moment regularity, and weak convergence results for the spectral and temporal approximation.
Our method of proof follows the duality idea of~\cite{cohen2022} for the stochastic wave equation on the sphere and is inspired by~\cite{AKL16}, which uses Malliavin calculus. Weak convergence in Euclidean domains and Hilbert spaces has been shown during the last two decades with Kolmogorov equations, a mild It\^o formula, and Malliavin calculus (e.g., in \cite{DP09, debussche2011weak, kovacs2012weak, LS13, andersson2016weak, AKL16, brehier2017approximation, BHS18, CJK19, refId0}). Results on the sphere are rare, where we are only aware of the results in~\cite{cohen2022} for the stochastic wave equation and mean and second moment estimates for the stochastic heat equation in~\cite{lang2023}. These are the first weak convergence results for the stochastic heat equation on the sphere with general test functions, which have neither been considered for L\'evy noise nor for the classical Wiener case before. Our test functions need to be only once Fr\'echet differentiable and of polynomial growth, which are rather weak assumptions compared to the literature.

The paper is structured as follows: In Section~\ref{Sec:set}, we describe the functional setting and the precise driving noise of \eqref{eq:HeatEq} and prove properties of the solution to~\eqref{eq:HeatEq}, i.e., we compute its first and second moments and prove sharp spatial regularity results for the solution. In Section~\ref{Sec:SpectApp}, we prove strong and weak convergence of the spectral approximation, where the rates depend on the smoothness of the L\'evy noise. In Section~\ref{Sec:4}, the Euler--Maruyama scheme is analyzed, showing strong and weak convergence rates. Finally, in Section~\ref{sec:Num}, numerical simulations confirm the theoretical results for different examples of driving L\'evy processes. The codes that were used to generate the samples and the convergence plots for the numerical examples are available at \cite{lang2025levy}.

\section{Setting}\label{Sec:set}

In this section we first introduce the necessary notation, then we define the Lévy noise and derive the solution to \eqref{eq:HeatEq}. 
We denote by $\S^2$ the unit sphere in $\R^3$, i.e., 
\[\S^2= \{x\in\R^3: \|x\| = 1\},\]
where $\|\cdot\|$ denotes the Euclidean norm. We equip $\S^2$ with the geodesic metric given for all $x,y\in\R^3$ by $d(x,y) = \arccos\langle x,y \rangle_{\R^3}$ and couple the Cartesian coordinates $y\in\S^2$ to the polar coordinates $(\vartheta,\varphi)\in[0,\pi]\times[0,2\pi)$ by the transformation
\[y=(\sin(\vartheta)\cos(\varphi),\sin(\vartheta)\sin(\varphi),\cos(\vartheta)).\]
Let $\sigma$ be the Lebesgue measure on the sphere, which has the representation
\[\diff \sigma(y) = \sin(\vartheta)\diff \vartheta \diff \varphi,\]
and denote by $\mathcal{B}(\S^2)$ the Borel $\sigma$-algebra of $\S^2$. Note that $(\S^2,\mathcal{B}(\S^2),\sigma)$ is a measure space and $L^2(\S^2)$ is a Hilbert space with scalar product $\langle \cdot,\cdot \rangle_{L^2(\S^2)}$, which is defined for all $f,g\in L^2(\S^2)$ by
\[\langle f,g \rangle_{L^2(\S^2)}=\int_{\S^2} f(y)g(y) \diff \sigma(y).\]

We represent the real-valued spherical harmonic functions by $\mathcal{\realSH}= (\realSH_{\ell,m})_{\ell\in\N_0, m\in\{-\ell,\ldots,\ell\}}$, which consist of $\realSH_{\ell,m}\colon[0,\pi]\times[0,2\pi) \to \R$ given by
\begin{equation}
    \realSH_{\ell,m}(\vartheta,\varphi) = 
    \begin{cases}
        \sqrt{2}(-1)^m \sqrt{\frac{2\ell+1}{4\pi}\frac{(\ell-|m|)!}{(\ell+|m|)!}} P_{\ell,|m|}(\cos(\vartheta))\sin(|m|\varphi), & \text{for } m<0 \\[3mm]
     \sqrt{\frac{2\ell+1}{4\pi}} P_{\ell,m}(\cos(\vartheta)), 
     & \text{for } m=0 \\[3mm]
         \sqrt{2}(-1)^m \sqrt{\frac{2\ell+1}{4\pi}\frac{(\ell-m)!}{(\ell+m)!}} P_{\ell,m}(\cos(\vartheta))\cos(m\varphi), & \text{for }  m>0.
    \end{cases}
\end{equation}
Here, $(P_{\ell,m})_{\ell\in\N_0, m\in\{0,\ldots,\ell\}}$ are the associated Legendre polynomials, see for instance \cite{szeg1939orthogonal}. The real-valued spherical harmonics form an orthonormal basis of $L^2(\S^2)$, see \cite{BLANCO199719}.

The spherical Laplacian or Laplace--Beltrami operator is defined by 
\begin{equation}\label{eq:SphericalLaplacian}
    \Delta_{\S^2} = (\sin(\vartheta))^{-1} \frac{\partial}{\partial \vartheta}\Big(\sin(\vartheta)\frac{\partial}{\partial \vartheta} \Big) +  (\sin(\vartheta))^{-2} \frac{\partial^2}{\partial \varphi^2},    
\end{equation}
see \cite[Chapter 6]{byerly1893elementary}. For all $\ell\in\N_0=\N\cup\{0\}$, $m\in\{-\ell,\ldots,\ell\}$, we have that with $\lambda_\ell = \ell(\ell+1)$,
\begin{equation}\label{eq:EigenvaluesRealSH}
    \Delta_{\S^2} \realSH_{\ell,m} 
    =-\lambda_\ell \realSH_{\ell,m}.
\end{equation}
This is a consequence of the fact that the real-valued spherical harmonics are a linear combination of the complex-valued spherical harmonics with the same eigenvalues.

Next, we define the Sobolev spaces $H^\eta(\S^2)$ on the unit sphere with smoothness index $\eta\in\R$, which are equivalent to Bessel potential spaces, see \cite{AnMan} for further details. 
For $\eta\in\R_+$, we define the positive Sobolev space by
\begin{equation}\label{eq:DefSpaceBesselPot}
    H^\eta(\S^2) = (\Id-\Delta_{\S^2})^{-\eta/2}L^2(\S^2)  =\{f\colon\S^2\rightarrow \mathbb{R}:\ \exists\ g\in L^2(\S^2)\ \text{s.t.}\ (\operatorname{Id}-\Delta_{\S^2})^{\eta/2}f=g\}.
\end{equation}
Note that $f\in L^2(\S^2)$ can be rewritten as $f = \sum_{\ell=0}^\infty \sum_{m=-\ell}^\ell f_{\ell,m} \realSH_{\ell,m}$ with real-valued coefficients $f_{\ell,m}$. Hence, for $f\in H^\eta(\S^2)$,
\begin{equation}
    (\Id-\Delta_{\S^2})^{\eta/2} f = \sum_{\ell=0}^\infty \sum_{m=-\ell}^\ell (1+\lambda_\ell  )^{\eta/2} f_{\ell,m} \realSH_{\ell,m}\in L^2(\S^2).
\end{equation}
The inner product in $H^\eta(\S^2)$ is given for all $f,g\in H^\eta(\S^2)$ by
\begin{equation}\label{eq:DefSpaceBesselPotSP}
    \langle f,g \rangle_{H^\eta(\S^2)} = \langle (\Id-\Delta_{\S^2})^{\eta/2} f, (\Id-\Delta_{\S^2})^{\eta/2} g\rangle_{L^2(\S^2)}.
\end{equation}
For $\eta\in \R_-$, we define the negative Sobolev space as the space of distributions generated by
$$
H^\eta(\mathbb{S}^2)=\left\{u=(\Id-\Delta_{\mathbb{S}^2})^\kappa v,\ v\in H^{2\kappa +\eta}(\mathbb{S}^2)\right\},
$$
where $\kappa\in\mathbb{N}$ is the smallest integer such that $2\kappa+\eta>0$. In this case, we define the norm to be
$$
\|u\|_{H^\eta(\mathbb{S}^2)}=\|v\|_{H^{2\kappa+\eta}(\mathbb{S}^2)}.
$$
Finally, we set $H^0(\mathbb{S}^2)=L^2(\mathbb{S}^2)$.  
As the subsequent sections involve the (formal) series expansion of the noise in terms of the real-valued spherical harmonic functions with coefficients in $L^2(\Omega, \R)$, we interpret the previously defined Hilbert spaces within the framework of a Gelfand triple, i.e. we consider for $H=L^2(\S^2)$ and $\eta\geq0$ the following embeddings
\begin{equation}\label{eq:g-3}
V=H^\eta(\S^2) \hookrightarrow H\simeq H^*\hookrightarrow V^*\simeq H^{-\eta}(\S^2).
\end{equation}
For useful details on the introduced spaces and the Bessel potentials, we refer, for instance, to \cite{AnMan}. 
We expand $u \in H^\eta(\S^2)$ in the spherical harmonics
\begin{equation}\label{eq:sobolev-harmonic-expansion}
  u=\sum_{\ell=0}^\infty\sum_{m=-\ell}^\ell
        u_{\ell,m}\realSH_{\ell,m},
  \qquad
  \|u\|_{H^\eta(\S^2)}^2
  =\sum_{\ell=0}^\infty\sum_{m=-\ell}^\ell
        (1+\lambda_\ell)^\eta|u_{\ell,m}|^2
\end{equation}
with convergence in~$H^\eta(\S^2)$, where the coefficients $u_{\ell,m}$ are given for $\eta \ge 0$ by
$$
u_{\ell,m}=
  \langle u,\realSH_{\ell,m}\rangle_{L^2(\S^2)},
$$
and for $-\eta\in (-1,0)$ and $u\in H^{-\eta}(\S^2)$, set
\[
  u_{\ell,m}=
  \langle u,\realSH_{\ell,m}\rangle_{H^{-\eta}(\S^2),H^{\eta}(\S^2)},
\]
where the dual pairing is the extension of the $L^2(\S^2)$ inner product. Since
$\realSH_{\ell,m}\in C^\infty(\S^2)\subset H^{\eta}(\S^2)$ for every $\eta\in \R$,
this coefficient is well-defined.

In this article, we use the Lebesgue--Bochner spaces $L^p(\Omega,H^\eta(\S^2))$ for $p\geq1$, $\eta\in \R$ with the norm
$$
\|X\|_{L^p(\Omega,H^\eta(\S^2))}=\mathbb{E}[\|X\|^p_{H^\eta(\S^2)}]^{1/p}
$$
similarly to \cite{lang2015}. In the following, we use a generic constant $C\in(0,\infty)$ to simplify notation. This constant is allowed to change from line to line. In the convergence proofs, it is always independent of the investigated quantities and we give the dependence on the involved parameters where relevant. 
The last thing to introduce from \eqref{eq:HeatEq} before being able to solve it is the driving noise.

\subsection{The driving L\'evy process}
\label{subse:noise_intro}

Let us introduce the driving Lévy process $L=(L(t))_{t\in[0,\infty)}$ with its properties in what follows. In the framework of~\cite[Definition 4.1]{peszat2007}, we assume that $L$ belongs to the class of square-integrable Lévy processes with values in $H^\eta(\S^2)$, $\eta\in(-1,\infty)$, i.e., $L(t) \in L^2(\Omega,H^\eta(\S^2))$. More specifically, $L$ has stationary and independent increments, is stochastically continuous, and satisfies $L(0)=0$. By \cite[Theorem 4.3]{peszat2007}, $L$ has a càdlàg modification and by \cite[Theorem 4.44]{peszat2007}, the mean $\E[L(t)] = t \mean \in H^\eta(\S^2)$ exists.

Let $\nu^{(\eta)}$ be the L\'evy measure of $L$ on
$\chi_\eta=H^\eta(\S^2)\setminus\{0\}$, let $N^{(\eta)}$ be the corresponding Poisson
random measure, and let
\[
   \widetilde N^{(\eta)}(\diff t,\diff\zeta)
   =N^{(\eta)}(\diff t,\diff\zeta)-\nu^{(\eta)}(\diff\zeta)\diff t
\]
be the compensated Poisson random measure associated to $N^{(\eta)}$. The L\'evy--It\^o decomposition in the Hilbert space $H^\eta(\S^2)$ consists of a
deterministic drift, a Gaussian part, and a jump part, see e.g.,
\cite[Theorem~2.4.16]{applebaum2007levy} and
\cite[Theorem~4.23]{peszat2007}. Since $L$ is square-integrable,
\[
  \int_{\chi_\eta}\|\zeta\|_{H^\eta(\S^2)}^2\,\nu^{(\eta)}(\diff\zeta)<\infty,
\]
see \cite[Theorem 4.49]{peszat2007}, and we use throughout the equivalent representation
\begin{align}\label{ass3}
\begin{split}
L(t)
   &=t\mean+W^{(\eta)}(t)
     +\int_0^t\!\!\int_{\chi_\eta}
        \zeta\,\widetilde N^{(\eta)}(\diff s,\diff\zeta),
\end{split}
\end{align}
where $W^{(\eta)}$ is a $\cova_W^{(\eta)}$-Wiener process in $H^\eta(\S^2)$ and $\int_{0}^{t}\!\!\int_{\chi_\eta}\zeta\,\tilde N^{(\eta)}(\diff s,\diff\zeta)$ is a martingale independent of $W^{(\eta)}$ due to \cite[Proposition 4.10]{Rud04}. Observe that the stochastic integral in
\eqref{ass3} is understood as an $L^2(\Omega,H^\eta(\S^2))$ limit.

We expand $L$ in the spherical harmonics to obtain
\begin{equation}\label{eq:KL-expansion_Levy}
   L(t)=\sum_{\ell=0}^\infty\sum_{m=-\ell}^\ell
            L_{\ell,m}(t)\realSH_{\ell,m},
\end{equation}
which converges in $L^2(\Omega,H^\eta(\S^2))$ by \cite[Theorem 4.39, Theorem 4.44]{peszat2007}. The coefficients given by
\begin{equation}\label{eq:def:Llm}
   L_{\ell,m}(t)
   =\begin{cases}
      \langle L(t),\realSH_{\ell,m}\rangle_{L^2(\S^2)},& \text{ for } L\in L^2(\Omega,H^\eta(\S^2)), \eta\geq0,\\
       \langle L(t),\realSH_{\ell,m}\rangle_{H^{-\eta}(\S^2),H^{\eta}(\S^2)},&\text{ for } L\in L^2(\Omega,H^{-\eta}(\S^2)), -\eta\in(-1,0)
   \end{cases}
\end{equation}
are real-valued c\`adl\`ag L\'evy processes, see \cite[Section 4.8]{peszat2007}.

We set $\meanlm=\E[L_{\ell,m}(1)]$, $  \varlm=\operatorname{Var}(L_{\ell,m}(1))=\E[(L_{\ell,m}(1) - \E[L_{\ell,m}(1)])^2]$, for all $\ell\in \N_0$, and $m\in\{-\ell,\dots,\ell\}$. Then
\begin{equation}\label{eq:mean_Levy}
  t\mean=\E[L(t)]= \sum_{\ell=0}^\infty \sum_{m=-\ell}^\ell \E[L_{\ell,m}(t)] \realSH_{\ell,m}
  =t\sum_{\ell=0}^\infty\sum_{m=-\ell}^\ell
       \meanlm\realSH_{\ell,m}
  \quad\text{in }H^\eta(\S^2),
\end{equation}
and 
\begin{equation}\label{eq:trace-covariance-eta}
      \E[\|L(t) - \E[L(t)]\|_{H^\eta(\S^2)}^2] = \| \std \|_{H^\eta(\S^2)}^2
      =\sum_{\ell=0}^\infty\sum_{m=-\ell}^\ell
           (1+\lambda_\ell)^\eta\varlm<\infty,
    \end{equation}
where $\std$ is the weighted series of coordinate standard deviation defined by
\[
   \std=\sum_{\ell=0}^\infty\sum_{m=-\ell}^\ell
       \sqrt{\varlm}\,\realSH_{\ell,m}\, \in\,H^\eta(\S^2).
\]

Let us recall some properties of the Lévy processes $L_{\ell,m}$, for each $\ell\in\N_0,\ m\in\{-\ell,\dots,\ell\}$, in the decomposition \eqref{eq:KL-expansion_Levy}. 
Note that $L_{\ell,m}$ is a real-valued square-integrable L\'evy process with characteristic triplet $(\driftlm,\sigma^{2}_{\ell,m},\nu_{\ell,m})$ and
\begin{equation}\label{lmes-lm}
\nu_{\ell,m}=\bigl((\pi^{(\eta)}_{\ell,m})_{\#}\nu^{(\eta)}\bigr)\big|_{\mathbb R\setminus\{0\}},
\qquad \pi^{(\eta)}_{\ell,m}\colon H^{\eta}(\mathbb S^{2})\to\mathbb R,\qquad \pi^{(\eta)}_{\ell,m}(\zeta)=
      \langle \zeta,\realSH_{\ell,m}\rangle,
\end{equation} 
where the scalar product is understood as in \eqref{eq:def:Llm}, i.e.~$\nu_{\ell,m}$ is the pushforward of $\nu^{(\eta)}$ the Lévy measure of $L$ 
under the bounded linear functional $\pi^{(\eta)}_{\ell,m}$, see \cite[Definition 4.14]{peszat2007},
\cite[Definition 2.5 and Remark 2.6]{app10}.
Then we have the representation
\begin{equation}\label{eq:scal-var-dec}
  \meanlm = \driftlm +\int_{\{|z|\geq 1\}} z^2\,\nu_{\ell,m}(\diff z), \qquad \varlm
  =\sigma_{\ell,m}^2
    +\int_{\R\setminus\{0\}}z^2\,\nu_{\ell,m}(\diff z).
\end{equation}
For $p\ge2$, we further set the $p$--th absolute moment of the L\'evy measure of $L_{\ell,m}$
\begin{equation}\label{eq:scal-mom-dec}
  \left|\nu_{\ell,m}\right|^{(p)}
  =\int_{\R\setminus\{0\}}|z|^p\,\nu_{\ell,m}(\diff z).
\end{equation}
In particular, note that $\left|\nu_{\ell,m}\right|^{(2)} \le v_{\ell,m}.$

We consider two different assumptions on the L\'evy process for our analysis. The first one is the following: 
\begin{assumption}\label{ass:StrongLevyProcess}
Let $\eta\in(-1,\infty)$ and let $p\ge2$. The driving noise $L$ is an
$H^\eta(\S^2)$-valued L\'evy process with
\[
  L(1)\in L^p(\Omega,H^\eta(\S^2)),\text{ that is }
  \E\bigl[\|L(1)\|^p_{H^\eta(\S^2)}\bigr]<\infty .
\]
\end{assumption}
Note that for the strong error estimates in Section~\ref{Sec:SpectApp} and~\ref{Sec:4} we only require a square-integrable L\'evy process, i.e., $p=2$. 
The case $p>2$ is used in the weak error analysis, where the polynomial growth of the test function must be matched with the moments of the solution.

Under Assumption~\ref{ass:StrongLevyProcess} with $p=2$ the expansion~\eqref{eq:KL-expansion_Levy}
converges in $L^2(\Omega,H^\eta(\S^2))$
and for every $t\ge0$,
\begin{align*}
  &\E\bigl[\|L(t)\|_{H^\eta(\S^2)}^2\bigr]
  =t\sum_{\ell=0}^\infty\sum_{m=-\ell}^\ell(1+\lambda_\ell)^\eta\varlm
    +t^2\sum_{\ell=0}^\infty\sum_{m=-\ell}^\ell(1+\lambda_\ell)^\eta\meanlm^2\\
  &\quad =t\|\std\|_{H^\eta(\S^2)}^2+t^2\|\mean\|_{H^\eta(\S^2)}^2<\infty ,
\end{align*}
in particular $\std$ and $\mean$ belong to $H^\eta(\S^2)$.

For $p\geq 2$, Assumption~\ref{ass:StrongLevyProcess} can equivalently be expressed through the
L\'evy measure, i.e., for a L\'evy process $L$, $L(1)\in L^p(\Omega,H^\eta(\S^2))$ if and only if  \begin{equation}\label{eq:finite_pth_moment_Levy_measure}
  \int_{\chi_\eta}\|\zeta\|^p_{H^\eta(\S^2)}\,\nu^{(\eta)}(\diff\zeta)<\infty,  
\end{equation}
see \cite[Proposition 1.8.13, Theorem 1.8.3]{jurek1993operator}.

Our second assumption connects our results to the isotropic setting considered in~\cite{lang2015, lang2023},  with angular power spectrum $A_\ell = a_\ell^2$ in the earlier works, and provides essential structural assumptions to obtain improved regularity and weak convergence results compared to Assumption \ref{ass:StrongLevyProcess}.

\begin{assumption}\label{ass:SpecialLevyProcess}
\looseness=-1
	The L\'evy process~$L$ given by the expansion~\eqref{eq:KL-expansion_Levy} satisfies for some $\alpha >0$
	\begin{equation*}
		L_{\ell,m} = a_{\ell} \hat{L}_{\ell,m},
	\end{equation*}
	where $a_0 \geq0$, $a_\ell \leq C \ell^{-\alpha/2}$ for all $\ell\in \N$, and $(\hat{L}_{\ell,m})_{\ell\in\N_0, m\in\{-\ell,\ldots,\ell\}}$ is a sequence of identically distributed real-valued square-integrable Lévy processes with mean $\E[\hat{L}_{\ell,m}(t)] = t\hatmean$ and variance \[\E[(\hat{L}_{\ell,m}(t) - \E[\hat{L}_{\ell,m}(t)])^2] = t\hatvar.\]
\end{assumption}

We observe that Assumption~\ref{ass:SpecialLevyProcess} implies that $L\in L^2(\Omega,H^\eta(\S^2))$ for all $\eta<\alpha/2-1$, since 
\begin{equation}\label{special:varconv}
	\E[\|L(t)\|_{H^\eta(\S^2)}^2]
	\le C t \sum_{\ell=1}^\infty(1+\lambda_\ell )^\eta\ell^{-\alpha}(1+2\ell)<\infty,
\end{equation}
if $2\eta - \alpha + 1 <-1$. Note that in these frameworks, we are not explicitly using the decomposition of the L\'evy process $L$ with the eigenfunctions of the covariance operator $Q$, hence we do not have any information on the dependence structure of the real--valued processes $L_{\ell,m}$ in contrast to \cite[Section 4.8.2]{peszat2007}.

For $p=2$ and $\eta\geq 0$ Assumption \ref{ass:StrongLevyProcess} covers the square-integrable subclass of L\'evy processes treated in~\cite{peszat2007}. Real-valued square-integrable processes that can be used to define a L\'evy process satisfying Assumption \ref{ass:SpecialLevyProcess} include Wiener, variance gamma, normal inverse Gaussian and square-integrable compound Poisson processes as well as truncated or tempered versions of $\alpha$-stable processes.  
Processes beyond $L^2(\Omega)$ are subject to future work.

We conclude this section by stating some tools we will use throughout the manuscript. First, recall an $L^p(\Omega)$--bound for the scalar  L\'evy processes $L_{\ell,m}$, which is based on the Lévy--It\^o decomposition, see e.g.~\cite[Theorem 2.4.16]{applebaum2009levy}, Burkholder--Davis--Gundy inequality, and Kunita inequality, see e.g.~\cite[Lemma 2.1]{dareiotis2016}.

\begin{lemma}\label{lm:jump-p-bnd}
   Let $p\geq 2$, let $V$ be a square-integrable scalar L\'evy process with characteristic triplet $(\drift_V, \sigma^2_V, \nu_V)$,  
   set $\mean_V = \drift_V + \int_{\{|z|\geq 1\}} z^2\,\nu_V(\diff z) $, $v_V=\sigma_V^2+\int_{\mathbb R\setminus\{0\}} z^2\,\nu_V(\diff z)$, $\left|\nu_{V}\right|^{(p)}=\int_{\mathbb R\setminus\{0\}}|z|^p\,\nu_V(\diff z)$, and let $g\in 
   L^p([0,t])$ be deterministic. Then
\[
\E\Big[\Big|\int_0^t g(s)\diff V(s)\Big|^p\Big]\,\le\,C\Big[\Big(v_V\int_0^t g(s)^2\diff s\Big)^{p/2}+\,\left|\nu_{V}\right|^{(p)}\int_0^t|g(s)|^p\diff s + |m|^p \Big(\int_0^t g(s)\diff s\Big)^{p}\Big],
\]
where $C$ is a constant dependent on $p$.
\end{lemma}

Second, let $t\in[0,T]$, $\ell\in\N_0$ and $p\geq 1$, then 
\begin{equation}\label{eq:kern-est}
\int_0^t e^{-p\lambda_\ell(t-s)} \diff s = \begin{cases}
    \frac{1-e^{-p\lambda_\ell  t}}{p\lambda_\ell   }\leq\frac1{\lambda_\ell }\leq \frac32\frac{1}{1+\lambda_\ell,  }& \ell\ge 1\\
    t,&\ell=0
\end{cases}
    \,\, \leq C \frac{1}{1+\lambda_\ell},
\end{equation}
where $C$ is a constant dependent on $T$.

\subsection{Solution of the stochastic heat equation}

In the next step, we solve the stochastic heat equation \eqref{eq:HeatEq}. For this, we use the spherical harmonic functions~${\mathcal{\realSH}}$ as ansatz, set $X_{\ell,m}(t) = \langle X(t), \realSH_{\ell,m}\rangle_{L^2(\S^2)}$ for all $\ell\in\N_0$, $m\in\{-\ell,\ldots,\ell\}$, and obtain the expansion
\begin{equation}\label{eq:SolSeriesCalc}
    \sum_{\ell=0}^\infty \sum_{m=-\ell}^\ell X_{\ell,m}(t) \realSH_{\ell,m} = \sum_{\ell=0}^\infty  \sum_{m=-\ell}^\ell \left(X_{\ell,m}(0 ) \realSH_{\ell,m} +  \int_0^t X_{\ell,m}(s) \Delta_{\S^2} \realSH_{\ell,m} \diff s + L_{\ell,m}(t) \realSH_{\ell,m} \right).
\end{equation}
Since the spherical harmonics are eigenfunctions of the spherical Laplacian, see \eqref{eq:EigenvaluesRealSH}, we get 
\begin{equation}\label{eq:SolSeries}
    \sum_{\ell=0}^\infty \sum_{m=-\ell}^\ell X_{\ell,m}(t) \realSH_{\ell,m} = \sum_{\ell=0}^\infty  \sum_{m=-\ell}^\ell \Big(X_{\ell,m}(0 )  -\lambda_\ell  \int_0^t X_{\ell,m}(s) \diff s + L_{\ell,m}(t) \Big) \realSH_{\ell,m}.
\end{equation}
Hence, the solution of \eqref{eq:HeatEq} is given by the solutions $(X_{\ell,m})_{\ell\in\N_0,m\in\{-\ell,\ldots,\ell\}}$ to the stochastic differential equations
\begin{equation}\label{eq:SDEComp}
    X_{\ell,m}(t) = X_{\ell,m}(0 )  -\lambda_\ell \int_0^t X_{\ell,m}(s) \diff s + L_{\ell,m}(t).
\end{equation}
By the variation of constants formula, we obtain the solutions
\begin{align}\label{eq:SDESol}
    \begin{split}
    X_{\ell,m}(t) &= e^{-\lambda_\ell t} X_{\ell,m}(0 ) + \int_0^t e^{-\lambda_\ell (t-s)} \diff L_{\ell,m}(s)\,= e^{-\lambda_\ell }  t X_{\ell,m}(0 ) + \convLlm(t),
    \end{split}
\end{align}
which yields the solution to~\eqref{eq:HeatEq}
\begin{equation}
\begin{aligned}\label{eq:SolSeries2}
        X(t) &= \sum_{\ell=0}^\infty \sum_{m=-\ell}^\ell \left( e^{-\lambda_\ell  t}X_{\ell,m}(0)+ \convLlm(t)\right) \realSH_{\ell,m}\,=\sum_{\ell=0}^\infty \sum_{m=-\ell}^\ell  e^{-\lambda_\ell  t}X_{\ell,m}(0)\realSH_{\ell,m} + \convL(t),
\end{aligned}
\end{equation}
where we have defined $\convLlm(t) =\int_0^t e^{-\lambda_\ell (t-s)} \diff L_{\ell,m}(s)$ and the stochastic convolution 
\begin{equation}\label{def:convolutionZ}
    \convL= \sum_{\ell=0}^\infty \sum_{m=-\ell}^\ell \convLlm \realSH_{\ell,m}.
\end{equation}
If not specified otherwise, we assume that $X(0)=X^0 \in L^2(\Omega, L^2(\S^2))$. The proof of existence, uniqueness, and $L^2$-regularity of the solution is classical, and the reader is directed to \cite{Hausenblas2005, peszat2007, ALBEVERIO2009835, Benth2023}. In the specific situation that $\eta \in (-1,0)$, it follows from Theorems~9.29 and~9.15 in~\cite{peszat2007}, where (F) and (G) are satisfied with $a(t) = b(t) = t^{\eta/2}$, when $L$ is split in its mean and martingale part. For $\eta \ge 0$, even the condition (GI) is satisfied, which guarantees the existence of a c\`adl\`ag modification.

In the following lemma, we compute the moments of the solution.

\begin{lemma}\label{L:MESol}
    Under Assumption \ref{ass:StrongLevyProcess}, the stochastic heat equation \eqref{eq:HeatEq} satisfies
    \begin{align}
        \E[X(t)] &=\sum_{\ell=0}^\infty \sum_{m=-\ell}^\ell \Big( \E[X^0_{\ell,m}
        ]\, e^{-\lambda_\ell t}  +  \meanlm \int_0^t  e^{-\lambda_\ell  (t-s)} \diff s \Big)\realSH_{\ell,m}
    \end{align}
    and
    \begin{equation}
        \begin{aligned}
            &\E\big[\|X(t)\|^2_{L^2(\S^2)}\big] = \sum_{\ell=0}^\infty \sum_{m=-\ell}^\ell  \E\Big[\Big( X_{\ell,m}^0  e^{-\lambda_\ell  t} + \meanlm \int_0^t e^{-\lambda_\ell (t-s)} \diff s\Big)^2\Big] + \varlm\int_0^t e^{-2\lambda_\ell  (t-s)} \diff s.
        \end{aligned}
    \end{equation}
\end{lemma}
\begin{proof}
	For the expectation we obtain 
	\begin{equation}
		\E[X(t)] =  \E[X^0] +\int_0^t \Delta_{\S^2} \E[X(s)]\diff s + \E[L(t)].
	\end{equation}
	Since $\E[L(t)] = t \mean$ by~\eqref{eq:mean_Levy} and using $u(t) =  \E[X(t)]$ and $u_0 = \E[X^0]$, we need to solve the partial differential equation 
	\begin{equation}
		\partial_t u = \Delta_{\S^2} u + \mean, \qquad u(0) = u_0.
	\end{equation}
    This is done similarly to \cite[Appendix A]{lang2023}, where the homogeneous equation is considered, and yields
    \begin{equation}
		\begin{aligned}
			\E[X(t)] 
			& =  \sum_{\ell=0}^\infty \sum_{m=-\ell}^\ell \Big( e^{-\lambda_\ell  t} \E[\langle X(0), \realSH_{\ell,m}\rangle_{L^2(\S^2)}] + \int_0^t  e^{-\lambda_\ell (t-s)} \meanlm \diff s \Big)\realSH_{\ell,m}.
		\end{aligned}
	\end{equation}
	For the second moment of the solution, we calculate using \eqref{eq:SolSeries2},
	\begin{equation}\label{eq:SecMom1}
		\begin{aligned}
			\E[\|X(t)\|^2_{L^2(\S^2)}]
			& = \sum_{\ell=0}^\infty \sum_{m=-\ell}^\ell  
			\Big(e^{-2{\lambda_\ell }  t} \E[|X_{\ell,m}^0|^2]
			+ \E\big[\big| \convLlm(t) \big|^2\big]\Big)\\
			& \quad + 2 \sum_{\ell=0}^\infty \sum_{m=-\ell}^\ell e^{-\lambda_\ell t}\, \E[ X_{\ell,m}^0]\, \E\big[\convLlm(t) \big],
		\end{aligned}
	\end{equation}
	where $X^0$ is independent of~$L$.
	
	By \cite[Example 15.12]{brockwell2024} we obtain for the second term that
	\begin{equation}\label{eq:SecMom4}
		\begin{aligned}
			\E\big[\big| \convLlm(t) \Big|^2\Big]
			&= \varlm\int_0^t e^{-2\lambda_\ell (t-s)} \diff s + \Big( \meanlm \int_0^t e^{-\lambda_\ell  (t-s)} \diff s\Big)^2;\\
            \E\big[\convLlm(t) \big]	
			&= \meanlm  \int_0^t e^{-\lambda_\ell (t-s)} \diff s.
		\end{aligned}
	\end{equation}
	So in conclusion, we get
	\begin{equation}
		\begin{aligned}
			&\E\big[\|X(t)\|^2_{L^2(\S^2)}\big] = \sum_{\ell=0}^\infty \sum_{m=-\ell}^\ell  \E\Big[\Big( X_{\ell,m}^0  e^{-\lambda_\ell  t} + \meanlm \int_0^t e^{-\lambda_\ell  (t-s)} \diff s\Big)^2\Big] + \varlm\int_0^t e^{-2\lambda_\ell (t-s)} \diff s,
		\end{aligned}
	\end{equation}
	which finishes the proof.
\end{proof}

In what follows, we establish higher-order regularity of the solution to~\eqref{eq:HeatEq}, a key ingredient in deriving weak convergence rates for a broad class of test functions in the next section.

\begin{proposition}\label{prop:regul}
Let $\varrho\in\R$ and let $X$ be the mild solution given in~\eqref{eq:SolSeries2} with initial value $X^0\in L^{p}(\Omega,H^\gamma(\S^2)),\, \gamma \in \R$. If one of the following assumptions holds:
\begin{enumerate}
    \item  $p\geq2$, $L = W$ is a Wiener process in $L^p(\Omega, H^\eta(\S^2))$, and $\varrho-\eta\leq1$;\label{assRegW}
    \item $p\geq 2$, $L$ satisfies Assumption~\ref{ass:SpecialLevyProcess}, $(\hat L_{\ell,m})_{\ell,m}$ i.i.d., $\mathbb E[|\hat L_{0,0}(1)|^{p}]<\infty$, and $\varrho<\alpha/2$;\label{assRegSpecLP}
    \item $p \ge 2$, $L$ satisfies Assumption~\ref{ass:StrongLevyProcess} with $p$, and $\varrho-\eta\le2/p$;\label{assRegGenLP} 
\end{enumerate}
then 
\begin{align}\label{thm:reg}
\|X(t)\|_{L^{p}(\Omega,H^\varrho(\S^2))} \leq C (1 
	+  t^{-\max\{\varrho-\gamma,0\}/2}\|X^0\|_{L^{p}(\Omega,H^\gamma(\mathbb{S}^2))}) <\infty,
\end{align}
where $C$ is a constant dependent on $T$, $\gamma$, and $\varrho$. 
\end{proposition}

\begin{proof}
We consider the mild solution~$X$ in~\eqref{eq:SolSeries2}. Using the triangle inequality, we split the norm
\begin{align*}
    &\|X(t)\|_{L^{p}(\Omega,H^\varrho(\S^2))}
    	\leq \Bigl\|\sum_{\ell=0}^\infty \sum_{m=-\ell}^\ell e^{-\lambda_\ell  t}X_{\ell,m}(0)\realSH_{\ell,m}\Bigr\|_{L^{p}(\Omega,H^\varrho(\S^2))}
    		+ \|\convL(t)\|_{L^{p}(\Omega,H^\varrho(\S^2))}.
\end{align*}
The first term satisfies
\begin{align*}
	& \Bigl\|\sum_{\ell=0}^\infty \sum_{m=-\ell}^\ell e^{-\lambda_\ell  t}X_{\ell,m}(0)\realSH_{\ell,m}\Bigr\|_{L^{p}(\Omega,H^\varrho(\S^2))}^p
    = \E \Bigl[ \Bigl(\sum_{\ell=0}^\infty \sum_{m=-\ell}^\ell (1+\lambda_\ell  )^\varrho e^{-2\lambda_\ell  t}(X_{\ell,m}(0))^2 \Bigr)^{p/2} \Bigr],
\end{align*}
and for $\varrho \le \gamma$, it is bounded by $\|X^0\|_{L^{p}(\Omega,H^\gamma(\S^2))}^{p}$. For $\varrho > \gamma$, we observe with $x^{\varrho-\gamma} e^{-x} \le C$, with the finite constant depending on $\varrho-\gamma$, hence for $\ell \neq 0$,
\begin{equation}\label{rhonuestimate}
	(1+\lambda_\ell )^\varrho e^{-2\lambda_\ell  t}
		\le  C \, (1+\lambda_\ell  )^\gamma t^{-(\varrho-\gamma)},
\end{equation}
while for $\ell=0$ we bound $1\leq T^{(\varrho-\gamma)}t^{-(\varrho-\gamma)}$. This yields
\begin{equation}\label{semiest-ic}
	\Bigl\|\sum_{\ell=0}^\infty \sum_{m=-\ell}^\ell e^{-\lambda_\ell  t}X_{\ell,m}(0)\realSH_{\ell,m}\Bigr\|_{L^{p}(\Omega,H^\varrho(\S^2))}
		\le C\, t^{- \max\{\varrho-\gamma,0\}/2}\|X^0\|_{L^p(\Omega,H^\gamma(\S^2))},
\end{equation}
which yields the second summand of the claim for all cases. It remains to bound~$\convL$.

First, we prove the claim assuming \ref{assRegSpecLP}. Let $(\hatdrift, \hat\sigma^2,\hat \nu)$ denote the characteristic triplet of the Lévy processes $\hat L_{\ell,m}$, which is independent of $\ell$ and $m$ as $\hat L_{\ell,m}$ are i.i.d. Define $\hatmean$, $\hat v$ and $|\hat \nu|^{(p)}$ based on $(\hatdrift, \hat\sigma^2,\hat \nu)$ similarly as in \eqref{eq:scal-var-dec}, \eqref{eq:scal-mom-dec}. By Assumption~\ref{ass:SpecialLevyProcess}, $\nu_{\ell,m}$ is the image measure of $\hat\nu$ under the map $x\mapsto a_{\ell}x$, and
\begin{equation}\label{eq:relLhatL}
    v_{\ell,m}=a^2_{\ell}\hat v,\qquad
\left|\nu_{\ell,m}\right|^{(p)}=a_{\ell}^{p}\left|\hat\nu\right|^{(p)},\qquad
\tilde m_{\ell,m}=a_{\ell}\hat m.
\end{equation}

Now we fix $t\in[0,T]$, use the series expansion \eqref{def:convolutionZ} of $\convL$, and apply \cite[Corollary 3]{Latala_momentsSum} to the independent and non-negative  $(\|\convLlm(t)\realSH_{\ell,m} \|_{H^\varrho(\S^2)}^2)_{\ell\in\N_0,m\in\{-\ell,\ldots,\ell\}}$  to obtain
\begin{align}
\mathbb E\Big[\bigl\|\convL(t)\bigr\|^{p}_{H^{\varrho}(\S^2)} \Big]
\le C\Bigl(\tfrac{p/2}{\ln(p/2)}\Bigr)^{p/2}
\max\Biggl\{&\Bigl(\sum_{\ell =0}^\infty \sum_{m =-\ell}^\ell\mathbb E\bigl[\|\convLlm(t)Y_{\ell,m} \|_{H^\varrho(\S^2)}^2\bigr]\Bigr)^{p/2},\\
&
\ \sum_{\ell =0}^\infty \sum_{m =-\ell}^\ell \mathbb E\bigl[\|\convLlm(t)Y_{\ell,m} \|_{H^\varrho(\S^2)}^{p}\bigr]\Biggr\}.
\end{align}

For the first term we obtain by \eqref{eq:SecMom4} and \eqref{eq:kern-est},
\begin{equation}\label{firstsummand}
\Big( \sum_{\ell =0}^\infty \sum_{m =-\ell}^\ell \mathbb E\bigl[\|\convLlm(t)Y_{\ell,m} \|_{H^\varrho(\S^2)}^2 \bigr]\Big)^{p/2}
\le C\Big(\sum_{\ell =0}^\infty \sum_{m =-\ell}^\ell (1+\lambda_{\ell})^{\varrho-1} \varlm +(1+\lambda_{\ell})^{\varrho-2} \meanlm^{2}\Big)^{p/2},
\end{equation}
Using \eqref{eq:relLhatL} and the fact that $\ell^2\leq 1+\lambda_\ell \leq 3 \ell^{2}$ yields
\begin{align}\label{eq:est1ii}
\begin{split}
&\Big(\sum_{\ell =0}^\infty \sum_{m =-\ell}^\ell (1+\lambda_\ell)^{\varrho-1}v_{\ell,m} +(1+\lambda_\ell )^{\varrho-2}\mean_{\ell,m}^{2}\Big)^{p/2}\\
&\quad =
\Big(\sum_{\ell =0}^\infty \sum_{m =-\ell}^\ell(1+\lambda_\ell)^{\varrho-1}a^2_{\ell}\hat v
+(1+\lambda_\ell)^{\varrho-2}a_{\ell}^{2}\hat m^2\Big)^{p/2}\\
&\quad =\Big(\sum_{\ell =1}^\infty (1+2\ell)\big((1+\lambda_\ell )^{\varrho-1}a^2_{\ell}\hat v
+(1+\lambda_\ell )^{\varrho-2}a_{\ell}^{2}\hat m^2\big)+a_0^2(\hat v+\hat m^2)\Big)^{p/2}\\
&\quad \leq\Big(C\sum_{\ell =1}^\infty \big(\ell^{2\varrho-1}a^2_{\ell}\hat v
+\ell^{2\varrho-3}a_{\ell}^{2}\hat m^2\big)+a_0^2(\hat v+\hat m^2)\Big)^{p/2} 
\leq C\Big(\sum_{\ell =1}^\infty \ell^{2\varrho-1-\alpha}\Big)^{p/2}+C.
\end{split}
\end{align}
For the second term, we obtain using Lemma~\ref{lm:jump-p-bnd}, \eqref{eq:kern-est}, and the superadditivity of $x\mapsto x^{p/2}$ 
\begin{align}\label{secondsummand}
\begin{split}
&\sum_{\ell =0}^\infty \sum_{m =-\ell}^\ell \mathbb E\bigl[\|\convLlm(t)Y_{\ell,m} \|_{H^\varrho(\S^2)}^{p}\bigr]\\
& =\sum_{\ell =0}^\infty \sum_{m =-\ell}^\ell (1+\lambda_{\ell})^{\varrho p/2}\,\mathbb E\Big[\Bigl|\int_{0}^{t}e^{-\lambda_{\ell}(t-s)}\diff L_{\ell,m}(s)\Bigr|^{p}\Big] \\
& \le C\sum_{\ell =0}^\infty \sum_{m =-\ell}^\ell \Bigl(
(1+\lambda_{\ell})^{(\varrho-1)p/2}v_{\ell,m}^{p/2}
+(1+\lambda_{\ell})^{\varrho p/2-1}\left|\nu_{\ell,m}\right|^{(p)}
+(1+\lambda_{\ell})^{(\varrho-2)p/2}|\tilde m_{\ell,m}|^{p}\Bigr)\\
&\le C\Big( \sum_{\ell =0}^\infty \sum_{m =-\ell}^\ell 
(1+\lambda_{\ell})^{\varrho-1}v_{\ell,m}
+(1+\lambda_{\ell})^{\varrho-2}|\tilde m_{\ell,m}|^{2}\Big)^{p/2}
+  C \sum_{\ell =0}^\infty \sum_{m =-\ell}^\ell 
(1+\lambda_{\ell})^{\varrho p/2-1}\left|\nu_{\ell,m}\right|^{(p)}.
\end{split}
\end{align}
Similarly to \eqref{eq:est1ii}, we obtain
\begin{align}\label{eq:est2ii}
    \sum_{\ell =0}^\infty \sum_{m =-\ell}^\ell(1+\lambda_\ell )^{\varrho p/2-1}\left|\nu_{\ell,m}\right|^{(p)}\leq C \sum_{\ell=1}^\infty\ell^{\varrho p-\alpha p/2-1} +C.
\end{align}
Hence, combining \eqref{firstsummand}, \eqref{eq:est1ii}, \eqref{secondsummand}, and \eqref{eq:est2ii} we obtain
\begin{align}
\mathbb E\Big[\bigl\|\convL(t)\bigr\|^{p}_{H^{\varrho}(\S^2)} \Big]
\le C\Big(\sum_{\ell =1}^\infty \ell^{2\varrho-1-\alpha}\Big)^{p/2} +  C \sum_{\ell=1}^\infty\ell^{\varrho p -1 -\alpha p/2} +C.
\end{align}
If $2\varrho-1-\alpha<-1$ resp.~$\varrho p -1 -\alpha p/2 <-1$, i.e., $\varrho<\alpha/2$, the series in this estimate are finite and $C$ is a constant uniform on $[0,T]$ only dependent on $T$ and $p$.

Next we prove the claim assuming \ref{assRegGenLP} with $p\geq 2$ on a more abstract level. For this let $(e^{\Delta_{\S^2}t})_{t \geq 0}$ denote the semigroup generated by~$\Delta_{\S^2}$, which is given by $e^{\Delta_{\S^2}t} \realSH_{\ell,m} = e^{-\lambda_\ell  t} \realSH_{\ell,m}$ and satisfies for $f \in H^\eta(\S^2)$ that $e^{\Delta_{\S^2}t} f = \sum_{\ell=0}^\infty \sum_{m=-\ell}^\ell e^{-\lambda_\ell t}f_{\ell,m} \realSH_{\ell,m}$. Since $L(t)\in H^\eta(\S^2)$ satisfies \eqref{ass3}, by the linearity of the semigroup,
\begin{align}
    \convL(t) & = \int_0^t e^{\Delta_{\S^2}(t-s)}\mean \diff s + \int_0^t e^{\Delta_{\S^2}(t-s)}\diff W^{(\eta)}(s) + \int_0^t\int_{\chi_\eta} e^{\Delta_{\S^2}(t-s)}{\zeta}\tilde N^{(\eta)}(\diff s,\diff \zeta)\\
    & = \int_0^t e^{\Delta_{\S^2}(t-s)}\mean \diff s +\convW(t) +\convN(t).
\end{align}
Hence,
\begin{align*}
   &\E \big[ \| \convL(t)\|^p_{H^\varrho(\S^2)}\big] \\
   &\quad \leq C\Big( \E \Big[ \Big \|\int_0^t e^{\Delta_{\S^2}(t-s)}\mean \diff s \Big\|^p_{H^\varrho(\S^2)}\Big] + \E \big[ \|\convW(t) \|^p_{H^\varrho(\S^2)}\big] + \E \big[ \|\convN(t)\|^p_{H^\varrho(\S^2)}\big].
\end{align*}
The second term is is bounded for all $p\geq 2$ and $\varrho-\eta\leq1$ by \cite[Theorem~4.36]{DaPrato_Zabczyk_1992}, which implicitly proves the claim assuming~\ref{assRegW}.

For the first and third term, we set $\varrho=\eta+2/p$, which is the maximal regularity and implies the claim for all $\varrho <\eta+2/p$ since the Sobolev spaces are nested. 
Using \cite[Lemma~3.1]{MARINELLI2010616}, we obtain
\begin{align}
	 \E\left[\|\convN(t)\|^p_{H^{\eta+2/p}(\S^2)}\right]
		\leq  C\int_0^t \Big( &
			\int_{\chi_\eta} \|e^{\Delta_{\S^2}(t-s)}\zeta\|_{H^{\eta+2/p}(\S^2)}^p\nu^{(\eta)}(\diff \zeta)\\
			&+ \Bigl( \int_{\chi_\eta} \|e^{\Delta_{\S^2}(t-s)}\zeta\|_{H^{\eta+2/p}(\S^2)}^2\nu^{(\eta)}(\diff \zeta)\Bigr)^{p/2}
			 \Big)\diff s.
\end{align}
For all $\zeta\in H^{\eta}(\mathbb S^{2})$ and $t\in[0,T]$, by the definition of the semigroup and H\"older's inequality,
\begin{align*}
  \bigl\|e^{\Delta_{\mathbb S^{2}}(t-s)}\zeta\bigr\|^{2}_{H^{\eta+2/p}(\S^2)}
  & =
    \sum_{\ell=0}^\infty\sum_{m=-\ell}^\ell\bigl((1+\lambda_\ell )^{\eta+1}e^{-p\lambda_\ell (t-s)}\zeta_{\ell,m}^{2}\bigr)^{2/p}
    \bigl((1+\lambda_\ell )^{\eta}\zeta_{\ell,m}^{2}\bigr)^{1-2/p}\\
  & \le\Bigl(\sum_{\ell=0}^\infty\sum_{m=-\ell}^\ell(1+\lambda_\ell)^{\eta+1} e^{-p\lambda_\ell   (t-s)}\zeta_{\ell,m}^{2}\Bigr)^{2/p}
    \|\zeta\|^{2(1-2/p)}_{H^{\eta}(\S^2)} .
\end{align*}
Hence, by \eqref{eq:kern-est},
\begin{align}\label{eqneeded}
  &\int_{0}^{t}\bigl\|e^{\Delta_{\mathbb S^{2}}(t-s)}\zeta\bigr\|^{p}_{H^{\eta+2/p}(\S^2)}\diff s\\
  &\quad\le\|\zeta\|^{p-2}_{H^{\eta}(\S^2)}\sum_{\ell=0}^\infty\sum_{m=-\ell}^\ell (1+\lambda_\ell )^{\eta+1}
    \Bigl(\int_{0}^{t}e^{-p\lambda_\ell(t-s)}\diff s\Bigr)\zeta_{\ell,m}^{2}\le C \|\zeta\|^{p}_{H^{\eta}(\S^2)}.
\end{align}
Using these calculations, Fubini's theorem for the first summand resp.\ Minkovski's integral inequality for the second summand, and \eqref{eq:finite_pth_moment_Levy_measure} with Assumption~\ref{ass:StrongLevyProcess}, we obtain that $\E\left[\|\convN(t)\|^p_{H^{\eta+2/p}(\S^2)}\right]\leq C$, where $C\in(0,\infty)$ depends on $T$, $p$, $\eta$, and $\nu^{(\eta)}$. 
With similar computation we prove that
\[
\Bigl\|\int_{0}^{t}e^{\Delta_{\S^2}(t-s)}\tilde m\,\diff s\Bigr\|_{H^{\eta+2/p}(\S^2)}^p\le C,
\]
which concludes the proof.
\end{proof}

Note that in the previous proposition assuming \ref{assRegSpecLP} implies assumption~\ref{assRegGenLP}. To see that, an estimate similar to \eqref{special:varconv} has to be found in $L^p(\Omega, H^{\eta}(\S^2))$ for $\eta <\alpha/2-1$. This estimate is obtained with the same arguments as in the proof of Proposition~\ref{prop:regul}\ref{assRegSpecLP} with $\eta$ instead of $\varrho$ and $1$ instead of $e^{-\lambda_\ell(t-s)}$. However from the additional structural assumption in \ref{assRegSpecLP} we gain $1-2/p$ orders of regularity compared to~\ref{assRegGenLP}. 

The following example illustrates that in the generality presented in Proposition \ref{prop:regul}\ref{assRegGenLP} the restriction $\varrho-\eta\le2/p$ is optimal.

\begin{example}\label{ex:optimality}
Define a sequence of independent Poisson processes $N_{\ell,m}$ with intensity $r_{\ell}=\bigl((2\ell+1)\,\ell\,\log^{2}(2+\ell)\bigr)^{-1}$ for $\ell \geq 1$ and $m\in\{-\ell,\ldots,\ell\}$. Further define $c_{\ell}=(1+\lambda_{\ell})^{-\eta/2}$ with $\eta\in(-1,\infty]$ for $\ell \geq 1$, 
\[L_{\ell,m}(t)=c_\ell( N_{\ell,m}(t)- r_\ell t)\] for  $\ell \geq 1$ and $m\in\{-\ell,\ldots,\ell\}$ and $L_{0,0}\equiv 0$.

Then for every $p\in[2,\infty)$ the process $L= \sum_{\ell =0}^\infty \sum_{m = -\ell}^\ell L_{\ell,m} \realSH_{\ell,m}$ is a Lévy process in $L^p(\Omega,H^\eta(\S^2))$ and for every $\varrho > \eta+p/2$ and $t>0$ the convolution with the heat semigroup satisfies
\[
\mathbb E\bigl[\| \convL(t)\|^{p}_{H^{\varrho}(\mathbb S^{2})}\bigr]=\infty.\]
To see this, we first note that for $\ell\geq 1$ and $m\in\{-\ell, \ldots, \ell\}$ the Lévy measure of $L_{\ell,m}$ is $\nu_{\ell,m} = r_\ell \delta_{c_\ell}$ and hence $\meanlm = 0$, $v_{\ell,m}=c_{\ell}^{2}r_{\ell}$, and $\left|\nu_{\ell,m}\right|^{(p)}=c_{\ell}^{p}r_{\ell}$. This yields 
\begin{align*}
    & \E[|L_{\ell,m}(1)|^p ] 
    \leq C( v_{\ell,m}^{p/2} + |\nu_{\ell,m}|^{p})
    \leq C( c_\ell^p r_\ell^{p/2} + c_\ell^p r_\ell) \leq C  c_\ell^p r_\ell.
\end{align*}
Using the previous observations and \cite[Corollary 3]{Latala_momentsSum} we bound for $p\geq 2$,
\begin{align*}
    &\mathbb E\bigl[\| L(1)\|^{p}_{H^{\eta}(\mathbb S^{2})}\bigr] \\
    & \quad =\E \Big[ \Big(\sum_{\ell=1}^\infty\sum_{m=-\ell}^\ell (1+\lambda_{\ell})^{\eta} |L_{\ell,m}(1)|^2 \Big)^{p/2} \Big] \\
    &\quad \leq C_p \max\Big\{ \Big(\sum_{\ell=1}^\infty\sum_{m=-\ell}^\ell (1+\lambda_{\ell})^{\eta} \E[|L_{\ell,m}(1)|^2 ]\Big)^{p/2},  \sum_{\ell=1}^\infty\sum_{m=-\ell}^\ell (1+\lambda_{\ell})^{\eta p/2} \E[|L_{\ell,m}(1)|^p ] \Big\} \\
    &\quad \leq C_p \max\Big\{ \Big(\sum_{\ell=1}^\infty \frac{1}{\ell \log^2(2+\ell)}\Big)^{p/2},  \sum_{\ell=1}^\infty \frac{1}{\ell \log^2(2+\ell)} \Big\} <\infty. 
\end{align*}
Hence, $L\in L^p(\Omega,H^\eta(\S^2))$.

Next, we fix $t\in[0,T]$ and define $\ell_0 =\min\{\ell\in\N: \lambda_\ell \geq 1/(2t)\}$. For $\ell \geq \ell_0$ and $m\in\{-\ell,\ldots,\ell\}$ we define the set  $E_{\ell,m} \subset \Omega$ by
\[E_{\ell,m} = \Big\{ N_{\ell,m}(t) -  N_{\ell,m}(t- \tfrac{1}{2\lambda_\ell }) =1, N_{\ell,m}(t- \tfrac{1}{2\lambda_\ell}) =0\Big\}.\]
Let $\theta$ denote a uniformly distributed random variable on $[0,1/(2\lambda_{\ell}))$. Then on $E_{\ell,m}$ the time of the first jump is distributed like $t-\theta$, see e.g.~\cite[Section 5.3.5]{ross2024introduction}. Together with the fact that $r_\ell\lambda_\ell^{-1}\leq 1/6\leq e^{-1/2}/2$, as $r_\ell\leq r_1\leq 1/3$ and $\lambda_\ell\geq2$, we obtain
\begin{align}
    \left|\convLlm(t)\right|
     & =\Big|c_{\ell}\int_{t-\frac{1}{2\lambda_{\ell}}}^t e^{-\lambda_{\ell}(t-s)}\diff N_{\ell,m}(s)-c_{\ell}r_{\ell}\int_{0}^t e^{-\lambda_{\ell}(t-s)}\diff s\Big| \\
    &=\left|c_{\ell}e^{-\lambda_{\ell}\theta}-c_{\ell}r_{\ell}\lambda_{\ell}^{-1}(1-e^{-\lambda_{\ell} t})\right|
    \geq c_{\ell}e^{-1/2}-c_{\ell}r_{\ell}\lambda_{\ell}^{-1}\geq \frac{c_{\ell}}{2}e^{-1/2}.
\end{align}
Further, using that $N_{\ell,m}$ is a Poisson process we obtain
\[
  \mathbb P(E_{\ell,m})
  =e^{-r_{\ell}(t-1/(2\lambda_{\ell}))}\,
   \frac{r_{\ell}}{2\lambda_{\ell}}e^{-{r_{\ell}}/{(2\lambda_{\ell})}}
  =\frac{r_{\ell}}{2\lambda_{\ell}}e^{-r_{\ell}t}
  \ \ge\ \frac{e^{-t/3}}{2}\,\frac{r_{\ell}}{\lambda_{\ell}}.
\]
The superadditivity of $x\mapsto x^{p/2}$ and the previous estimates yield
\begin{align}
    &\mathbb E\left[\bigl\| \convL(t)\bigr\|^{p}_{H^{\varrho}(\S^2)}\right]\\
    & \quad \ge\sum_{\ell= \ell_0}^\infty \sum_{m=-\ell}^\ell\E\Big[ \left((1+\lambda_{\ell})^{\varrho}\bigl(\convLlm(t)\bigr)^{2}\right)^{p/2}\Big]
    \ge\sum_{\ell= \ell_0}^\infty \sum_{m=-\ell}^\ell (1+\lambda_{\ell})^{\varrho p/2}\E\Big[ \bigl|\convLlm(t)\bigr|^p \mathds{1}_{E_{\ell,m}}\Big]\\
    &\quad \ge\sum_{\ell= \ell_0}^\infty \sum_{m=-\ell}^\ell (1+\lambda_{\ell})^{\varrho p/2} \Big(\frac{c_{\ell}}{2}e^{-1/2}\Big)^p \P(E_{\ell,m})
    \ge C \sum_{\ell=\ell_0}^\infty (2\ell+1)(1+\lambda_{\ell})^{\varrho p/2}\frac{c_{\ell}^{p}r_{\ell}}{\lambda_{\ell}}\\
    &\quad \geq C\sum_{\ell=\ell_0}^\infty \frac{\ell^{(\varrho-\eta)p-2}}{\ell\,\log^{2}(2+\ell)},
\end{align}
where $C$ is a constant dependent on $p$ and $T$. For $(\varrho-\eta)p>2$ the series on the right-hand side diverges, and hence for $\varrho > \eta+p/2$ and $t>0$, $
\mathbb E\bigl[\| \convL(t)\|^{p}_{H^{\varrho}(\mathbb S^{2})}\bigr]=\infty.$ 
\end{example}

Note that the Lévy process studied in this example does not satisfy Assumption \ref{ass:SpecialLevyProcess} as the construction $L_{\ell,m} = a_\ell \hat L_{\ell,m}$ with $\hat L_{\ell,m}$ identically distributed does not allow for an $\ell$-dependent intensity of the underlying Poisson processes. 

The lower bounds constructed in the example are exploiting exactly this $\ell$-dependence of the intensity and are based on the event where one large jump of the components of the Lévy process happens shortly before time $t$ and is hence barely smoothed by the semigroup.

\begin{remark}\label{rmk:counter}
    The regularity results proven in Proposition \ref{prop:regul} are optimal. Assuming \ref{assRegGenLP}, this is shown in the previous example. Assuming \ref{assRegSpecLP} we consider the case where $a_\ell = C \ell^{-\alpha/2}$ for all $\ell\geq 1$. Then we obtain by Jensen's inequality, \eqref{eq:SecMom4}, \eqref{eq:kern-est}, $(1+\lambda_\ell) \geq \ell^2$, and $\lambda_\ell \leq 2\ell^2$ that
    \begin{align}
        &\E\big[\| \convL(t)\|^{p}_{H^{\alpha/2}(\S^2)}\big]^{2/p}\\
        & \quad \ge\E\big[\| \convL(t)\|^{2}_{H^{\alpha/2}(\S^2)}\big]
        \geq \sum_{\ell =1}^\infty \sum_{m=-\ell}^\ell (1+\lambda_\ell)^{\alpha/2} \varlm \frac{1-e^{-2\lambda_\ell t}}{2\lambda_\ell}\\
        & \quad\geq \frac{1-e^{-2 t}}{2} \sum_{\ell =1}^\infty \sum_{m=-\ell}^\ell (1+\lambda_\ell)^{\alpha/2} \varlm \lambda_\ell^{-1} 
         = C \hatvar \frac{1-e^{-2 t}}{2} \sum_{\ell =1}^\infty \sum_{m=-\ell}^\ell (1+\lambda_\ell)^{\alpha/2} \ell^{-\alpha} \lambda_\ell^{-1} \\
        &\quad \geq C \hatvar \frac{1-e^{-2 t}}{2} \sum_{\ell =1}^\infty (2\ell+1) \ell^{\alpha} \ell^{-\alpha} \ell^{-2} \geq  C \hatvar \frac{1-e^{-2 t}}{2} \sum_{\ell =1}^\infty \ell^{-1} = \infty.
    \end{align}
    This also proves optimality assuming \ref{assRegW} as for arbitrary $\alpha > 2(\eta+1)$ and a Wiener process that satisfies assumption \ref{assRegSpecLP} with $a_\ell = C \ell^{-\alpha/2}$ we can do the same calculations as above, which proves sharpness using~\eqref{special:varconv}.
\end{remark}

\section{Spectral approximation in space}\label{Sec:SpectApp}

To approximate the solution to~\eqref{eq:HeatEq}, we start with a semi-discrete scheme obtained by a spatial truncation similarly to \cite{lang2015, cohen2022, lang2023}. More specifically, using the series expansion~\eqref{eq:SolSeries2}, we set for $\kappa  \in \N$
\begin{equation}\label{eq:ApproxSpace}
    X^{(\kappa)}(t) = \sum_{\ell=0}^\kappa \sum_{m=-\ell}^\ell \Big( e^{-\lambda_\ell  t} X_{\ell,m}^0 + \convLlm (t)  \Big)\realSH_{\ell,m}.
\end{equation}
With the same computations as in Lemma~\ref{L:MESol}, we observe that the expectation satisfies
\begin{equation}\label{eq:expect_spectral}
	\E[X^{(\kappa)}(t)] = \sum_{\ell=0}^\kappa \sum_{m=-\ell}^\ell \Big( \E[X_{\ell,m}^0] e^{-\lambda_\ell  t}  + \meanlm \int_0^t  e^{-\lambda_\ell   (t-s)}  \diff s \Big)\realSH_{\ell,m}
\end{equation}
and the second moment satisfies
\begin{align}\label{eq:second-moment_spectral}
	\begin{split}
		& \E\big[\|X^{(\kappa)}(t)\|^2_{L^2(\S^2)}\big]\\
		&\quad =  \sum_{\ell=0}^\kappa \sum_{m=-\ell}^\ell \bigg(   \E\Big[\Big( X_{\ell,m}^0  e^{-\lambda_\ell   t} + \meanlm \int_0^t e^{-\lambda_\ell  (t-s)} \diff s\Big)^2\Big]
		+ \varlm\int_0^t e^{-2\lambda_\ell  (t-s)} \diff s \bigg).
	\end{split}
\end{align}
Since $X^{(\kappa)}$ is given by a finite expansion of the smooth spherical harmonic functions, it is clear that it is smooth for each fixed~$\kappa$.

Similarly to \cite[Lemma~7.1]{lang2015} and \cite[Lemma~3.1]{lang2023}, we show strong convergence, where the rates coincide with those in the Wiener setting in $L^p(\Omega)$ under Assumption~\ref{ass:SpecialLevyProcess}, see \cite[Lemma~7.2]{lang2015}. Error rates in~$L^p(\Omega)$ beyond the Gaussian setting would require higher moment assumptions similar to Proposition~\ref{prop:regul}\ref{assRegGenLP} and are subject to future research.

\begin{theorem}\label{L:SA1}
Let $X$ be given by \eqref{eq:SolSeries2} and $X^{(\kappa)}$ by its spectral approximation~\eqref{eq:ApproxSpace}. Then,
\begin{equation*}
 		 \| X(t) - X^{(\kappa)}(t)\|_{L^2(\Omega, L^2(\S^2))}
 		\leq e^{-(\kappa+1)^2 t} \|X^0\|_{L^2(\Omega,L^2(\S^2))}
 		+ C\kappa^{-\beta},
 \end{equation*}
 where $\beta = \eta+1$ under Assumption~\ref{ass:StrongLevyProcess} and $\beta = \alpha/2$ under Assumption~\ref{ass:SpecialLevyProcess}.
\end{theorem}

\begin{proof}
    Using the series expansions \eqref{eq:SolSeries2} and~\eqref{eq:ApproxSpace}, we obtain by the triangle inequality
    \begin{align*}
                \| X(t) - X^{(\kappa)}(t)\|_{L^2(\Omega, L^2(\S^2))}
                & \leq \Big\| \sum_{\ell=\kappa+1}^\infty \sum_{m=-\ell}^\ell  e^{-\lambda_\ell  t} X_{\ell,m}^0 \realSH_{\ell,m} \Big\|_{L^2(\Omega, L^2(\S^2))}\\
                & \qquad + \Big\| \sum_{\ell=\kappa+1}^\infty \sum_{m=-\ell}^\ell \convLlm (t)\realSH_{\ell,m}\Big\|_{L^2(\Omega, L^2(\S^2))}.
    \end{align*}
    The first summand satisfies analogously to \cite[Lemma 3.1]{lang2023}
    \begin{equation*}
                \Big\| \sum_{\ell=\kappa+1}^\infty \sum_{m=-\ell}^\ell  e^{-\lambda_\ell  t} X_{\ell,m}^0 \realSH_{\ell,m} \Big\|_{L^2(\Omega, L^2(\S^2))} 
                \leq e^{-(\kappa+1)(\kappa +2)t} \|X^0\|_{L^2(\Omega,L^2(\S^2))},
    \end{equation*}  
    while \eqref{eq:SecMom4} in the proof of Lemma~\ref{L:MESol}, implies for the second term   
    \begin{align*}
            &\Big\| \sum_{\ell=\kappa+1}^\infty \sum_{m=-\ell}^\ell \convLlm(t) \realSH_{\ell,m}\Big\|^2_{L^2(\Omega, L^2(\S^2))}\\
            &\quad = \sum_{\ell=\kappa+1}^\infty \sum_{m=-\ell}^\ell \Big(
            \varlm\int_0^t e^{-2\lambda_\ell  (t-s)} \diff s + \Big( \meanlm \int_0^t e^{-\lambda_\ell (t-s)} \diff s\Big)^2\Big)\\
            &\quad \leq \sum_{\ell=\kappa+1}^\infty \sum_{m=-\ell}^\ell \Big(
            \varlm\frac{1}{2}\ell^{-2} + \meanlm^2 \ell^{-4} \Big),
    \end{align*}
    where we computed and bounded the integrals.

    If Assumption~\ref{ass:StrongLevyProcess} is satisfied for some $\eta >-1$, 
    \begin{align*}
            \sum_{\ell=\kappa+1}^\infty \sum_{m=-\ell}^\ell 
            \varlm \ell^{-2}
            & = \sum_{\ell=\kappa+1}^\infty \sum_{m=-\ell}^\ell 
            \varlm(1+\ell (\ell +1))^{\eta}(1+\ell (\ell +1))^{-\eta}\ell^{-2}\\ 
            & \leq  C (\kappa +1)^{-2-2\eta} \sum_{\ell=0}^\infty \sum_{m=-\ell}^\ell 
            \varlm (1+\ell (\ell +1))^{\eta}
            \leq C\kappa^{-2(1+\eta)} \|\std\|_{H^\eta(\S^2)}^2,
    \end{align*}
    and with the same arguments,
    \begin{equation*}
        \sum_{\ell=\kappa+1}^\infty \sum_{m=-\ell}^\ell \meanlm^2\ell^{-4} 
            \leq C \kappa^{-2(2+\eta)} \|\mean\|^2_{H^\eta(\S^2)}.
    \end{equation*}
    This proves the claim with $\beta = 1+\eta$ and implies under Assumption~\ref{ass:SpecialLevyProcess} convergence of order up to $\alpha/2$ but not including $\beta=\alpha/2$.
    We sharpen this bound by using the properties of the special L\'evy process in Assumption \ref{ass:SpecialLevyProcess}
    \begin{align*}
            \sum_{\ell=\kappa+1}^\infty \sum_{m=-\ell}^\ell \Big(
            \varlm\frac{1}{2}\ell^{-2} + \meanlm^2 \ell^{-4} \Big)
            & = \sum_{\ell=\kappa+1}^\infty \sum_{m=-\ell}^\ell \Big(
            a_\ell^2 \hatvarlm \frac{1}{2}\ell^{-2} + a_\ell^2 \hatmeanlm^2 \ell^{-4} \Big)\\ 
            & \leq C \hatvar \sum_{\ell=\kappa+1}^\infty (2\ell +1) \ell^{-\alpha-2} + C\hatmean^2 \sum_{\ell=\kappa+1}^\infty (2\ell +1) \ell^{-\alpha-4}
    \end{align*}
    where $\hat v,\ \hat m$ are defined as in Assumption \ref{ass:SpecialLevyProcess}. This expression is bounded by $\kappa^{-\alpha}$ as computed, e.g., in the proof of \cite[Prop.~5.2]{lang2015}, which yields the improved bound~$\alpha/2$.
\end{proof}

Having established strong convergence rates, we continue with weak convergence estimates. 
More precisely, taking an appropriate class of test functions $\varphi$, we aim to show that $\mathbb{E}[\varphi(X^{(\kappa)}(t))]$ converges to $\mathbb{E}[\varphi(X(t))]$, where $X$ is given by \eqref{eq:SolSeries2} and $X^{(\kappa)}$ by its spectral approximation~\eqref{eq:ApproxSpace}.
The reason for such an effort lies in the fact that for the approximation of stochastic partial differential equations the weak convergence rate is expected to be up to twice the strong convergence rate; see \cite{Lang_WR} for more details.

Based on the strong convergence proof of Theorem~\ref{L:SA1}, we obtain weak rates for the expectation and second moment as special test functions. While the error of the expectation contains an additional term compared to the Wiener case in~\cite[Lemma 3.2]{lang2023}, since the L\'evy process is not assumed to be centered, the second moment converges with the same rate under Assumption~\ref{ass:SpecialLevyProcess}.

\begin{corollary}\label{L:SA2}
 Let $\beta = \eta +1$ under Assumption~\ref{ass:StrongLevyProcess} and $\beta = \alpha/2$ under Assumption~\ref{ass:SpecialLevyProcess}.
 Given the mild solution~$X$ in~\eqref{eq:SolSeries2} and its spectral approximation~$X^{(\kappa)}$ in~\eqref{eq:ApproxSpace}, the error of the mean is bounded by
 \begin{equation*}
 	\| \E[X(t)] -\E[X^{(\kappa)}(t)]\|_{L^2(\S^2)} \leq  e^{-(\kappa+1)^2 t} \|\E[X^0]\|_{L^2(\S^2)}
 	+ C \kappa^{-(\beta+1)}.
 \end{equation*}
 Furthermore, the error of the second moment is bounded by
 \begin{equation*}
    |\E[\|X(t)\|^2_{L^2(\S^2)} - \|X^{(\kappa)}(t)\|^2_{L^2(\S^2)} ]|
 	 	\leq 2 e^{-(\kappa+1)^2 t} \|X^0\|^2_{L^2(\Omega,L^2(\S^2))}
 	            + C \kappa^{-2\beta}.
 \end{equation*}
\end{corollary}

\begin{proof}
   We plug in the expressions for the two expectations derived in Lemma~\ref{L:MESol} and~\eqref{eq:expect_spectral} to obtain 
    \begin{equation*}
            \| \E[X(t)] -\E[X^{(\kappa)}(t)]\|_{L^2(\S^2)} 
             \le e^{-(\kappa+1)^2 t} \|\E[X^0]\|_{L^2(\S^2)}
            	+ \Bigl(\sum_{\ell=\kappa +1}^\infty \sum_{m=-\ell}^\ell \ell^{-4}  \meanlm^2\Bigr)^{1/2}.
    \end{equation*}
    The claim follows with the same computations as in Theorem~\ref{L:SA1}.
    To bound the error of the second moment, we observe that the same computations as in \cite[Prop.~4]{cohen2022} yield
    \begin{equation*}
            | \E[\|X(t)\|^2_{L^2(\S^2)} - \|X^{(\kappa)}(t)\|^2_{L^2(\S^2)} ]|
             = \E[\|  X(t) - X^{(\kappa)}(t) \|^2_{L^2(\S^2)}  ],
    \end{equation*}
    which implies the claim by Theorem~\ref{L:SA1}.
\end{proof}

Next we prove weak error rates for the spectral approximation in space for a broader class of test functions. We obtain convergence rates that depend on the regularity of the test function, the solution, the initial condition, and the noise. For our estimates we use a similar strategy to \cite{cohen2022}, where weak rates for the spectral approximation of the stochastic wave equation on the sphere are proven.
We first present the assumption on the test functions. 

\begin{assumption}\label{ass:TestFunc}
    Let $\varphi$ be a Fréchet differentiable test function that satisfies for fixed $\varrho \ge 0$ and all $t\in[0,T]$,
    \begin{equation}
        \| \int_0^1 \varphi' (y X(t) + (1- y) X^{(\kappa)}(t)) \diff y \|_{L^2(\Omega, H^\varrho(\S^2))} \leq C<\infty.
    \end{equation}
\end{assumption}

As shown in \cite{cohen2022}, independently of the specific equation, if the solution to \eqref{eq:HeatEq} satisfies $X(t)\in L^{2q}(\Omega, H^\varrho(\S^2))$, this implies that the class of test functions includes those with polynomial growth of the derivative in 
$H^\varrho(\S^2)$ up to degree $q$, i.e.
\begin{equation}\label{eq:test_function-reg-poly}
    \|\varphi'(f)\|_{H^\varrho(\S^2)}\leq C\left(1+\|f\|^{q}_{H^\varrho(\S^2)}\right),\text{ for all }f\in H^\varrho(\S^2).
\end{equation}

With this, we are able to state our weak convergence result.

\begin{theorem}\label{weakr}
Assume that $X$ and $X^{(\kappa)}$ are given by~\eqref{eq:SolSeries2} and~\eqref{eq:ApproxSpace}, respectively, and $\varphi$ satisfies Assumption~\ref{ass:TestFunc} for some $\varrho \ge 0$. Then the weak error is bounded by
\begin{equation*}
    \left|\mathbb{E}\left[ \varphi(X(t)) - \varphi(X^{(\kappa)}(t))\right]\right|
    \leq C \big( e^{-(\kappa+1)^2 t} (\kappa+1)^{-\varrho} \|X^0\|_{L^2(\Omega,L^2(\S^2))}
 		+ \kappa^{-(\beta+\varrho)}\big),
\end{equation*}
where $\beta = \eta+1$ under Assumption~\ref{ass:StrongLevyProcess} and $\beta = \alpha/2$ under Assumption~\ref{ass:SpecialLevyProcess}.
\end{theorem}

\begin{proof}
As in \cite{cohen2022}, we consider the Gelfand triple $V\subset H\subset V^*$ with $V=H^\varrho(\S^2),\ H=L^2(\S^2),\ V^*=H^{-\varrho}(\S^2)$ and obtain using the mean value theorem and Hölder's inequality
\begin{align*}
& \left |\mathbb{E}\left[\varphi(X(t)) - \varphi(X^{(\kappa)}(t))\right]\right|\\
&\quad=\left |\mathbb{E}\left[\left<\int_0^1\varphi'\left(y X(t)+(1-y)X^{(\kappa)}(t)\right)\diff y , X(t)-X^{(\kappa)}(t)\right>_{V\times V^*} \right]\right|\\
&\quad \leq \Bigl\|\int_0^1\varphi'\left(y X(t)+(1-y)X^{(\kappa)}(t)\right) \diff y \Bigr\|_{L^2(\Omega,V)}\|X(t)-X^{(\kappa)}(t)\|_{L^2(\Omega,V^*)}.
\end{align*} 
The first norm is bounded by Assumption~\ref{ass:TestFunc}, while the second norm is the strong error computed in the weaker $H^{-\varrho}(\S^2)$ norm than in Theorem~\ref{L:SA1}. Following the proof of Theorem~\ref{L:SA1} with new weights therefore yields
\begin{align*}
    &\|X(t)-X^{(\kappa)}(t)\|^2_{L^2(\Omega,V^*)}
     \leq \sum_{\ell=\kappa+1}^\infty\sum_{m=-\ell}^\ell \Big( e^{-2\lambda_\ell  t}\mathbb{E}\left[|X^0_{\ell,m}|^2\right] 
        + \varlm\frac{1}{2}\ell^{-2} + \meanlm^2 \ell^{-4} \Big)(1+\lambda_\ell  )^{-\varrho}
\end{align*}
and increased polynomial convergence rates in the initial condition and noise by~$\varrho$. This finishes the proof.
\end{proof}

We observe that under the assumptions of Proposition~\ref{prop:regul},
\begin{align*}
	\sup_{\kappa\in\mathbb N} \|X^{(\kappa)}(t)\|_{L^{p}(\Omega,H^\varrho(\S^2))}^{p}
    \le \E\Bigl[ \bigl(\sum_{\ell=0}^\infty \sum_{m=-\ell}^\ell X_{\ell,m}(t)^2 (1+\lambda_\ell  )^\varrho \bigr)^{p/2}\Bigr]
    = \|X(t)\|_{L^{p}(\Omega,H^\varrho(\S^2))}^{p}.
\end{align*}
If $X^0\in L^p(\Omega, H^{\gamma}(\S^2))$ with $\gamma \geq \varrho$ this implies
\begin{equation}\label{unifRegSpecApprox}
    \sup_{t\in[0,T]}\|X^{(\kappa)}(t)\|_{L^{p}(\Omega,H^{\varrho}(\S^2))}
\le\sup_{t\in[0,T]}\|X(t)\|_{L^{p}(\Omega,H^{\varrho}(\S^2))}<\infty.
\end{equation}
Hence, under these assumptions every test function satisfying \eqref{eq:test_function-reg-poly} with $q=p/2$ satisfies Assumption~\ref{ass:TestFunc} with a constant independent of~$\kappa$ and $t$. This also implies that the class of test functions is not trivial.

Based on the general weak convergence result and our regularity estimates for the solution in Proposition~\ref{prop:regul}, we summarize the results in the following corollary.
We obtain the usual rule of thumb with a weak convergence rate that is twice the strong one in the first two settings of Proposition~\ref{prop:regul}. For the more general third setting, the rate is essentially $2\eta+1+1/q$, i.e., twice the strong rate when $\varphi'$ grows linearly and else depending on the degree of the polynomial growth condition \eqref{eq:test_function-reg-poly}.

\begin{corollary}\label{weakr_spec}
Assume that $X$ and $X^{(\kappa)}$ are given by~\eqref{eq:SolSeries2} and~\eqref{eq:ApproxSpace}, respectively.
Let one of the following assumptions hold:
\begin{enumerate}
    \item  $p\geq2$, $L = W$ is a Wiener process in $L^p(\Omega, H^\eta(\S^2))$, and set $\varrho=\beta=\eta +1$; 
    \item $p \geq 2$, $L$ satisfies Assumption~\ref{ass:SpecialLevyProcess}, $(\hat L_{\ell,m})_{\ell,m}$ i.i.d.\ and $\mathbb E[|\hat L_{0,0}(1)|^{p}]<\infty$, set $\varrho<\alpha/2$ and $\beta=\alpha/2$;
    \item $p \ge 2$, $L\in L^p(\Omega,H^\eta(\mathbb S^2))$ satisfies Assumption~\ref{ass:StrongLevyProcess}, set $\varrho\leq \eta+2/p$ and $\beta=\eta+1$.
\end{enumerate}
Further, assume $X^{0}\in L^{p}(\Omega,H^\gamma(\S^{2}))$ with $\gamma \geq \varrho$.
Then for all test functions $\varphi$ satisfying~\eqref{eq:test_function-reg-poly} with $q=p/2$ and all
$t\in[0,T]$,
\[
\bigl|\E\bigl[\varphi(X(t))-\varphi(X^{(\kappa)}(t))\bigr]\bigr|
\le C\bigl(e^{-(\kappa+1)^{2}t}(\kappa+1)^{-\varrho}
   \|X^{0}\|_{L^{2}(\Omega,L^{2}(\S^{2}))}+\kappa^{-(\beta+\varrho)}\bigr).
\]
\end{corollary}

We conclude this section with summarizing the results in the different settings for the spectral approximation with truncation parameter~$\kappa$. Theorem~\ref{L:SA1} shows that the strong spectral convergence rate is of order $\kappa^{-\beta}$, with $\beta=\eta+1$ under Assumption \ref{ass:StrongLevyProcess} and $\beta=\alpha/2$ under Assumption~\ref{ass:SpecialLevyProcess}. The latter ensures $L \in L^2(\Omega, H^\eta(\S^2))$ for all $\eta < \alpha/2 - 1$, hence the boundary case can be approached, giving improved rates. 

In the weak sense, Theorem~\ref{weakr} and Corollary~\ref{weakr_spec} are able to enhance the rate to up to $\kappa^{-2\beta}$, reflecting the classical doubling of rates compared to strong convergence.
The assumptions determine the attainable weak rates. 
For a Wiener process or a Lévy process satisfying Assumption~\ref{ass:SpecialLevyProcess} with independent components, we obtain essentially twice the strong rate. Given sufficiently high moment bounds, these processes allow for a rich class of test functions with integrability $p$, which does not affect the convergence rate. In contrast, under Assumption \ref{ass:StrongLevyProcess}, where the noise is more general, the rate becomes $\kappa^{-(2\eta+1+2/p)}$. This introduces an explicit dependence on the parameter $p$: a larger $p$ permits a broader class of test functions but yields a reduced convergence rate. Example~\ref{ex:optimality} illustrates that it is not only the independence of the modes but also the common structure of Assumption~\ref{ass:SpecialLevyProcess} that removes this dependence on $p$.

\section{Euler--Maruyama approximation in time}\label{Sec:4}

In the previous section, we have approximated the solution to~\eqref{eq:HeatEq} in space by spatial truncation in~\eqref{eq:ApproxSpace}. To compute the solution with this scheme, one would have to simulate the stochastic convolution, which is possible in the Wiener case as shown in~\cite{lang2015}. Unfortunately, this is no longer the case for more general L\'evy noise. Therefore, we need to introduce an additional time discretization. As in~\cite{lang2023}, we introduce a forward and a backward Euler--Maruyama scheme and prove convergence. 

For $n\in\N$, we take an equidistant grid on the interval~$[0,T]$ with grid size $h = T/n$ and grid points $t_k = k h$ for $k\in\{0,\ldots,n\}$. The forward resp.\ backward Euler--Maruyama scheme for~\eqref{eq:SDEComp} is given by
\begin{equation}\label{eq:es}
    X^{(h)}_{\ell,m}(t_k) = \xi_\ell(h) X^{(h)}_{\ell,m}(t_{k-1}) +\xi_\ell(h)^\delta \Delta L_{\ell,m}(t_k),
\end{equation}
where $\Delta L_{\ell,m}(t_k) = L_{\ell,m}(t_k) - L_{\ell,m}(t_{k-1})$, $\xi_\ell(h) = (1- \lambda_\ell  h)$ and $\delta =0$ for the forward scheme, while $\xi_\ell(h) = (1+ \lambda_\ell   h)^{-1}$ and $\delta =1$ for the backward scheme.
Recursively, we obtain
\begin{equation}\label{eq:EulerSchemeComp}
    X^{(h)}_{\ell,m}(t_k) = \xi_\ell(h)^k X_{\ell,m}(0)+ \sum_{j=1}^k \xi_\ell(h)^{k-j+\delta} \Delta L_{\ell,m}(t_j),
\end{equation}
and hence, the fully discrete approximation of~\eqref{eq:HeatEq} is given by
\begin{equation}\label{eq:EulerScheme}
    X^{(\kappa, h)}(t_k) = \sum_{\ell=0}^\kappa \sum_{m=-\ell}^\ell X^{(h)}_{\ell,m}(t_k) \realSH_{\ell,m}.
\end{equation}

As discussed in~\cite{lang2023}, for the forward Euler--Maruyama scheme, stability is only guaranteed if $\lambda_\ell  h\le 1$, while the backward Euler--Maruyama scheme is unconditionally stable. Therefore, the backward scheme does not require any time restrictions and couplings of the spatial and temporal discretizations. As usual for error estimates of mild solutions with semigroups, see e.g., \cite{T06}, we couple the discretizations for optimal error estimates, which is in our context $\lambda_\ell  h\le C_c$ for some finite constant~$C_c$. Hence, in the following we assume $\lambda_\ell h\le C_c$ where $C_c \leq 1$ for the forward and $C_c$ is a finite constant for the backward Euler--Maruyama scheme.

The expectation of~$X^{(\kappa,h)}$ satisfies
	\begin{equation}\label{eq:expect_EM}
			\E[ X^{(\kappa, h)}(t_k) ] 
			= \sum_{\ell=0}^\kappa \sum_{m=-\ell}^\ell \Big( \xi_\ell(h)^k \E[ X_{\ell,m}^0] + \sum_{j=1}^k \xi_\ell(h)^{k-j+\delta} h\, \meanlm \Big)  \realSH_{\ell,m},
	\end{equation}
which follows by the linearity of the expectation. To show that the second moment satisfies
	\begin{equation}\label{eq:second-mom_EM}
			\begin{split}
			&\E[ \| X^{(\kappa, h)}(t_k)\|^2_{L^2(\S^2)} ]\\
			&\quad=\sum_{\ell=0}^\kappa \sum_{m=-\ell}^\ell \E\Big[\Big( X_{\ell,m}^0 \xi_\ell(h)^{k} + \meanlm h \sum_{j=1}^k \xi_\ell(h)^{k-j+\delta}\Big)^2\Big]
			+  \varlm h \sum_{j=1}^k \xi_\ell(h)^{2(k-j+\delta)} \Big),
			\end{split}
	\end{equation}   
we plug in \eqref{eq:EulerScheme} and compute the square obtaining
	\begin{align*}
			&\E[ \| X^{(\kappa, h)}(t_k)\|^2_{L^2(\S^2)} ]\\
            &\quad  = \E\Big[ \Big\| \sum_{\ell=0}^\kappa \sum_{m=-\ell}^\ell \xi_\ell(h)^k X_{\ell,m}(0) \realSH_{\ell,m} \Big\|^2_{L^2(\S^2)} \Big]\\
			& \quad\quad + 2 \E\Big[ \Big\langle  \sum_{\ell=0}^\kappa \sum_{m=-\ell}^\ell \xi_\ell(h)^k X_{\ell,m}^0 \realSH_{\ell,m}, \sum_{\ell=0}^\kappa \sum_{m=-\ell}^\ell \sum_{j=1}^k \xi_\ell(h)^{k-j+\delta} \Delta L_{\ell,m}(t_j) \realSH_{\ell,m} \Big\rangle_{L^2(\S^2)} \Big]\\
			& \quad\quad + \E\Big[ \Big\| \sum_{\ell=0}^\kappa \sum_{m=-\ell}^\ell \sum_{j=1}^k \xi_\ell(h)^{k-j+\delta} \Delta L_{\ell,m}(t_j) \realSH_{\ell,m} \Big\|^2_{L^2(\S^2)} \Big].
	\end{align*}
	While the first and third term are computed as in~\cite{lang2023},
	the additional second term satisfies with the independence of the Lévy process and the initial condition that
	\begin{align*}
		&\E\Big[ \Big\langle  \sum_{\ell=0}^\kappa \sum_{m=-\ell}^\ell \xi_\ell(h)^k X_{\ell,m}^0 \realSH_{\ell,m}, \sum_{\ell=0}^\kappa \sum_{m=-\ell}^\ell \sum_{j=1}^k \xi_\ell(h)^{k-j+\delta} \Delta L_{\ell,m}(t_j) \realSH_{\ell,m} \Big\rangle_{L^2(\S^2)} \Big]\\
		&\quad= \E\Big[ \sum_{\ell=0}^\kappa \sum_{m=-\ell}^\ell \xi_\ell(h)^k X_{\ell,m}^0 \sum_{j=1}^k \xi_\ell(h)^{k-j+\delta} \Delta L_{\ell,m}(t_j) \Big]\\
		&\quad= \sum_{\ell=0}^\kappa \sum_{m=-\ell}^\ell \xi_\ell(h)^k h  \sum_{j=1}^k \xi_\ell(h)^{k-j+\delta}\, \E[ X_{\ell,m}^0] \meanlm.
	\end{align*}

Moreover, similarly to $X$ and $X^{(\kappa)}$, we obtain the following spatial regularity result for $X^{(\kappa, h)}$.

\begin{proposition}\label{prop:EMreg}
    Let $\varrho\in\R$ and let $X^{(\kappa,h)}$ be the Euler--Maruyama approximation defined by~\eqref{eq:EulerScheme} with initial value $X^0\in L^{p}(\Omega,H^\gamma(\S^2)),\, \gamma \in \R$. If one of the following assumptions holds:
\begin{enumerate}
    \item  $p\geq2$, $L = W$ is a Wiener process in $L^p(\Omega, H^\eta(\S^2))$, and $\varrho-\eta\leq1$;\label{assRegEW}
    \item $p\geq 2$, $L$ satisfies Assumption~\ref{ass:SpecialLevyProcess}, $(\hat L_{\ell,m})_{\ell,m}$ i.i.d., $\mathbb E[|\hat L_{0,0}(1)|^{p}]<\infty$, and $\varrho<\alpha/2$;\label{assRegESpecLP}
    \item $p \ge 2$, $L$ satisfies Assumption~\ref{ass:StrongLevyProcess} with $p$, and $\varrho-\eta\le2/p$;\label{assRegEGenLP} 
\end{enumerate}
then 
\begin{align}
\|X^{(\kappa,h)}(t_k)\|_{L^{p}(\Omega,H^\varrho(\S^2))} \leq C \big(1 
	+  t_k^{-\max\{\varrho-\gamma,0\}/2}\|X^0\|_{L^{p}(\Omega,H^\gamma(\mathbb{S}^2))}\big),
\end{align}
where $C$ is a constant dependent on $T$, $\gamma$, and $\varrho$. 
\end{proposition}

\begin{proof}
    To simplify notations we define similarly to \cite{kruse2014} the operators $\approxOp^k$ and $\proOp$ for $f\in H^\eta(\S^2)$ by 
    \begin{equation}\label{ratapx}
    	\approxOp f= \sum_{\ell = 0}^{\infty}\sum_{m=-\ell}^\ell \xi_\ell(h)\, \langle f, \realSH_{\ell,m}\rangle \realSH_{\ell,m},
        \qquad 
        \proOp f= \sum_{\ell = 0}^{\kappa}\sum_{m=-\ell}^\ell \langle f, \realSH_{\ell,m}\rangle \realSH_{\ell,m},
    \end{equation}
    where $\langle \cdot,\cdot\rangle$ is understood as in \eqref{eq:def:Llm}.
    With these we define 
    \begin{equation}\label{eq:defDiscConvL}
        \begin{aligned}
            & \discConvL(t_k) = \sum_{j=1}^k \approxOp^{k-j+\delta} \proOp \Delta L(t_j).
        \end{aligned}
    \end{equation}
    Using this notation and the triangle inequality, we split the norm
\begin{align*}
    &\|X^{(\kappa,h)}(t_k)\|_{L^{p}(\Omega,H^\varrho(\S^2))}
    	\leq \Bigl\|\approxOp^k  \proOp X(0)\Bigr\|_{L^{p}(\Omega,H^\varrho(\S^2))}
    		+ \|\discConvL(t_k)\|_{L^{p}(\Omega,H^\varrho(\S^2))}.
\end{align*}
For the first term we note that by \cite[p.~116]{T06} there exists a constant $C$ dependent on $C_c$ such that for all $k\in\N$
\[ \xi_l(h)^{k} \leq e^{-C \lambda_\ell t_k}. \]
Using this and the same arguments as in the proof of Proposition \ref{prop:regul} we obtain
\begin{equation}\label{Esemiest-ic}
	\Bigl\|\approxOp^k  \proOp X(0)\Bigr\|_{L^{p}(\Omega,H^\varrho(\S^2))}
		\le C\, t_k^{- \max\{\varrho-\gamma,0\}/2}\|X^0\|_{L^p(\Omega,H^\gamma(\S^2))}.
\end{equation}

To prove the claim assuming \ref{assRegSpecLP}, note that for all $\ell \geq 1$ and $q\geq 1$ using the geometric series,
\begin{align}
    h\sum_{j=1}^{k}\xi_{\ell}(h)^{q(k-j+\delta)} \leq h\, \frac{1}{1-\xi_\ell(h)} 
    = 
    \begin{cases}
    \frac{h}{h\lambda_\ell  } =\frac{1}{\lambda_\ell  }, & \xi_\ell(h) = (1- \lambda_\ell  h)\\
    \frac{h(1+h\lambda_\ell) }{h\lambda_\ell  }\leq \frac{1+C_c}{\lambda_\ell},&\xi_\ell(h) = (1+ \lambda_\ell h)^{-1}
    \end{cases}
    \leq C \frac{1}{1+\lambda_\ell}.
\end{align}
Hence, defining for all $s\in[0,T] $ the projection to the previous grid point by $\underline{s} =\max\{t_k: k\in\{0,\ldots,n\}, t_k\leq s\}$ we obtain for all $\ell \in\N_0$ and $q\geq 1$, 
\begin{align}\label{eq:estDiscSum}
    &\int_0^{t_k} \xi_\ell(h)^{q(k-(\underline{s}+h)h^{-1}+\delta)} \diff s = h\sum_{j=1}^{k}\xi_{\ell}(h)^{q(k-j+\delta)} 
    \leq C \frac{1}{1+\lambda_\ell}.
\end{align}

Using this and \cite[Example 15.12]{brockwell2024} we obtain,
\begin{equation}\label{Efirstsummand}
\begin{aligned}
    \mathbb E\bigl[\|\discConvLlm(t_k)Y_{\ell,m} \|_{H^\varrho(\S^2)}^2 \bigr]
    & = (1+\lambda_{\ell})^{\varrho} \, \E\Big[\Big| \int_0^{t_k}\xi_\ell(h)^{k-(\underline{s}+h)h^{-1}+\delta} \diff L_{\ell,m}(s) \Big|^2\Big] \\
    & \le C \bigl( (1+\lambda_{\ell})^{\varrho-1} \varlm +(1+\lambda_{\ell})^{\varrho-2} \meanlm^{2} \bigr).
\end{aligned}
\end{equation}
Further, combining the ideas of \eqref{Efirstsummand} and \eqref{secondsummand}, by Lemma~\ref{lm:jump-p-bnd}, \cite[Example 15.12]{brockwell2024}, \eqref{eq:estDiscSum}, and the superadditivity of $x\mapsto x^{p/2}$, 
\begin{align}\label{Esecondsummand}
&\sum_{\ell =0}^\kappa \sum_{m =-\ell}^\ell \mathbb E\bigl[\|\discConvLlm(t_k)Y_{\ell,m} \|_{H^\varrho(\S^2)}^{p}\bigr]\\
& \le C\Big( \sum_{\ell =0}^\infty \sum_{m =-\ell}^\ell 
(1+\lambda_{\ell})^{\varrho-1}v_{\ell,m}
+(1+\lambda_{\ell})^{\varrho-2}|\tilde m_{\ell,m}|^{2}\Big)^{p/2}
+  C \sum_{\ell =0}^\infty \sum_{m =-\ell}^\ell 
(1+\lambda_{\ell})^{\varrho p/2-1}\left|\nu_{\ell,m}\right|^{(p)}.
\end{align}

Hence, similar to the proof of Proposition~\ref{prop:regul}\ref{assRegSpecLP} we obtain applying \cite[Corollary 3]{Latala_momentsSum} and combining  \eqref{Efirstsummand}, \eqref{Esecondsummand}, \eqref{eq:est1ii},  and \eqref{eq:est2ii} that
\begin{align}
\mathbb E\Big[\bigl\|\discConvL(t_k)\bigr\|^{p}_{H^{\varrho}(\S^2)} \Big]
\le C\Big(\sum_{\ell =1}^\infty \ell^{2\varrho-1-\alpha}\Big)^{p/2} +  C \sum_{\ell=1}^\infty\ell^{\varrho p-\alpha p/2-1} +C.
\end{align}
The right-hand side is finite if $\varrho<\alpha/2$.

Assuming \ref{assRegGenLP} we define
\begin{align}
    &\discConvW(t_k) = \sum_{j=1}^k \approxOp^{k-j+\delta} \proOp \Delta W_{t_j},\\ &\discConvN(t_k) = \sum_{j=1}^k \approxOp^{k-j+\delta} \proOp \int_{t_{j-1}}^{t_j}\!\! \int_{\chi_\eta} \!\!\zeta \widetilde N (\diff s, \diff \zeta)
\end{align}
and obtain by the linearity of $\proOp$ and $\approxOp$,
\begin{align*}
   &\E \big[ \| \discConvL(t_k)\|^p_{H^\varrho(\S^2)}\big] \leq C\Big(\Big \|\sum_{j=1}^k \approxOp^k \proOp \mean \Big\|^p_{H^\varrho(\S^2)}\\
   &\hspace{5.1cm}+ \E \big[ \|\discConvW(t_k) \|^p_{H^\varrho(\S^2)}\big] + \E \big[ \|\discConvN(t_k)\|^p_{H^\varrho(\S^2)}\big] \Big).
\end{align*}
The bound for the second term is proven using \cite[Theorem 4.36]{DaPrato_Zabczyk_1992} and \eqref{eq:estDiscSum}, which also shows the case~\ref{assRegW}. 

For the first and third term it is enough to consider $\varrho=\eta+2/p$. 
For all $\zeta\in H^{\eta}(\mathbb S^{2})$ and $t\in[0,T]$, by the definition of $\approxOp$ and $\proOp$ and H\"older's inequality,
\begin{align*}
  \|\approxOp^{k-j+\delta} \proOp \zeta \|_{H^{\eta+2/p}(\S^2)}^2
  &\le\Bigl(\sum_{\ell=0}^\kappa \sum_{m=-\ell}^\ell(1+\lambda_\ell)^{\eta+1} \xi_\ell(h)^{p(k-j+\delta)} \zeta_{\ell,m}^{2}\Bigr)^{2/p}
    \|\zeta\|^{2(1-2/p)}_{H^{\eta}(\S^2)} .
\end{align*}
Hence, by \eqref{eq:estDiscSum} similar as in \eqref{eqneeded},
\begin{align*}
  &\sum_{j=1}^k \int_{t_{j-1}}^{t_j}  \|\approxOp^{k-j+\delta} \proOp \zeta \|_{H^{\eta+2/p}(\S^2)}^p \diff s \le C \|\zeta\|^{p}_{H^{\eta}(\S^2)}.
\end{align*}

Using these estimates, \cite[Lemma~3.1]{MARINELLI2010616}, Fubini's theorem, Minkovski's integral inequality, and Assumption~\ref{ass:StrongLevyProcess} with \eqref{eq:finite_pth_moment_Levy_measure}, we obtain $\E\left[\|\discConvN(t)\|^p_{H^{\eta+2/p}(\S^2)}\right]\leq C$ and with similar estimates that 
$
\|\sum_{j=1}^k \approxOp^k \proOp \mean \|^p_{H^{\eta+2/p}(\S^2)} \le C,
$
where $C\in(0,\infty)$ depends on $T$, $p$, $\eta$ and $\nu^{(\eta)}$. 
\end{proof}

The following proposition will be needed in the convergence proofs and summarises results on the approximation of the exponential function mainly from~\cite{lang2023}.

\begin{proposition}\label{prop:AuxExp}
   For all $\mu\in(0,1]$ and $\gamma, \varrho \ge 0$, there exist constants $C_{\mu}, C_{\mu,\varrho,\gamma} \in(0,\infty)$ such that for all $\ell,k\in\N$ and $h\in(0,\infty)$ with $\lambda_\ell  h \leq C_c$ and $a \in \{1,2\}$
        \begin{align}
        |e^{-\lambda_\ell  t_k} -\xi_\ell(h)^k|^a (1+\lambda_\ell  )^{-\varrho} & \le C_{\mu,\varrho,\gamma} h^{a\mu} t_k^{\min\{a,\gamma+\varrho\}-a\mu} \lambda_\ell^{\min\{a-\varrho,\gamma\}},\label{Peqi}\\
        \sum_{j=1}^k \int_{t_{j-1}}^{t_j} \big(e^{-\lambda_\ell (t_k-s)}- \xi_\ell(h)^{k-j+\delta}\big)^2 \diff s 
        & \leq C_{\mu}\lambda_\ell  ^{2\mu-1}h^{2\mu},\label{Peqiii} \\
        \qquad\Big|\sum_{j=1}^k \int_{t_{j-1}}^{t_j} e^{-a \lambda_\ell  (t_k-s)}- \xi_\ell(h)^{a(k-j+\delta)} \diff s\Big| & \leq C_{\mu}\lambda_\ell ^{\mu-1}h^{\mu}.\label{Peqvi} 
        \end{align}
\end{proposition}

\begin{proof}
While \eqref{Peqiii} and \eqref{Peqvi} have been shown in~\cite{lang2023} by a combination of Propositions 4.1--4.3 and parts of the proofs of Theorems~4.4 and~4.5, we show \eqref{Peqi} by using
    \begin{equation*}
        |e^{-\lambda_\ell  h k}-\xi_\ell(h)^k|
             	\leq C_{\mu} \lambda_\ell^{1+\mu}h^{1+\mu}k e^{-\lambda_\ell h (k-1)}
    \end{equation*}
    from the same paper.
    The claim follows with 
   \begin{align*}
   	 |e^{-\lambda_\ell  t_k} -\xi_\ell(h)^k|^a (1+\lambda_\ell  )^{-\varrho}
   		& \le C \lambda_\ell^{a+a\mu} h^{a\mu} t_k^a e^{-a\lambda_\ell t_{k-1}} (1+\lambda_\ell  )^{-\varrho}\\
   		& \le C h^{a\mu} t_k^{\min\{a,\gamma+\varrho\}-a\mu} \lambda_\ell^{\min\{a-\varrho,\gamma\}},
   \end{align*}
   since $\exp(a\lambda_\ell   h) \le \exp(aC_c)$ and $x^{a\mu - \min\{a-\varrho,\gamma\} + a -\varrho} \exp(-x)$ is bounded.
\end{proof}

With all prerequisites at hand, we are now ready to prove first strong convergence in Theorem~\ref{Thm:ConvEM} and then weak convergence for the first and second moments in Theorem~\ref{Thm:ConvEM2} and in general in Theorem~\ref{weakrEM}.

\begin{theorem}\label{Thm:ConvEM}
Let $X^0 \in L^2(\Omega, H^{\gamma}(\S^2))$ for some $\gamma \ge 0$ and let $X^{(\kappa)}$ be given by~\eqref{eq:ApproxSpace} with Euler--Maruyama approximation~$X^{(\kappa,h)}$ defined by~\eqref{eq:EulerScheme}. Then the strong error in the time discretisation is bounded uniformly in~$\kappa$ by
    \begin{equation}
		\| X^{(\kappa)}(t_k) - X^{(\kappa,h)}(t_k) \|_{L^2(\Omega,L^2(\S^2))} \\
		\leq C \big(h^{\mu} t_k^{\min\{1,\gamma/2\}-\mu} \| X^0\|_{L^2(\Omega,H^{\min\{2,\gamma\}}(\S^2))} 
		+ h^{\min\{\beta/2,1\}}\big)
	\end{equation}
for $\mu \in (0,1]$, where $\beta = \eta +1$ under Assumption~\ref{ass:StrongLevyProcess} and $\beta < \alpha/2$ under Assumption~\ref{ass:SpecialLevyProcess}.
\end{theorem}

\begin{proof}
    Plugging in the definitions of the schemes, and the truncation \eqref{eq:ApproxSpace}, and computing the norm yields
    \begin{align*}
            & \| X^{(\kappa)}(t_k) - X^{(\kappa,h)}(t_k) \|^2_{L^2(\Omega,L^2(\S^2))} \\
            & \quad\leq 2 \sum_{\ell=0}^\kappa \sum_{m=-\ell}^\ell \Big((e^{-\lambda_\ell  t_k} -\xi_\ell(h)^k)^2 \E[|X_{\ell,m}^0|^2] \\
            & \hspace*{8em} + \E\Bigl[\Big(\convLlm(t_k)- \sum_{j=1}^k \xi_\ell(h)^{k-j+\delta} \Delta L_{\ell,m}(t_j) \Big)^2\Bigr]\Big).
    \end{align*}
   The first summand satisfies the claimed bound by~\eqref{Peqi} for $\varrho= 0$ using the definition of the Sobolev norm.
   
To write the second term as a stochastic integral, we denote for $s\in[0,T] $ the projection to the previous grid point by $\underline{s} =\max\{t_k: k\in\{0,\ldots,n\}, t_k\leq s\}$. Applying \cite[Example 15.12]{brockwell2024} as earlier and Jensen's inequality, we obtain
\begin{align*}
	& \E\Bigl[\Big(\convLlm(t_k) - \sum_{j=1}^k \xi_\ell(h)^{k-j+\delta} \Delta L_{\ell,m}(t_j) \Big)^2\Bigr]\\
	& \quad= \E\Big[   \Big| \int_{0}^{t_k} e^{-\lambda_\ell  (t_k-s)} - \xi_\ell(h)^{k-(\underline{s}+h)h^{-1}+\delta} \diff L_{\ell,m}(s) \Big|^2\Big]\\
	&\quad \le \big(\varlm +  \meanlm^2 T \big) \sum_{j=1}^k\int_{t_{j-1}}^{t_j} \big(e^{-\lambda_\ell (t_k-s)} - \xi_\ell(h)^{k-j+\delta}\big)^2 \diff s.
\end{align*}
Setting $\mu = \min\{ ( \eta+1)/2,1\}$ in~\eqref{Peqiii} bounds the last expression by
\begin{equation*}
	\sum_{j=1}^k\int_{t_{j-1}}^{t_j} \big(e^{-\lambda_\ell (t_k-s)} - \xi_\ell(h)^{k-j+\delta}\big)^2 \diff s
	\le C \lambda_\ell^{\min\{\eta,1\}} h^{2\min\{(\eta+1)/2,1\}}
\end{equation*}
 with the constant $C$ depending on $\eta$, and yields as bound for the second term using the definition of the Sobolev norms
\begin{align*}
	&  \sum_{\ell=0}^\kappa \sum_{m=-\ell}^\ell \E\Bigl[\Big(\convLlm(t_k) - \sum_{j=1}^k \xi_\ell(h)^{k-j+\delta} \Delta L_{\ell,m}(t_j) \Big)^2\Bigr]\\
	& \quad\le C h^{ 2 \min\{(\eta+1)/2,1\}} \big( \|\std \|^2_{H^{ \eta}(\S^2)} + T \| \mean \|^2_{H^{ \eta}(\S^2)}\big).
\end{align*}
This finishes the proof under Assumption~\ref{ass:StrongLevyProcess}. The claim for Assumption~\ref{ass:SpecialLevyProcess} follows since $L$ in $L^2(\Omega,H^\eta(\S^2))$ for all $\eta < \alpha/2 -1$ by \eqref{special:varconv}.
\end{proof}

As for the spectral approximation we derive weak rates for the expectation and second moment as special test functions based on the strong convergence proof of Theorem~\ref{Thm:ConvEM}. Again we obtain an additional term for the expectation compared to the Wiener case in \cite[Theorem 4.5]{lang2023}, while the second moment converges with the same rate under Assumption \ref{ass:SpecialLevyProcess}.

\begin{theorem}\label{Thm:ConvEM2}
 Let $\beta = \eta +1$ under Assumption~\ref{ass:StrongLevyProcess} and $\beta < \alpha/2$ under Assumption~\ref{ass:SpecialLevyProcess}.
Given the spectral approximation~$X^{(\kappa)}$ in~\eqref{eq:ApproxSpace} and the fully discrete approximation~$X^{(\kappa,h)}$ in~\eqref{eq:EulerScheme} with $X^0 \in L^2(\Omega, H^{\gamma}(\S^2))$ for some $\gamma \ge 0$,  and $\mu \in (0,1]$, the error of the mean is bounded uniformly in~$\kappa$ by
\begin{align*}
	\|\E[X^{(\kappa)}(t_k) - X^{(\kappa,h)}(t_k)]\|_{L^2(\S^2)} 
	&\leq  C \big(h^{\mu} t_k^{\min\{1, \gamma/2\} - \mu} \| \E[X^0]\|_{H^{\min\{2, \gamma\}}(\S^2)}
	+ h^{\min\{1,(\beta+1)/2\}} \big).
\end{align*}
Furthermore, the error of the second moment is bounded uniformly in~$\kappa$ by 
\begin{equation*}
	|\E[\|X^{(\kappa)}(t)\|^2_{L^2(\S^2)} - \|X^{(\kappa,h)}(t)\|^2_{L^2(\S^2)} ]|
	\leq C \big(h^{\mu} t_k^{\min\{1, \gamma\}-\mu} \| X^0\|^2_{L^2(\Omega,H^{\min\{1, \gamma\}}(\S^2))}
	+ h^{\min\{\beta, 1\}}\big).
\end{equation*}
\end{theorem}

\begin{proof}
To compute the error of the expectation, we plug in the representations~\eqref{eq:expect_spectral} and~\eqref{eq:expect_EM} and obtain
    \begin{align*}
            &\|\E[X^{(\kappa)}(t_k) - X^{(\kappa,h)}(t_k) ]\|^2_{L^2(\S^2)}\\
            &\quad\leq 2 \sum_{\ell=0}^\kappa \sum_{m=-\ell}^\ell \Bigl(\big( e^{-\lambda_\ell  t_k} -  \xi_\ell(h)^k \big)^2 \, \E[X_{\ell,m}^0]^2
            + \Big( \int_0^{t_k}  e^{-\lambda_\ell  (t_k-s)} \diff s - \sum_{j=1}^k \xi_\ell(h)^{k-j+\delta} h \Big)^2 \meanlm^2 \Bigr). 
    \end{align*}
    The bound for the first summand follows as in Theorem~\ref{Thm:ConvEM}.
    	For the second summand, we observe computing the integral and the geometric series that
    \begin{equation}\label{eqn:convolution_difference}
		\Big(\int_0^{t_k}  e^{-\lambda_\ell  (t_k-s)} \diff s - \sum_{j=1}^k \xi_\ell(h)^{k-j+\delta} h \Big)^2
			= \lambda_\ell^{-2} (e^{-\lambda_\ell  t_k}- \xi_\ell(h)^k)^2.
	\end{equation}
    Therefore, the second term is equal to the first with coefficients $(\meanlm/\lambda_\ell )^2$, which yields the claim when $\gamma$ is set to $\eta+2$.
    As in the previous theorem, the result for Assumption~\ref{ass:SpecialLevyProcess} follows since $L \in L^2(\Omega,H^\eta(\S^2))$ for all $\eta < \alpha/2 -1$ by \eqref{special:varconv}.
    
    For the second moment we plug in \eqref{eq:second-moment_spectral} and \eqref{eq:second-mom_EM} to obtain using for the non-variance term that $(a^2-b^2) = (a-b)(a+b)$,
    \begin{align*}
 		&\Big| \E\Big[\| X^{(\kappa)}(t_k)\|^2_{L^2(\S^2)}  -\| X^{(\kappa,h)}(t_k) \|^2_{L^2(\S^2)} \Big] \Big|\\
 		&\quad \leq \sum_{\ell=0}^\kappa \sum_{m=-\ell}^\ell \varlm \Big|\int_0^{t_k} e^{-2\lambda_\ell (t_k-s)} \diff s  - h \sum_{j=1}^k \xi_\ell(h)^{2(k-j+\delta)}  \Big|\\
 		& \quad\quad\quad\quad + \Big| \E\Big[ \Big( X^0_{\ell,m} \big( e^{-\lambda_\ell  t_k} -\xi_\ell(h)^{k}\big) + \meanlm \Big( \int_0^{t_k} e^{-\lambda_\ell  (t_k-s)} \diff s  - h \sum_{j=1}^k \xi_\ell(h)^{k-j+\delta}\Big) \Big)\\
            & \quad\quad\quad\quad \quad\quad \cdot \Big( X^0_{\ell,m} \big( e^{-\lambda_\ell t_k} +\xi_\ell(h)^{k}\big) + \meanlm \Big( \int_0^{t_k} e^{-\lambda_\ell (t_k-s)} \diff s  + h \sum_{j=1}^k \xi_\ell(h)^{k-j+\delta}\Big) \Big) \Big] \Big|
    \end{align*} 
    Note that $e^{-\lambda_\ell  }\leq 1$ and $\xi_\ell(h)\leq 1$, hence we use these estimates whenever possible and obtain with $ab \le a^2 + b^2$ that
    \begin{align*}
 		&\Big| \E\Big[\| X^{(\kappa)}(t_k)\|^2_{L^2(\S^2)}  -\| X^{(\kappa,h)}(t_k) \|^2_{L^2(\S^2)} \Big] \Big|\\
 		& \quad\leq \sum_{\ell=0}^\kappa \sum_{m=-\ell}^\ell \varlm \Big|\int_0^{t_k} e^{-2\lambda_\ell (t_k-s)} \diff s  - h \sum_{j=1}^k \xi_\ell(h)^{2(k-j+\delta)}  \Big|\\
            &\quad\quad+ C \sum_{\ell=0}^\kappa \sum_{m=-\ell}^\ell \E[ |X_{\ell,m}^0|^2]  \big| e^{-\lambda_\ell t_k} -\xi_\ell(h)^{k}\big|   \\
            &\quad\quad + C \sum_{\ell=0}^\kappa \sum_{m=-\ell}^\ell \Big( \E[ |X_{\ell,m}^0|]^2 + \meanlm^2  \Big) \big| e^{-\lambda_\ell  t_k} -\xi_\ell(h)^{k}\big|\Big| \int_0^{t_k} e^{-\lambda_\ell  (t_k-s)} \diff s  + h \sum_{j=1}^k \xi_\ell(h)^{k-j+\delta}\Big| \displaybreak[3] \\
            &\quad\quad+ C \sum_{\ell=0}^\kappa \sum_{m=-\ell}^\ell \Big( \E[ |X_{\ell,m}^0|]^2 + \meanlm^2 \Big) \Big| \int_0^{t_k} e^{-\lambda_\ell  (t_k-s)} \diff s  - h \sum_{j=1}^k \xi_\ell(h)^{k-j+\delta}\Big|.
    \end{align*} 
To resort and simplify this expression, we compute with the definition of $\xi_\ell(h)$ and $\delta$
and the bounds on $\lambda_\ell h$ that
\begin{align*}
		&\Big| \int_0^{t_k} e^{-\lambda_\ell  (t_k-s)} \diff s  + h \sum_{j=1}^k \xi_\ell(h)^{k-j+\delta}\Big|
        \leq \Big| \int_0^{t_k} e^{-\lambda_\ell (t_k-s)} \diff s \Big| + \Big| h \sum_{j=1}^k \xi_\ell(h)^{k-j+\delta}\Big|
		\leq C \lambda_\ell^{-1}
\end{align*}
and use~\eqref{eqn:convolution_difference}. This yields
\begin{align*}
 	&\Big| \E\Big[\| X^{(\kappa)}(t_k)\|^2_{L^2(\S^2)}  -\| X^{(\kappa,h)}(t_k) \|^2_{L^2(\S^2)} \Big] \Big|\\
		&\quad \leq \sum_{\ell=0}^\kappa \sum_{m=-\ell}^\ell \varlm \Big|\int_0^{t_k} e^{-2\lambda_\ell (t_k-s)} \diff s  - h \sum_{j=1}^k \xi_\ell(h)^{2(k-j+\delta)}  \Big|\\
            &\quad\quad + C \sum_{\ell=1}^\kappa \sum_{m=-\ell}^\ell \Bigl(\E[ |X_{\ell,m}^0|^2] +\lambda_\ell  ^{-1} (\E[ |X_{\ell,m}^0|]^2 + \meanlm^2)\Bigr)  \big| e^{-\lambda_\ell  t_k} -\xi_\ell(h)^{k}\big|.
\end{align*}
Since $\E[ |X_{\ell,m}^0|]^2$ and $\meanlm^2$ are smoothed with $\lambda_\ell^{-1}$, convergence is dominated by the remaining two terms. While the claim for the first one follows with \eqref{Peqvi}, the second one is bounded by \eqref{Peqi} with $\varrho=0$ by
    \begin{equation*}
	\sum_{\ell=0}^\kappa \sum_{m=-\ell}^\ell\big| e^{-\lambda_\ell  t_k} -\xi_\ell(h)^{k}\big| \E[ |X_{\ell,m}^0|^2]
	\leq C h^{\mu} t_k^{\min\{1, \gamma\}-\mu} \| X^0\|^2_{L^2(\Omega,H^{\min\{1, \gamma\}}(\S^2))},
\end{equation*}
with the constant $C$ depending $\gamma$.
\end{proof}

Next we prove weak error rates for the temporal approximation allowing for a broader class of test functions. These again depend on the regularity of the test function, the solution, the initial condition, and the noise.

\begin{theorem}\label{weakrEM}
	Assume that $X^{(\kappa)}$ and $X^{(\kappa,h)}$ are given by~\eqref{eq:ApproxSpace} and~\eqref{eq:EulerScheme}, respectively, and $\varphi$ satisfies Assumption~\ref{ass:TestFunc} with $X$ replaced by $X^{(\kappa,h)}$ for some $\varrho \ge 0$, where the constant $C$ is independent of $\kappa$ and $h$. Let $\mu\in(0,1]$, then the weak error is uniformly bounded in~$\kappa$ by
	\begin{align*}
		&\left|\mathbb{E}\left[ \varphi(X^{(\kappa)}(t_k)) - \varphi(X^{(\kappa,h)}(t_k))\right]\right|\\
		&\quad \leq C\big( h^{\mu} t_k^{\min\{1,(\gamma+\varrho)/2\}-\mu}  \| X^0\|_{L^2(\Omega,H^{\min\{2-\varrho,\gamma\}}(\S^2))} 
		+ h^{\min\{(\beta+\varrho)/2,1\}}\big),
	\end{align*}
	where $\beta = \eta+1$ under Assumption~\ref{ass:StrongLevyProcess} and $\beta < \alpha/2$ under Assumption~\ref{ass:SpecialLevyProcess}.
\end{theorem}

\begin{proof}
Due to the same arguments as in the proof of Theorem~\ref{weakr}, we only need to bound the strong error in the weaker $H^{-\varrho}(\S^2)$ norm. This gives in the proof of Theorem~\ref{Thm:ConvEM} new weights of the form $(1+\lambda_\ell   )^{-\varrho}$. Bounding the term based on the initial condition with $a=2$ in \eqref{Peqi}, and choosing $\mu = \min\{(\eta+1+\varrho)/2,1\}$ in the estimates of the stochastic term yields the claim of improved convergence rates of up to~$\varrho/2$. This proves the claim under Assumption~\ref{ass:StrongLevyProcess} and in combination with \eqref{special:varconv} under Assumption~\ref{ass:SpecialLevyProcess}.
\end{proof}

Based on Theorem~\ref{weakrEM} and the regularity estimates in \eqref{unifRegSpecApprox} and Proposition~\ref{prop:EMreg}, we conclude analogously to Corollary~\ref{weakr_spec}.

\begin{corollary}\label{weakr_specEM}
Assume that $X^{(\kappa)}$ and $X^{(\kappa,h)}$ are given by~\eqref{eq:ApproxSpace}
and~\eqref{eq:EulerScheme}, respectively.
Let one of the following assumptions hold: 
\begin{enumerate}
    \item  $p\geq2$, $L = W$ is a Wiener process in $L^p(\Omega, H^\eta(\S^2))$, and set $\varrho=\beta=\eta +1$;
    \item $p\geq 2$, $L$ satisfies Assumption~\ref{ass:SpecialLevyProcess}, $(\hat L_{\ell,m})_{\ell,m}$ i.i.d.~and $\mathbb E[|\hat L_{0,0}(1)|^{p}]<\infty$, set $\varrho<\alpha/2$ and $\beta<\alpha/2$;
    \item $p \ge 2$, $L\in L^p(\Omega,H^\eta(\mathbb S^2))$ satisfies Assumption~\ref{ass:StrongLevyProcess}, set $\varrho\leq \eta+2/p$ and $\beta=\eta+1$.
\end{enumerate}
Further assume $X^{0}\in L^{p}(\Omega,H^{\gamma}(\S^{2}))$ with $\gamma=\varrho \ge0$ and $\mu\in(0,1]$. Then,
for all test functions $\varphi$ satisfying~\eqref{eq:test_function-reg-poly} with $q=p/2$ and all
$t_{k}\in(0,T]$,
\[
\left|\E\left[\varphi(X^{(\kappa)}(t_{k}))-\varphi(X^{(\kappa,h)}(t_{k}))\right]\right|
\le C\big(h^{\mu}t_{k}^{\min\{1,\gamma\}-\mu}\|X^{0}\|_{L^{2}(\Omega,H^{\gamma}(\S^{2}))}
 +h^{\min\{(\beta+\varrho)/2,\,1\}}\big),
\]
where $C$ is independent of $\kappa$ and $h$.
\end{corollary}

As in the spectral case, we summarize results for the fully discrete scheme with truncation parameter $\kappa$ and time step $h$. By Theorem~\ref{Thm:ConvEM}, the strong convergence rate is $h^{\min\{\beta/2,1\}}$, where $\beta=\eta+1$ under Assumption~\ref{ass:StrongLevyProcess} and $\beta < \alpha/2$ under Assumption~\ref{ass:SpecialLevyProcess}. Weakly, Theorem~\ref{weakrEM} and Corollary~\ref{weakr_specEM} improve this to $h^{\min\{(\eta+1+\varrho)/2,1\}}$, where $\varrho < \eta + 2/p$ in the general setting, and $\varrho\le \eta + 1$ for Wiener noise or $\varrho<\alpha/2$ under Assumption~\ref{ass:SpecialLevyProcess} with i.i.d.~components, i.e., we obtain twice the strong rates. The dependence of the weak rates on the assumptions mirrors the spectral case.

We conclude this section by noting the asymmetric behavior between the spectral and temporal scheme in the case of Assumption~\ref{ass:SpecialLevyProcess}. In particular, in Section~\ref{Sec:SpectApp} we have convergence results for $\beta=\alpha/2$, while in this section we have shown the result up to $\beta<\alpha/2$. 
This is not a residual from the fact that we only computed the spectral error explicitly using the structure of Assumption~\ref{ass:SpecialLevyProcess}. Using a similar approach for the temporal error also only leads convergence results for $\beta<\alpha/2$, hence we do not gain anything from this proof and omit it here.

\section{Numerical simulation}\label{sec:Num}

We are now ready to support the theoretical conclusions from Sections \ref{Sec:SpectApp} and \ref{Sec:4} with numerical experiments. Specifically, we will compare the convergence rates of various errors for the spectral approximation, as well as for the forward and backward Euler–Maruyama schemes.

\begin{figure}[htb]
    \centering
    \begin{subfigure}[b]{0.24\textwidth}
        \includegraphics[width=\textwidth]{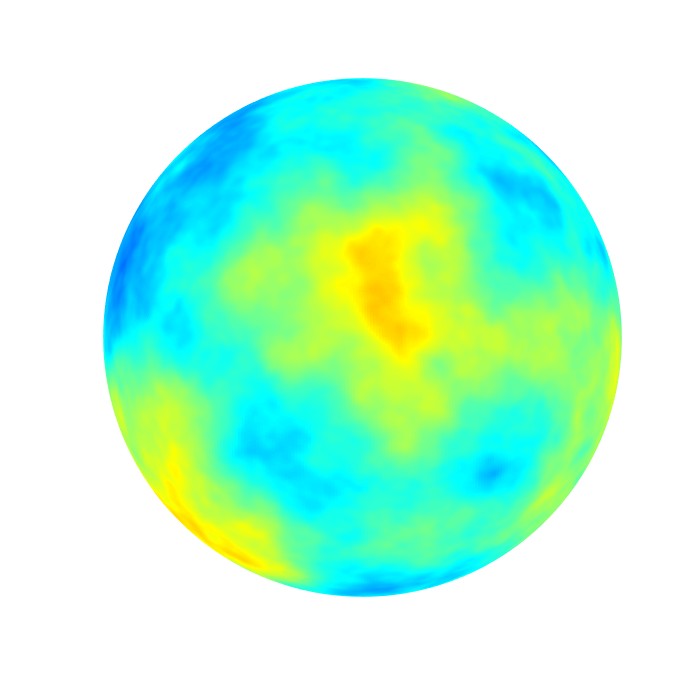}
        \caption{t=1}
        \label{fig:LP1-1}
    \end{subfigure}
    \begin{subfigure}[b]{0.24\textwidth}
        \includegraphics[width=\textwidth]{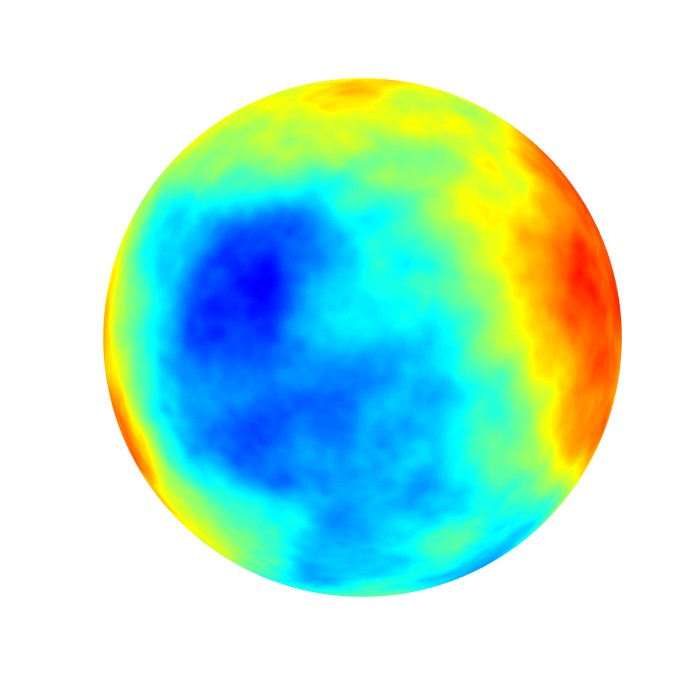}
        \caption{t=1}
        \label{fig:MLP3-2a}
    \end{subfigure}
    \begin{subfigure}[b]{0.24\textwidth}
        \includegraphics[width=\textwidth]{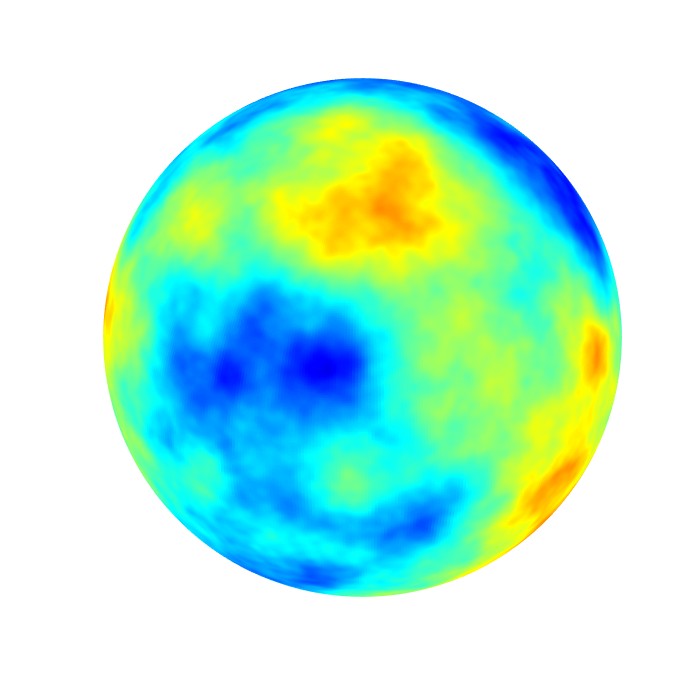}
        \caption{t=1}
        \label{fig:LP3-3a}
    \end{subfigure}
    \begin{subfigure}[b]{0.24\textwidth}
        \includegraphics[width=\textwidth]{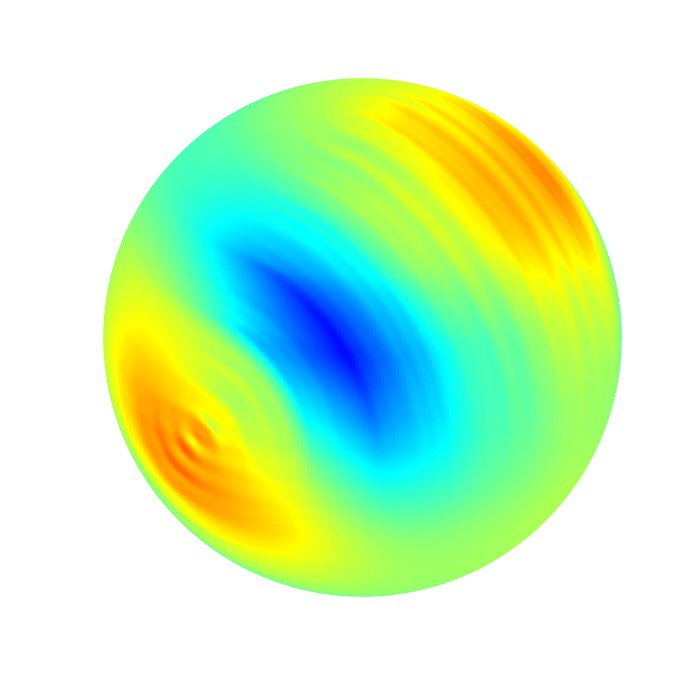}
        \caption{t=1}
        \label{fig:LP3-4a}
    \end{subfigure}
    \caption{Realization of different L\'evy processes $L$ on the sphere: (1) Wiener process, (2) Poisson process, (3) sum of a Wiener and a Poisson process, (4)  Poisson processes that are identical across all $m$ for a fixed $\ell$, but independent across different $\ell$.}
    \label{fig:sample}
\end{figure}

Although our theorems do not require assuming any particular structure in the decomposition components $L_{\ell,m}$ of our Lévy process~$L$, for the purpose of numerical simulation, we restrict ourselves to Assumption~\ref{ass:SpecialLevyProcess}. Specifically, we assume that $L_{\ell,m} = \ell^{-\alpha/2}\hat L_{\ell,m}$ for a given~$\alpha$, where the processes $\hat L_{\ell,m}$ are independent and identically distributed. Note that for the strong rates by our theory, it would be enough to assume identically distributed processes~$\hat L_{\ell,m}$. Indeed the strong convergence rate simulations obtain similar results assuming some correlation. However, we decided to present results based on independent noise processes here, such that also the assumptions for the weak convergence results are satisfied. In Figure~\ref{fig:sample}, we show realizations of such L\'evy processes~$L(t)$ with Brownian component, compensated Poisson component, the sum
of the two and a special case in which the real-valued processes $(L_{\ell,m})_{\ell,m}$ are compensated Poisson processes that are identical across all~$m$ for a fixed $\ell$, but independent across different $\ell$. In all cases, the intensity of the Poisson processes is~$1$, $\kappa = 2^7$, and $\alpha =4$.
More so, using the exponential transformation $\exp(L(t))$ on~$\S^2$, as in \cite{lang2015, cohen2022,lang2023}, Figure \ref{fig:SHEsample} shows realisations of the evolving stochastic noises through time. For these realisations, we set $\alpha=3$, $\kappa =2^7$ and add for the combination of the  Wiener and compensated Poisson processes the same samples of the separate processes.

\begin{figure}[htb]
    \centering
    \begin{subfigure}[b]{0.24\textwidth}
        \includegraphics[width=\textwidth]{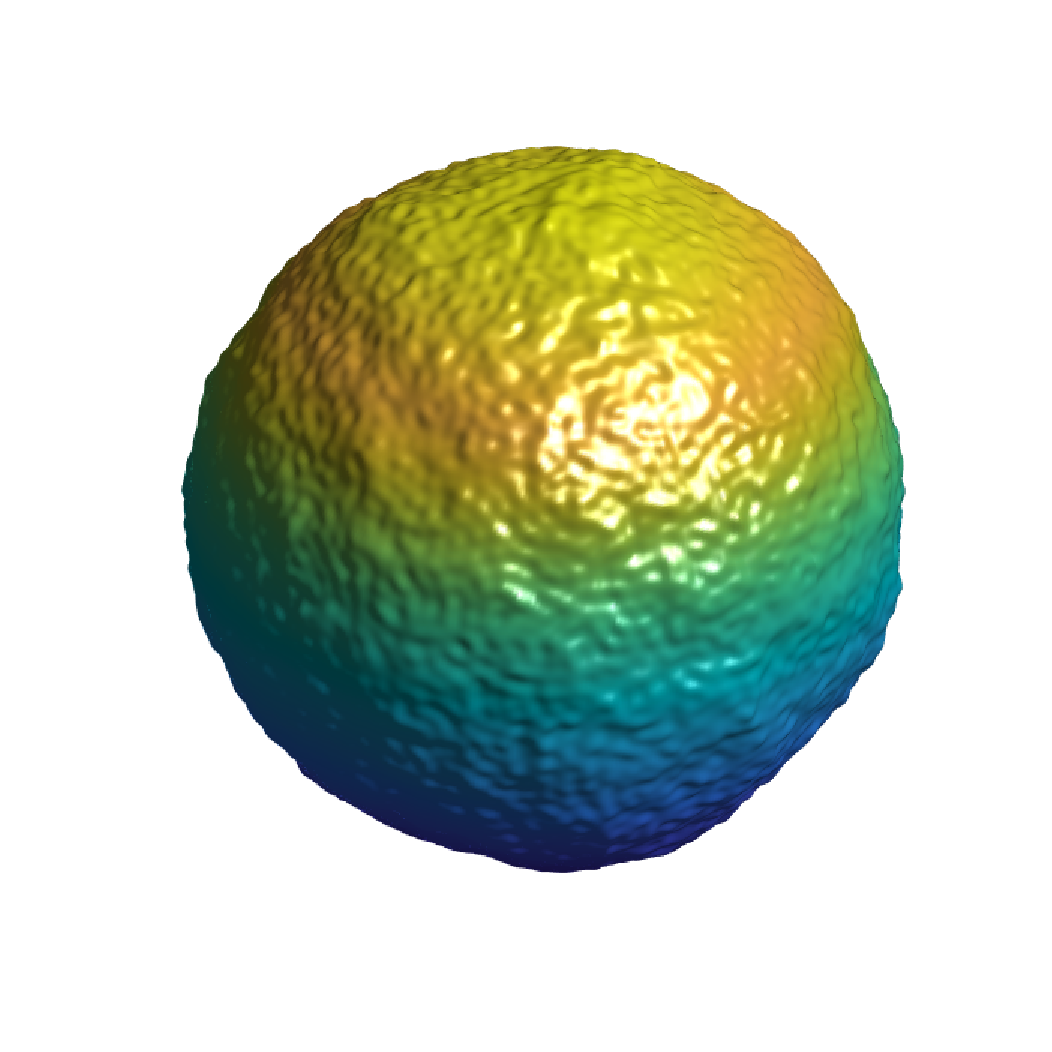}
        \caption{t = 0.0005}
        \label{fig:LP3-11}
    \end{subfigure}
    \begin{subfigure}[b]{0.24\textwidth}
        \includegraphics[width=\textwidth]{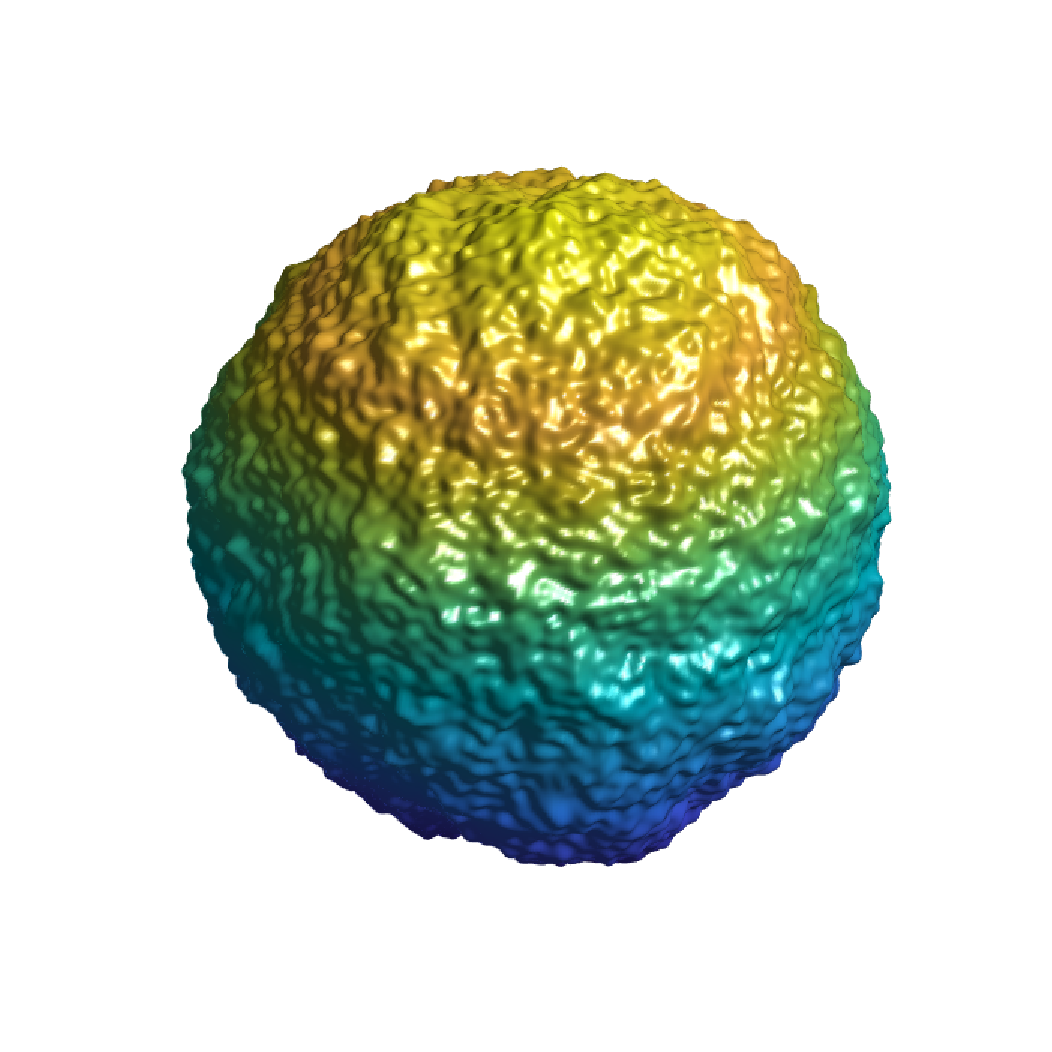}
        \caption{t = 0.0025}
        \label{fig:MLP3-2k}
    \end{subfigure}
    \begin{subfigure}[b]{0.24\textwidth}
        \includegraphics[width=\textwidth]{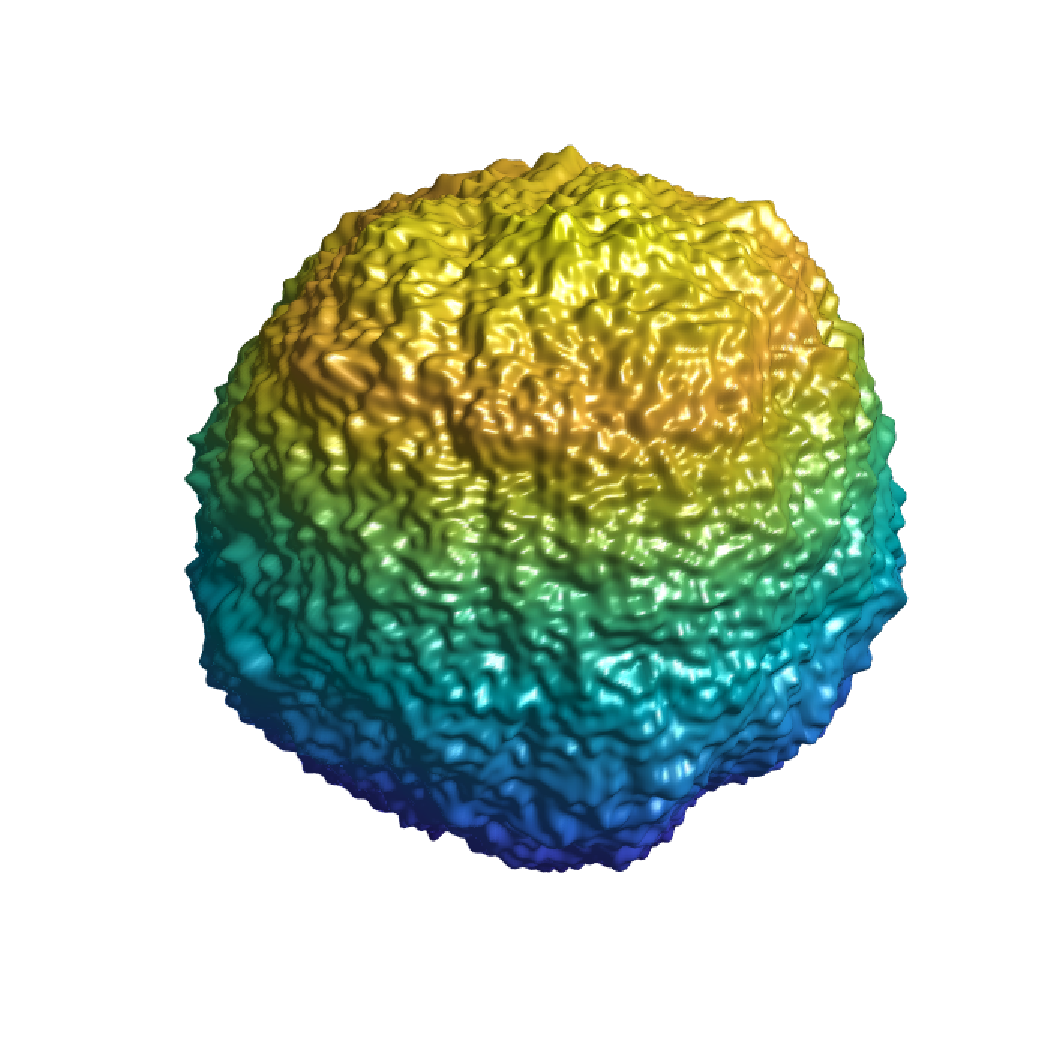}
        \caption{t= 0.005}
        \label{fig:LP3-3k}
    \end{subfigure}
    \begin{subfigure}[b]{0.24\textwidth}
        \includegraphics[width=\textwidth]{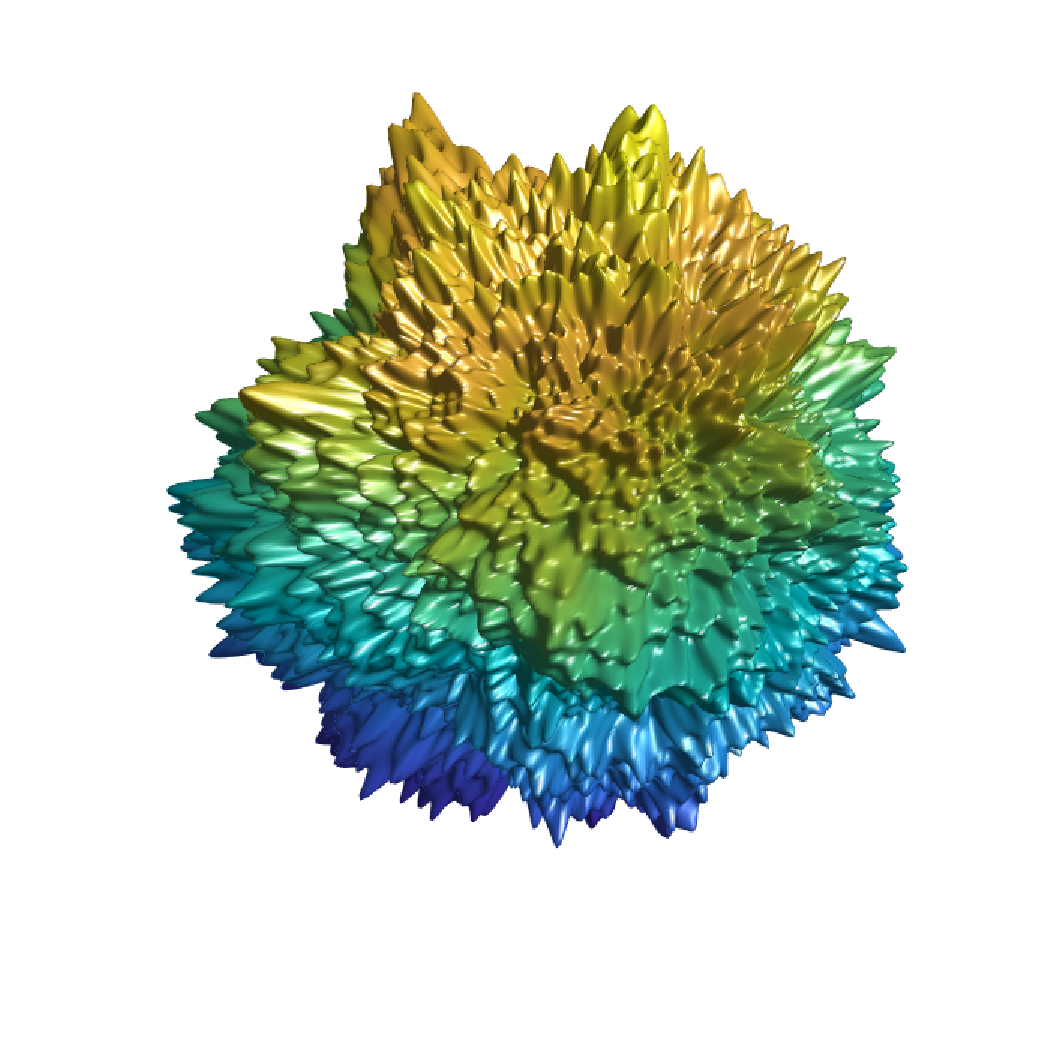}
        \caption{t= 0.05}
        \label{fig:LP3-4k}
    \end{subfigure}
    \begin{subfigure}[b]{0.24\textwidth}
        \includegraphics[width=\textwidth]{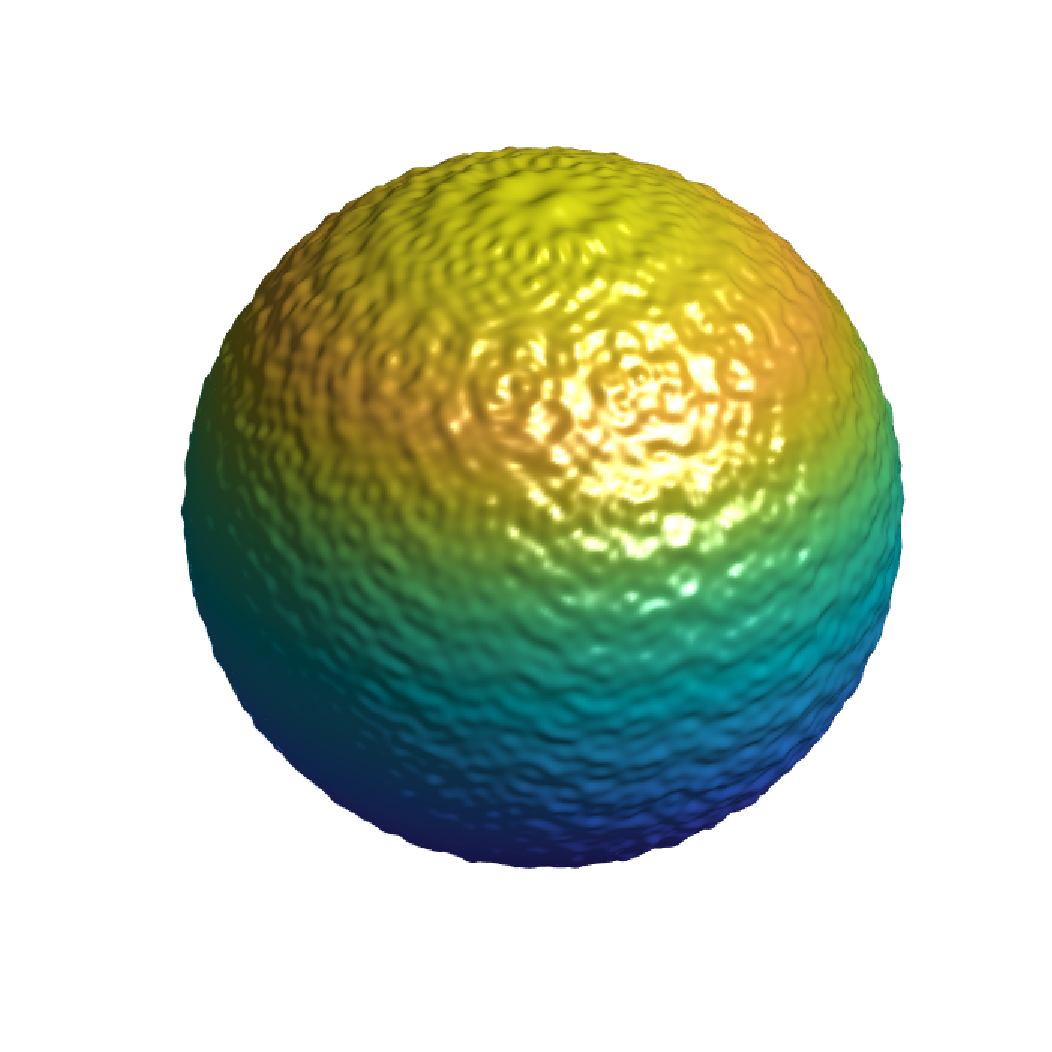}
        \caption{t = 0.0005}
        \label{fig:LP3-12}
    \end{subfigure}
    \begin{subfigure}[b]{0.24\textwidth}
        \includegraphics[width=\textwidth]{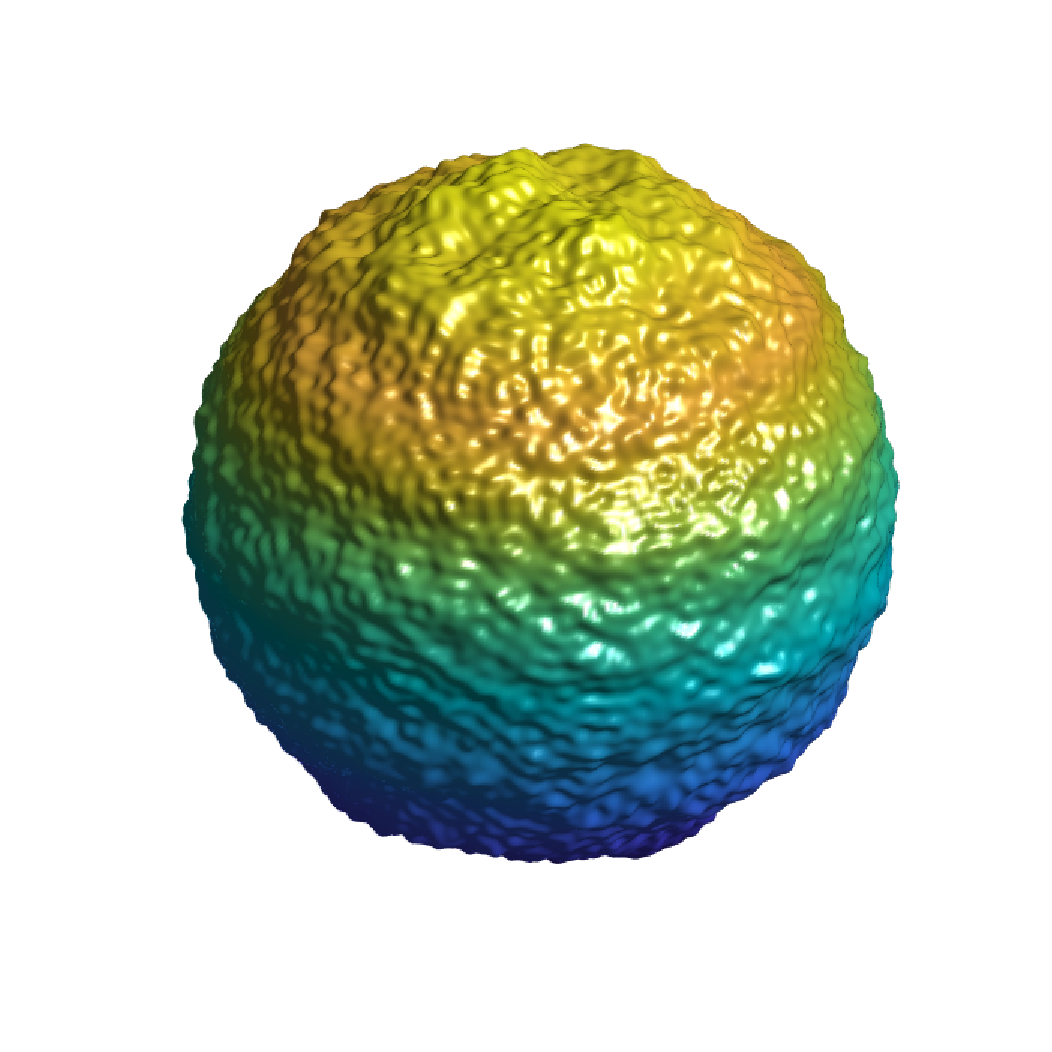}
        \caption{t = 0.0025}
        \label{fig:MLP3-2g}
    \end{subfigure}
    \begin{subfigure}[b]{0.24\textwidth}
        \includegraphics[width=\textwidth]{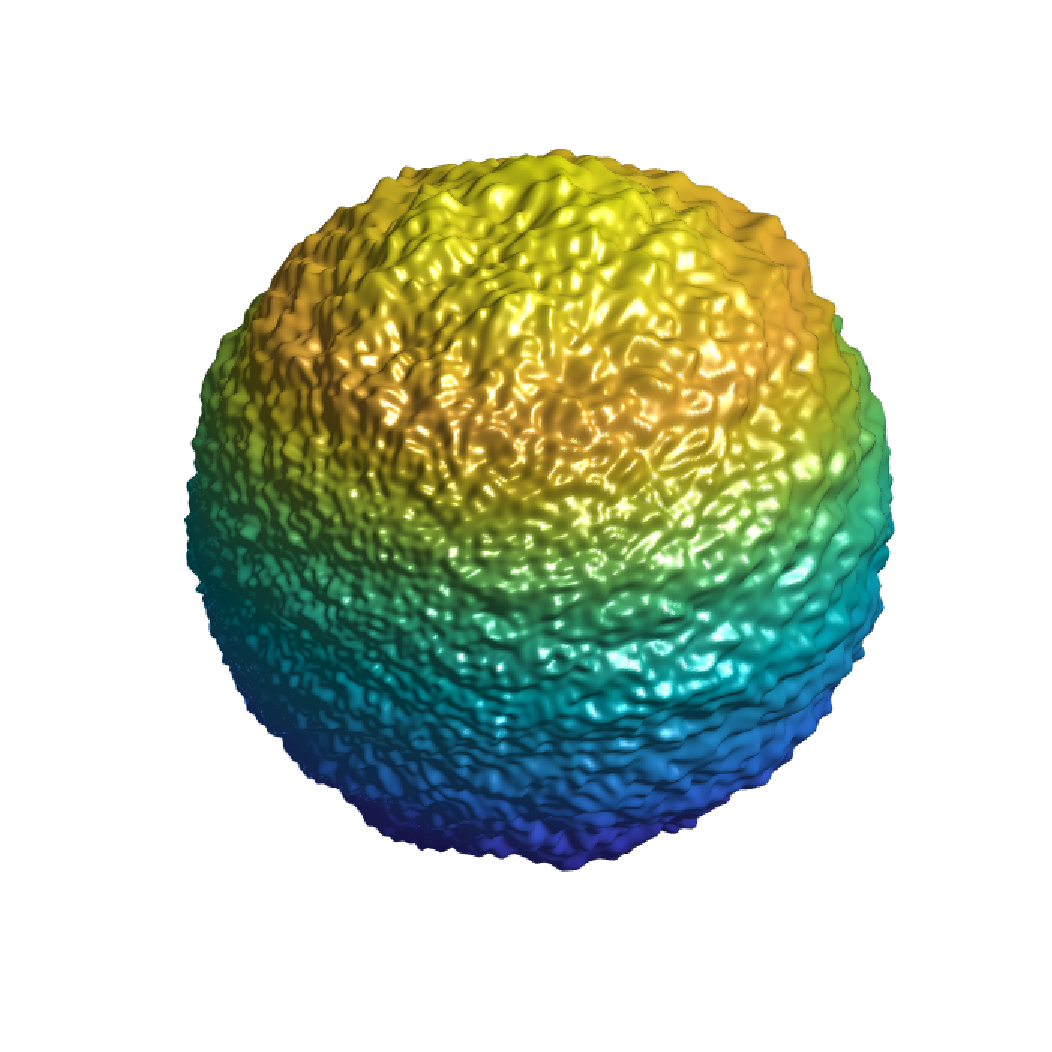}
        \caption{t= 0.005}
        \label{fig:LP3-3g}
    \end{subfigure}
    \begin{subfigure}[b]{0.24\textwidth}
        \includegraphics[width=\textwidth]{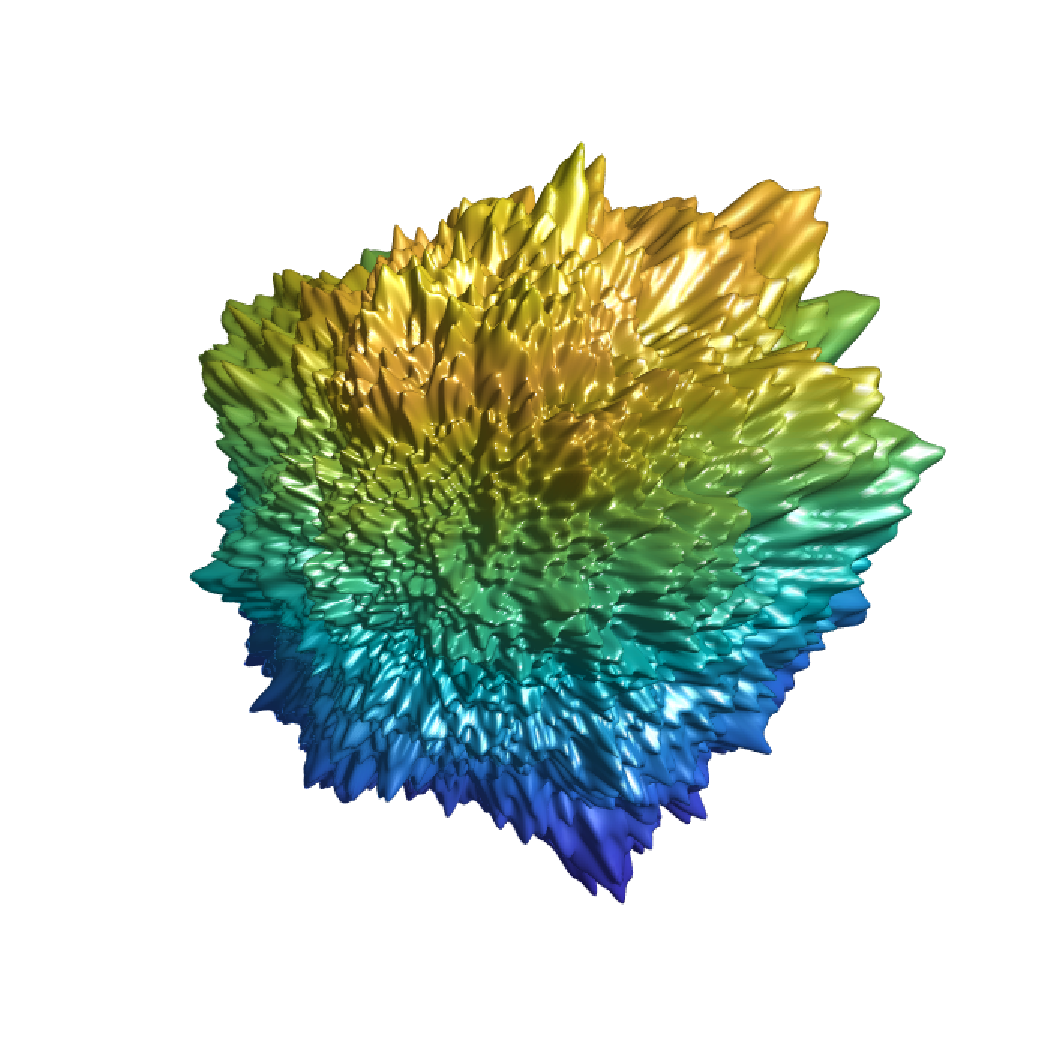}
        \caption{t= 0.05}
        \label{fig:LP3-4g}
    \end{subfigure}
    \begin{subfigure}[b]{0.24\textwidth}
        \includegraphics[width=\textwidth]{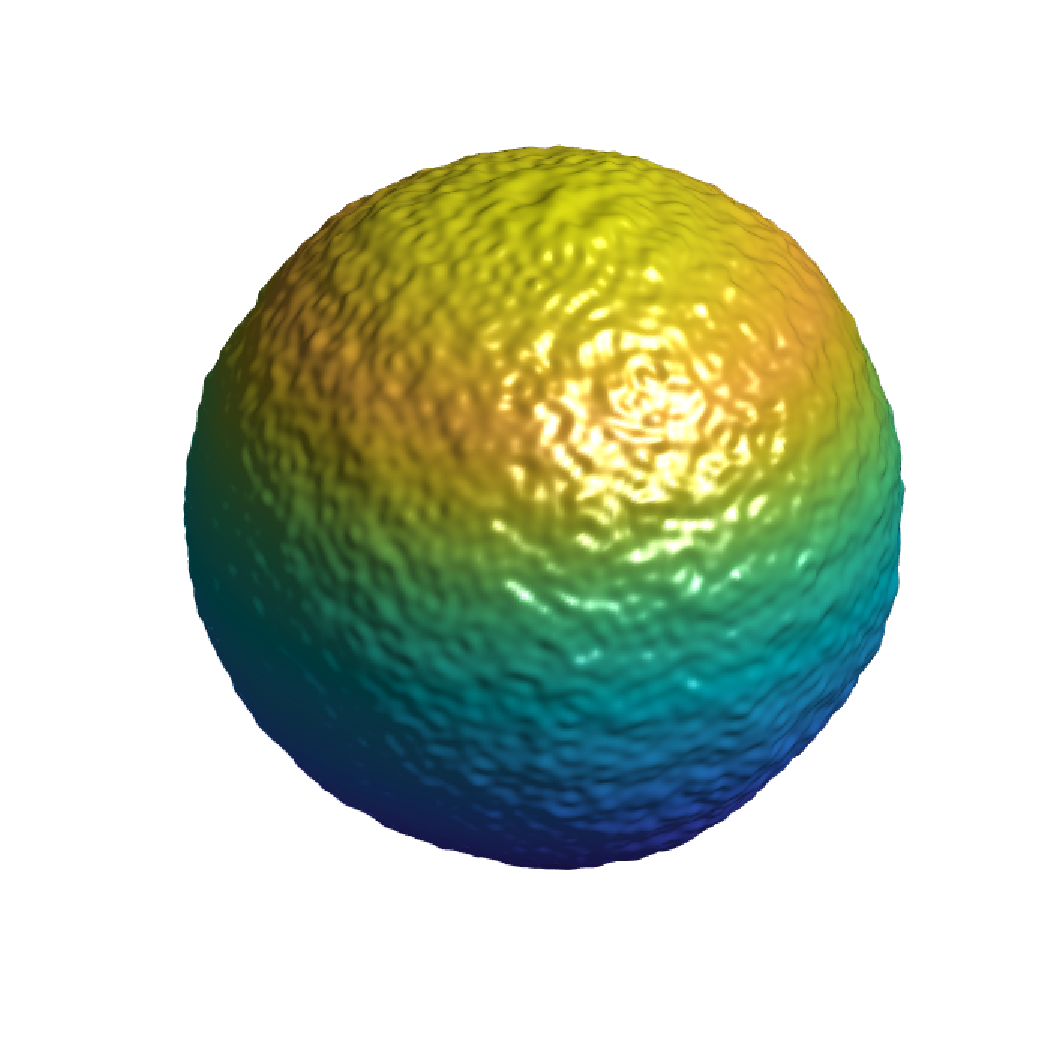}
        \caption{t = 0.0005}
        \label{fig:LP3-13}
    \end{subfigure}
    \begin{subfigure}[b]{0.24\textwidth}
        \includegraphics[width=\textwidth]{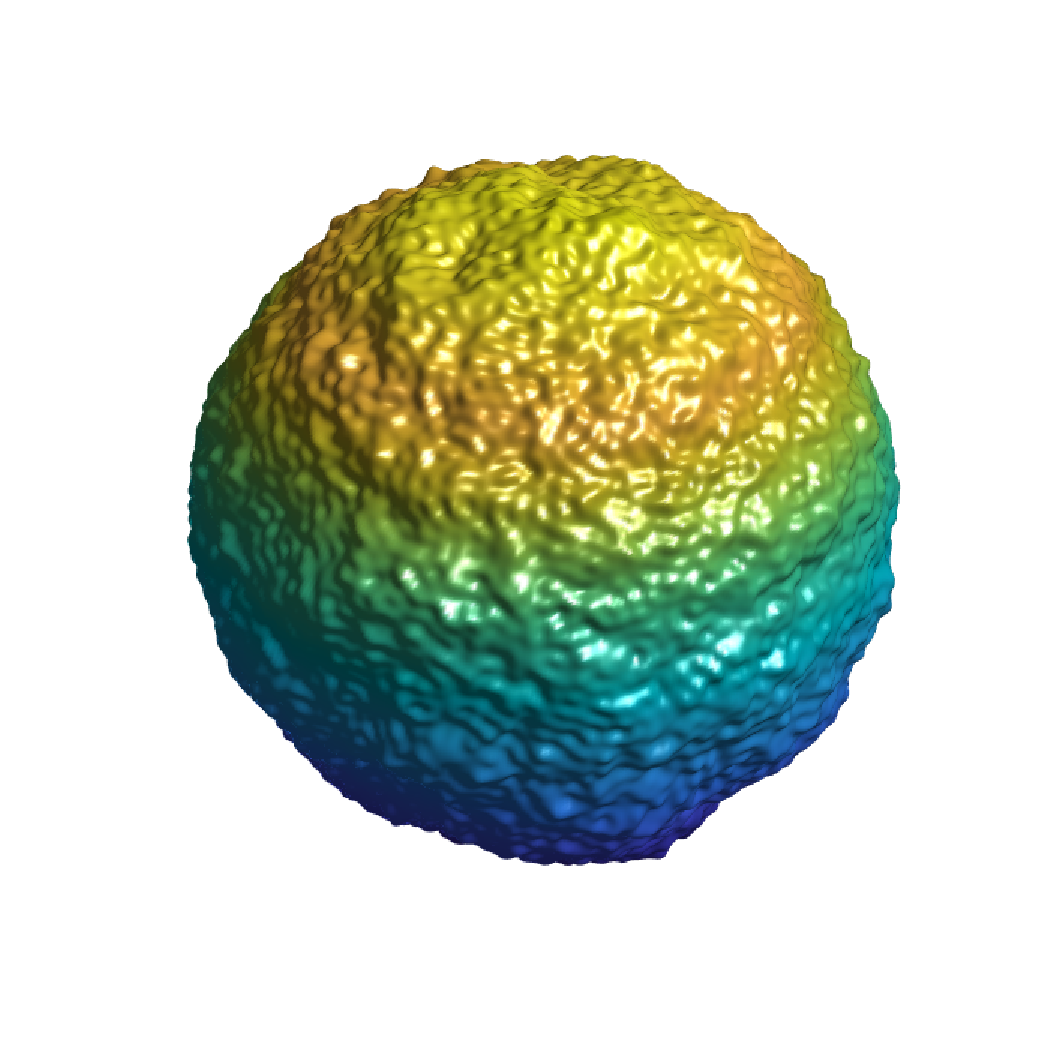}
        \caption{t = 0.0025}
        \label{fig:MLP3-2f}
    \end{subfigure}
    \begin{subfigure}[b]{0.24\textwidth}
        \includegraphics[width=\textwidth]{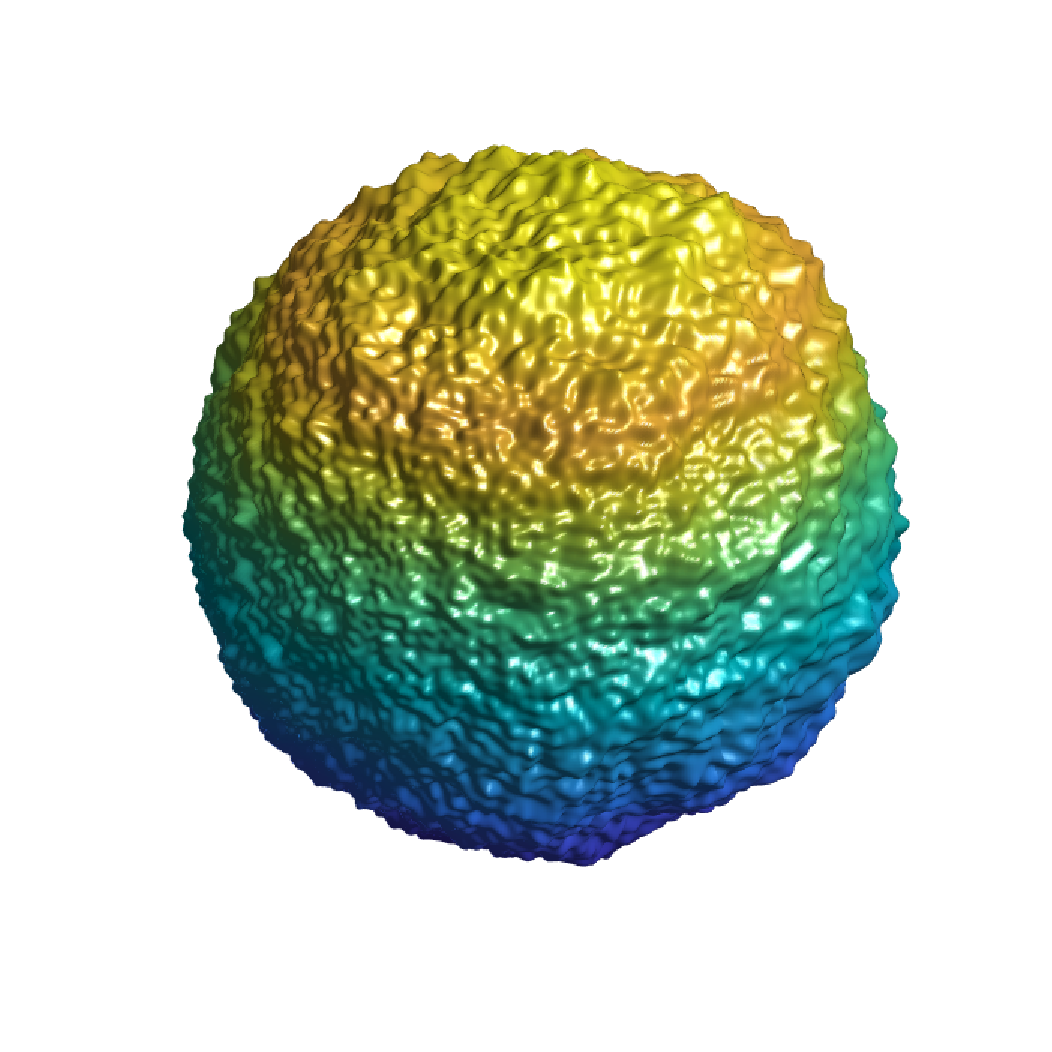}
        \caption{t= 0.005}
        \label{fig:LP3-3f}
    \end{subfigure}
    \begin{subfigure}[b]{0.24\textwidth}
        \includegraphics[width=\textwidth]{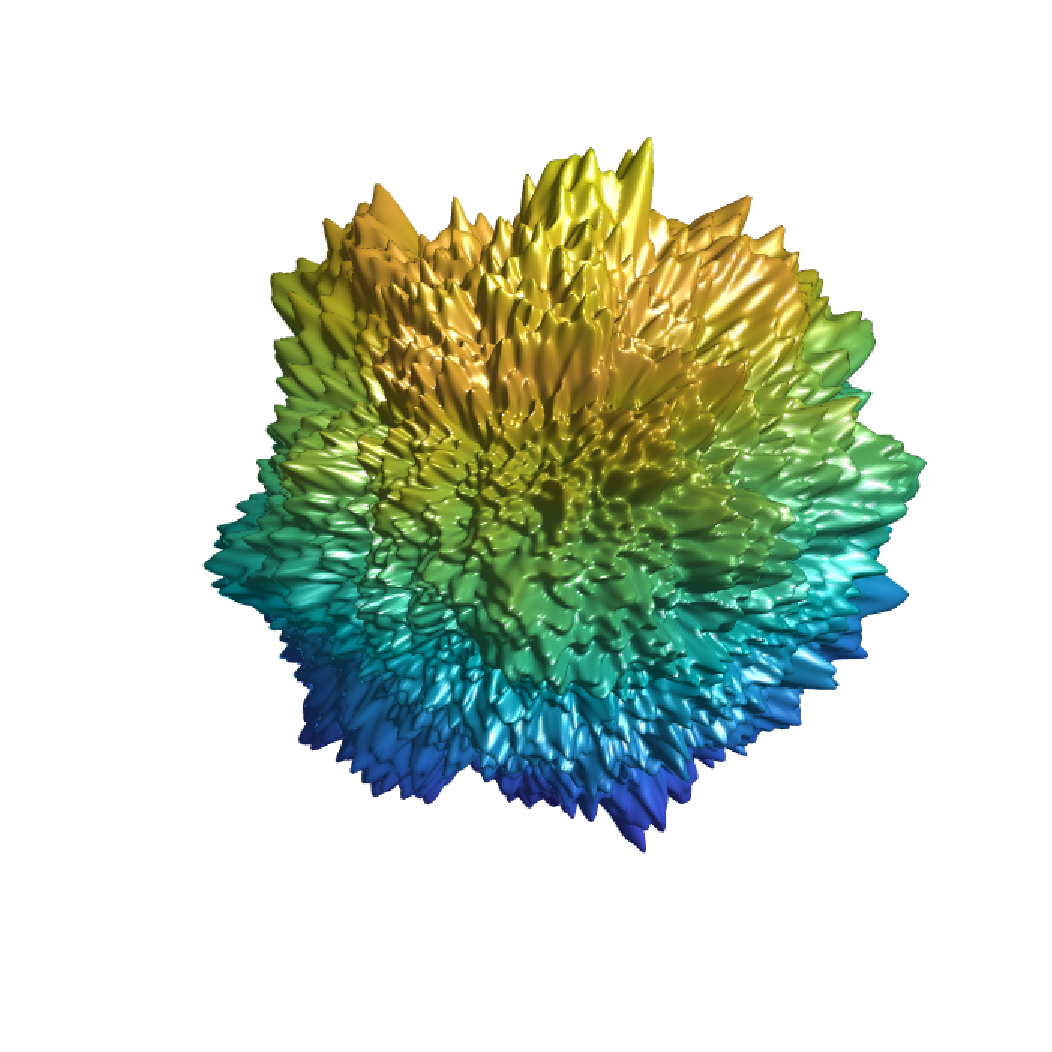}
        \caption{t= 0.05}
        \label{fig:LP3-4f}
    \end{subfigure}
    \caption{Sample of the noise evolving in time: (1)-(4) Wiener process, (5)-(8) Poisson process, (9)-(12) Sum of a Wiener and a Poisson process.}
    \label{fig:SHEsample}
\end{figure}

For the spectral approximation, we perform simulations using a reference solution at $\kappa=2^{10}$ and unit time $T=1$. The approximations are computed at different truncation levels $\kappa=2^j,$ with $j=0,...,9$.
In Figure \ref{fig:Spectral} we assume initial condition $X^0=0$ and a Lévy process with components $L_{\ell,m} = \ell^{-\alpha/2}\hat L_{\ell,m}$ for a given~$\alpha$, where the processes $\hat L_{\ell,m}$ are independent and identically distributed Poisson process with intensity~$1$. Note, this Lévy process satisfies Assumption \ref{ass:SpecialLevyProcess}.
In Figure~\ref{fig:SEs}, we present the strong error computed explicitly. 
The rates obtained for $\alpha=1,...,5$ confirm the theoretical results in Theorem~\ref{L:SA1}. 
In Figure~\ref{fig:MEs}, we observe the behavior of the expectation, which differs from the classical case of Wiener noise, since $\E[\hat{L}_{\ell,m}(1)]=1$. As shown in \cite{lang2023}, in that case only the exponential decay of the initial condition is observed, whereas here we see a decay in $\kappa$ that depends on the regularity of the mean of the driving noise confirming Corollary~\ref{L:SA2}. Similarly, as shown in Figure~\ref{fig:VEs}, the convergence of the second moment validates the theoretical results of Corollary~\ref{L:SA2}.

\begin{figure}[t]
    \centering
    \begin{subfigure}[b]{0.31\textwidth}
        \includegraphics[width=\textwidth]{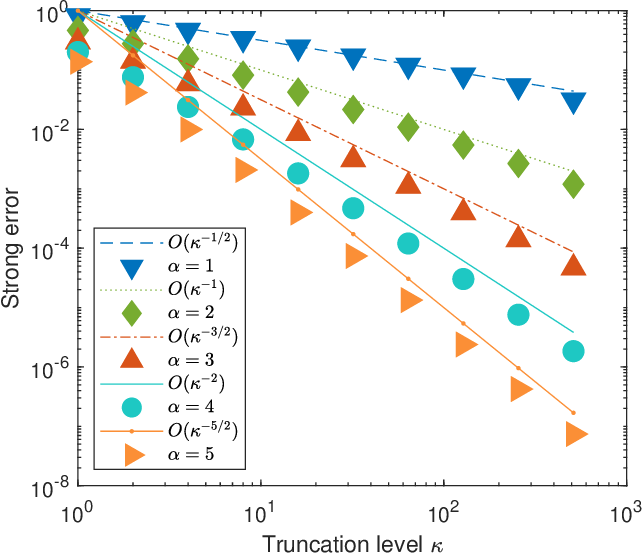}
        \caption{Strong error.}
        \label{fig:SEs}
    \end{subfigure}
    \begin{subfigure}[b]{0.31\textwidth}
        \includegraphics[width=\textwidth]{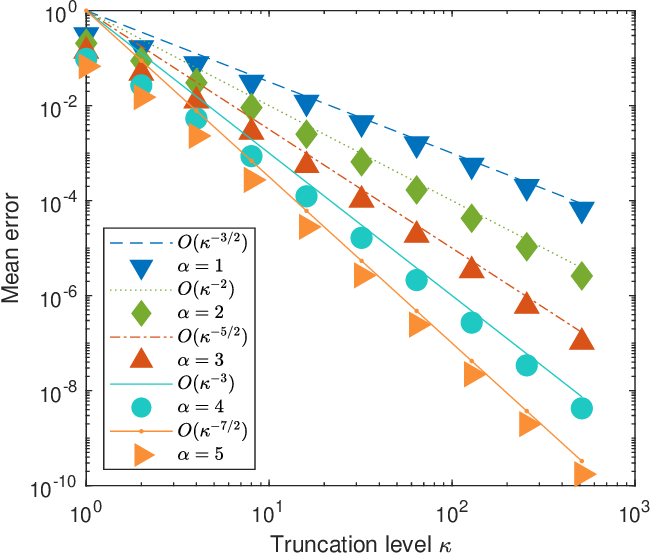}
        \caption{Mean error.}
        \label{fig:MEs}
    \end{subfigure}
    \begin{subfigure}[b]{0.31\textwidth}
        \includegraphics[width=\textwidth]{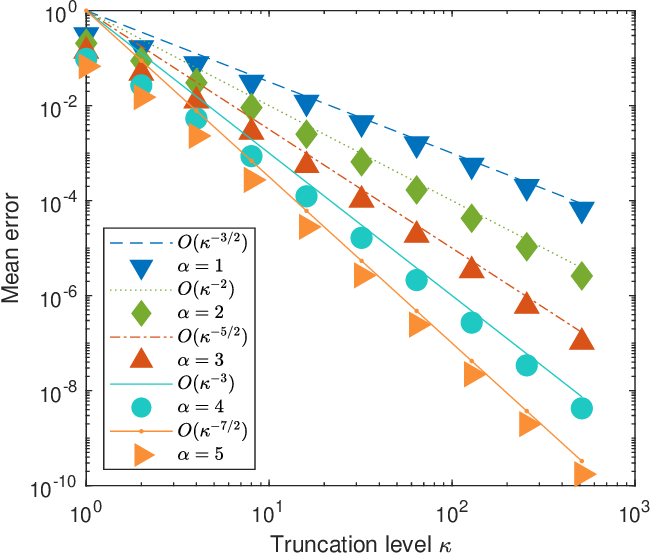}
        \caption{Second moment error.}
        \label{fig:VEs}
    \end{subfigure}
    \caption{Convergence of the spectral method for different $\alpha$.}
    \label{fig:Spectral}
\end{figure}

\begin{figure}[t]
    \centering
    \begin{subfigure}[b]{0.31\textwidth}
        \includegraphics[width=\textwidth]{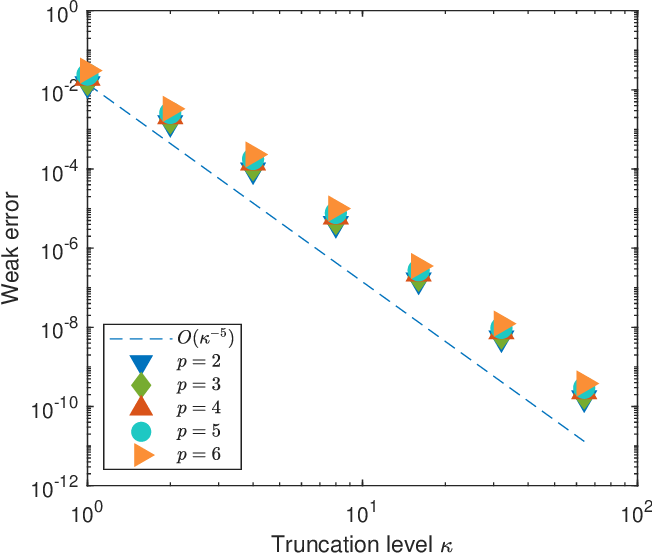}
        \caption{Fixed $\alpha=5$, Wiener.}
        \label{fig:WR_BM_p}
    \end{subfigure}
    \begin{subfigure}[b]{0.31\textwidth}
        \includegraphics[width=\textwidth]{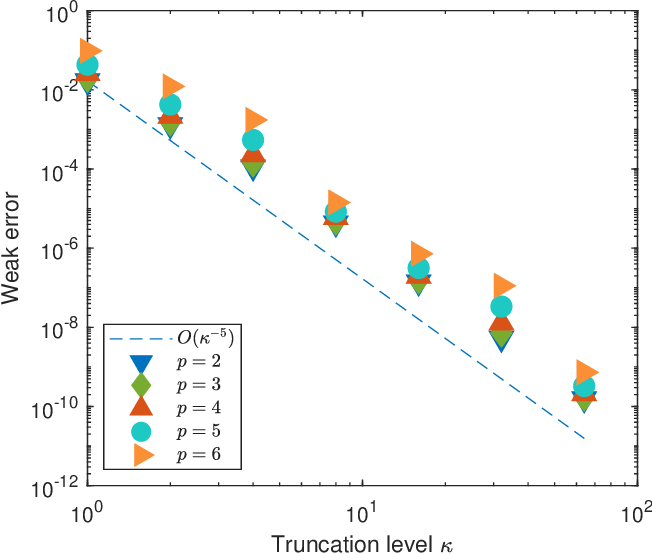}
        \caption{Fixed $\alpha=5$, Poisson.}
        \label{fig:WR_PP_p}
    \end{subfigure}
    \begin{subfigure}[b]{0.31\textwidth}
        \includegraphics[width=\textwidth]{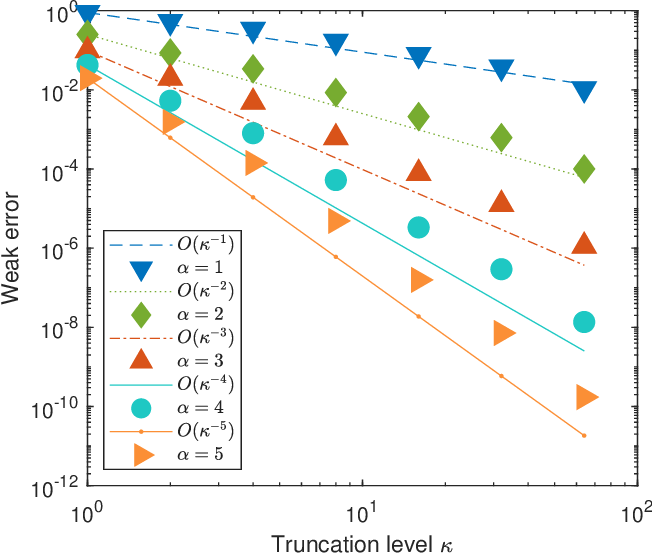}
        \caption{Fixed $p=3$, Poisson.}
        \label{fig:WR_BM_Alp}
    \end{subfigure}
    \caption{Spectral weak convergence for different processes~$L$, $p$, and $\alpha$: (1) Weak error of Wiener process with fixed $\alpha=5$, (2) Weak error of compensated Poisson process with fixed $\alpha=5$, (3) Weak error of compensated Poisson process with fixed $p=3$.}
\end{figure}
To confirm the weak convergence rates in Theorem~\ref{weakr}, we consider a reference solution at time $T=1$ with $\kappa=2^7$ with initial condition $X^0=0$.
The test functions are given by $\varphi(X)=\|X\|^p_{L^2(\S^2)}$, where $p \in \{2,3,4,5,6\}$, and we perform Monte Carlo simulations with 100 samples. Although this number is small, it is sufficient, since the same realisations are used at every discretisation level as well as for the reference solution. This introduces a correlation between the corresponding estimators, substantially reducing the variance as shown in~\cite{LP18}.

In Figure~\ref{fig:WR_BM_p}, we show that the weak rates for a Wiener noise with a covariance decay proportional to $\ell^{-5}$ are independent of~$p$ as proven in Theorem~\ref{weakr}.
In Figures~\ref{fig:WR_PP_p}, we simulate the weak rates for a Lévy process whose components are independent compensated Poisson processes with decay rate proportional to~$\ell^{-\alpha/2}$. We observe agreement of the decay with respect to the parameter~$\alpha$ and no dependence on $p$ on the rate of convergence as also shown in Theorem~\ref{weakr}.  
Furthermore, Figure~\ref{fig:WR_BM_Alp} illustrates for $p=3$ the dependence of the decay on the regularity parameter~$\alpha$ in agreement with the theoretical predictions.
We notice that we have assumed the same structure of the noise as in Figure \ref{fig:SHEsample}, i.e., the independence of the components $(L_{\ell,m})_{\ell,m}$. This limitation, arising from the simplified numerical framework, and how to surpass it, will be addressed in future work.

\begin{figure}[t]
    \centering
    \begin{subfigure}[b]{0.31\textwidth}
        \includegraphics[width=\textwidth]{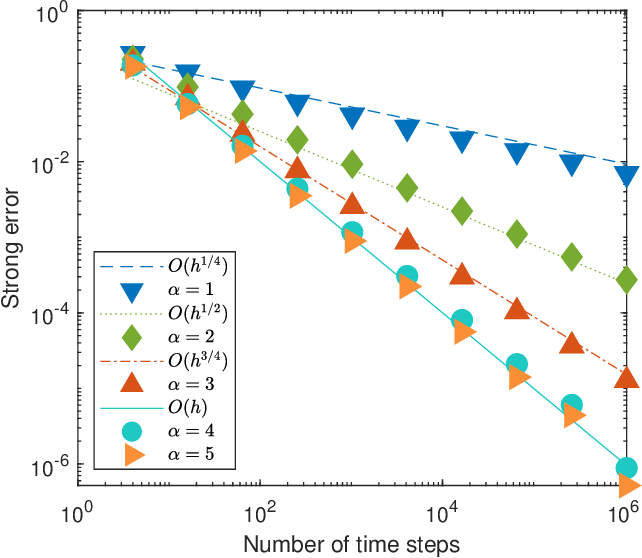}
        \caption{Strong error.}
        \label{fig:SEem}
    \end{subfigure}
    \begin{subfigure}[b]{0.31\textwidth}
        \includegraphics[width=\textwidth]{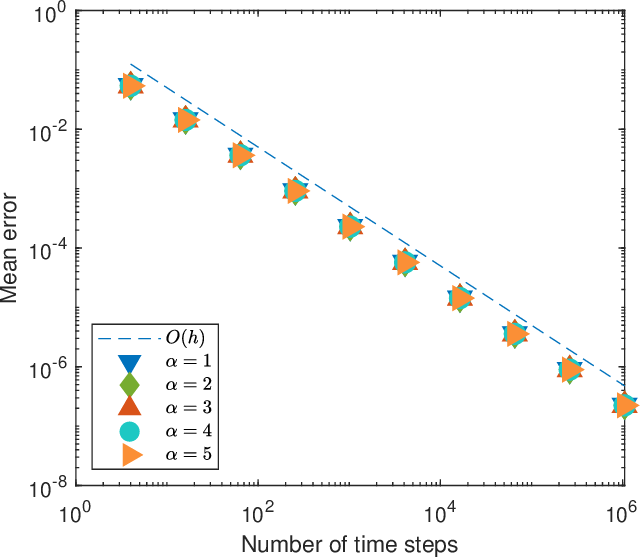}
        \caption{Mean error.}
        \label{fig:MEem}
    \end{subfigure}
    \begin{subfigure}[b]{0.31\textwidth}
        \includegraphics[width=\textwidth]{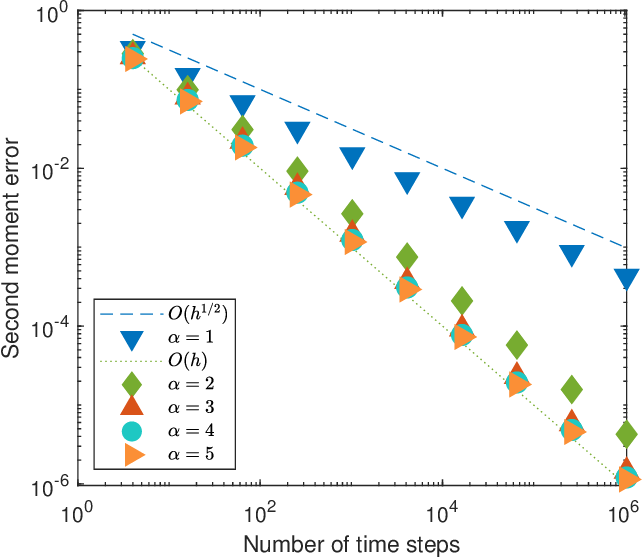}
        \caption{Second moment error.}
        \label{fig:SmEem}
    \end{subfigure}
    \caption{Convergence of the Euler--Maruyama method for different $\alpha$.}
\end{figure}

\begin{figure}[t]
    \centering
    \begin{subfigure}[b]{0.24\textwidth}
        \includegraphics[width=\textwidth]{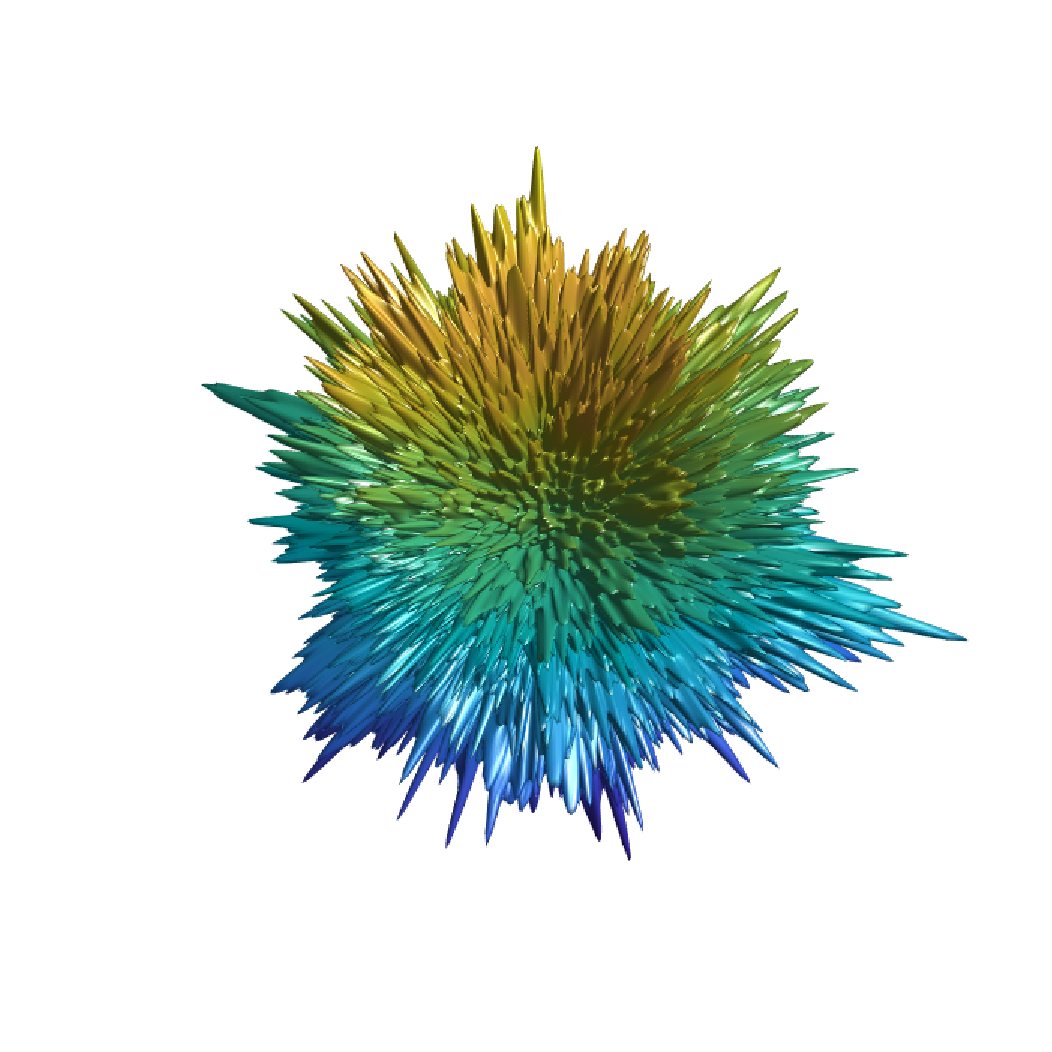}
        \caption{t = 0}
        \label{fig:LP3-1}
    \end{subfigure}
    \begin{subfigure}[b]{0.24\textwidth}
        \includegraphics[width=\textwidth]{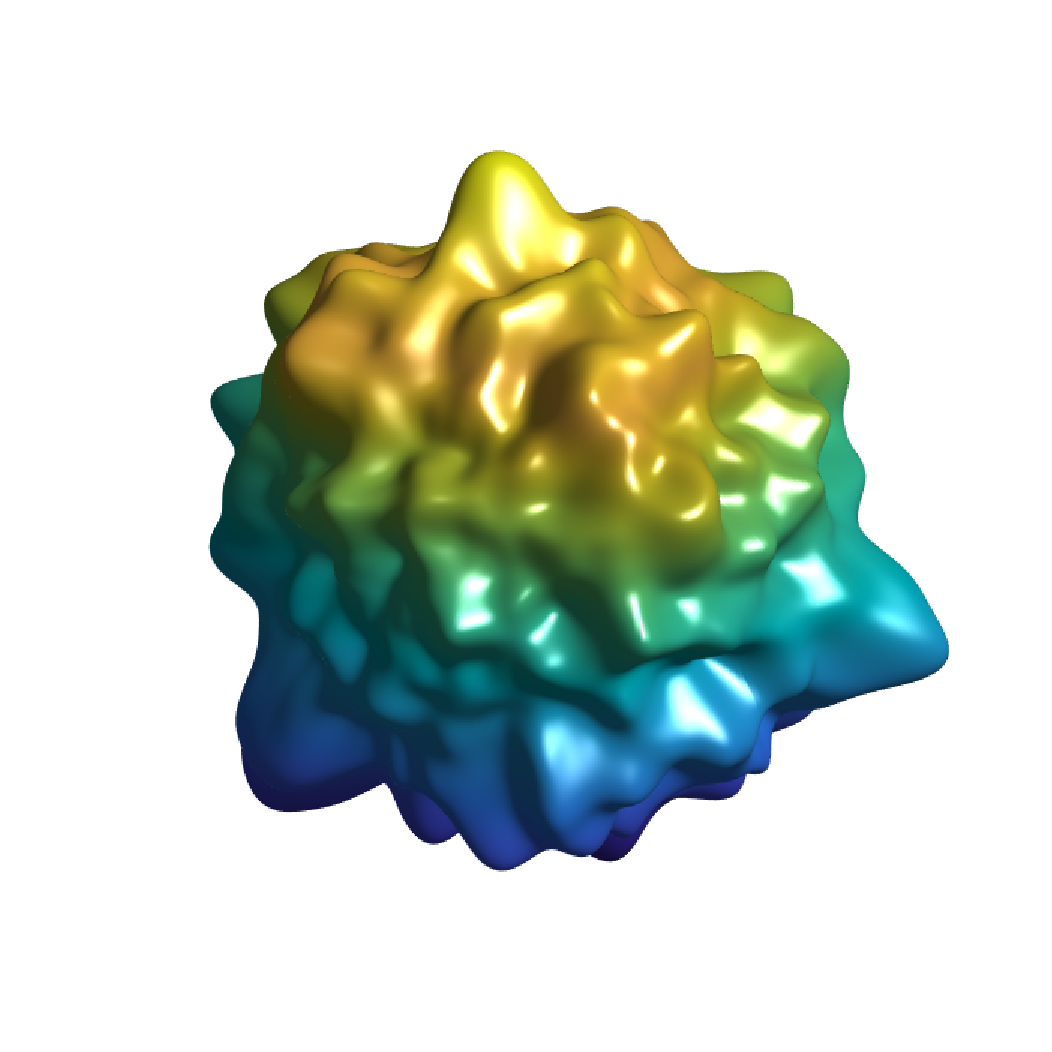}
        \caption{t = 0.0025}
        \label{fig:MLP3-2}
    \end{subfigure}
    \begin{subfigure}[b]{0.24\textwidth}
        \includegraphics[width=\textwidth]{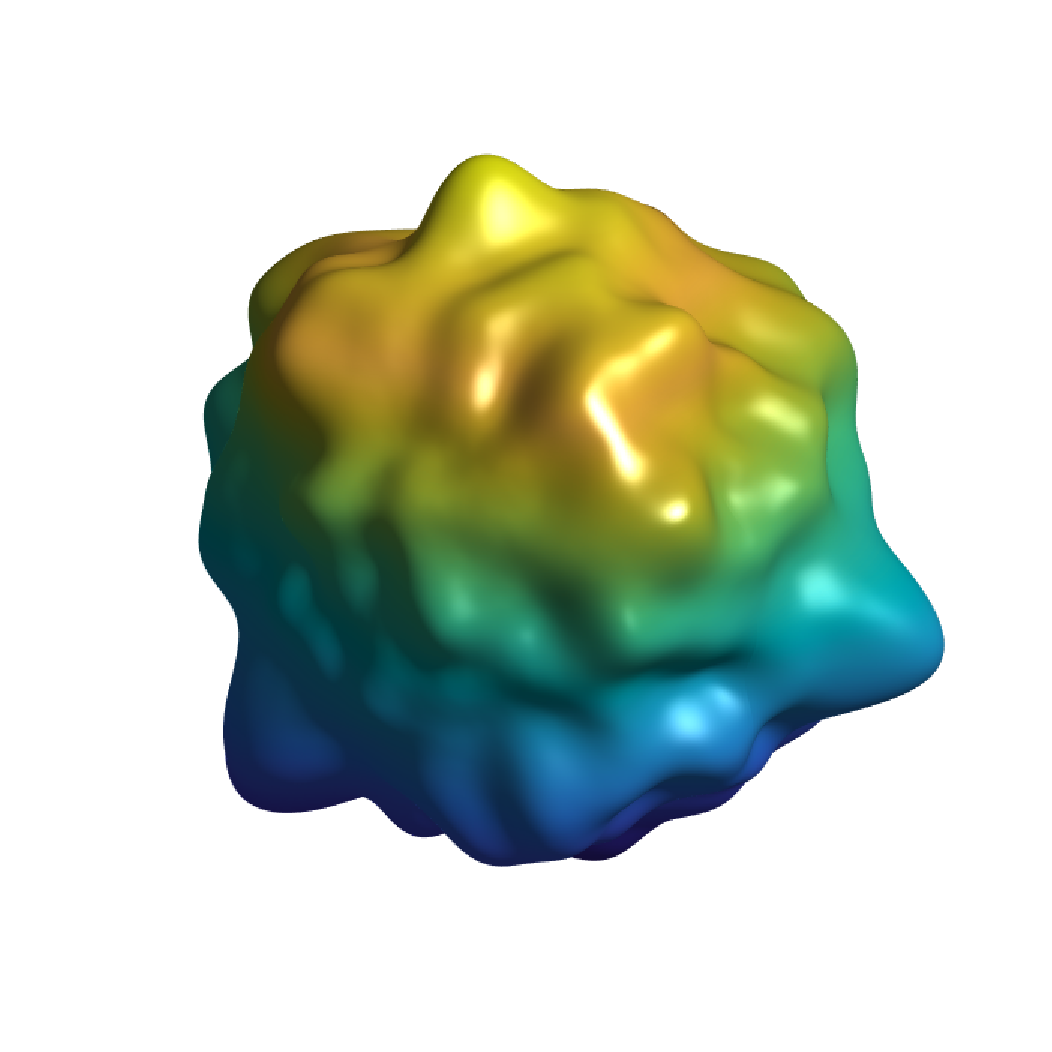}
        \caption{t= 0.005}
        \label{fig:LP3-3}
    \end{subfigure}
    \begin{subfigure}[b]{0.24\textwidth}
        \includegraphics[width=\textwidth]{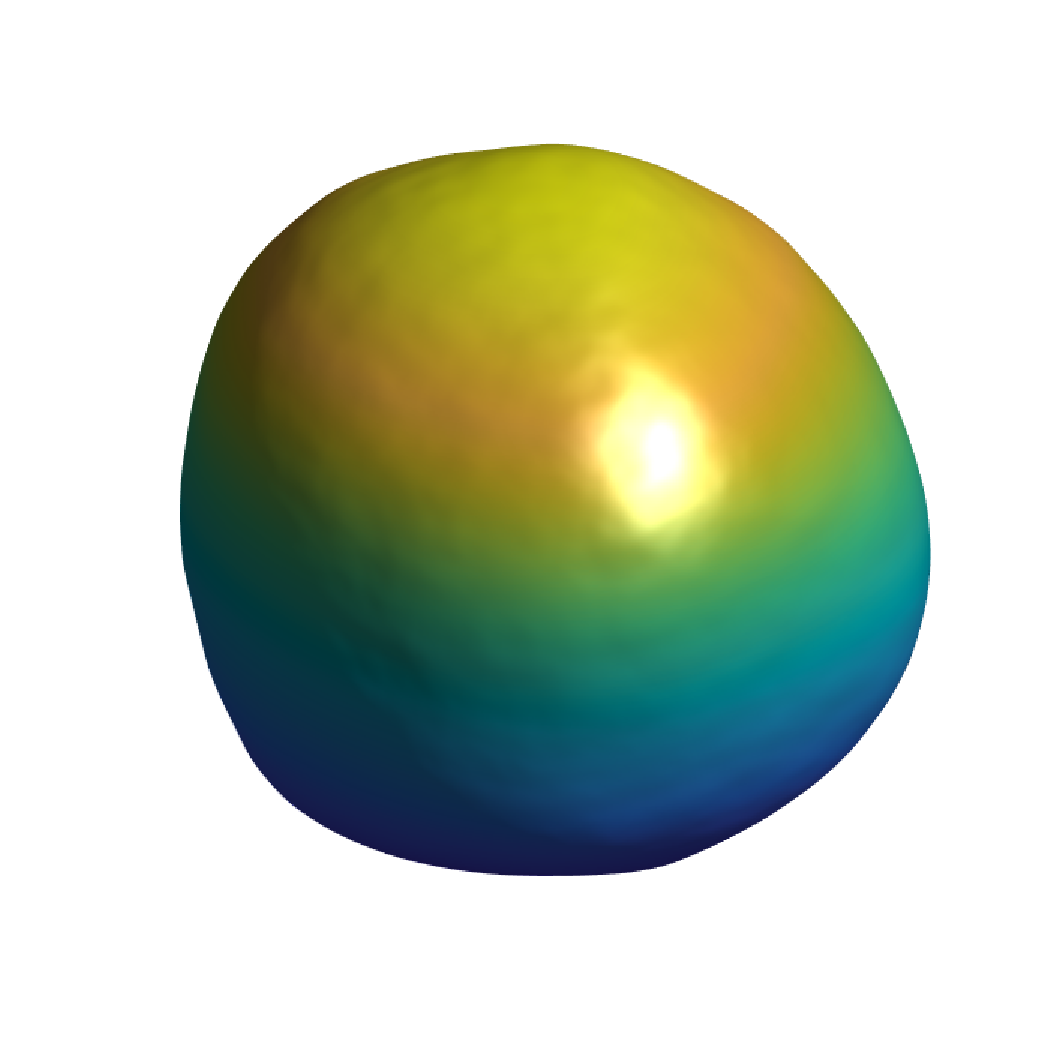}
        \caption{t= 0.05}
        \label{fig:LP3-4}
    \end{subfigure}
    \caption{Sample of the stochastic heat equation on the sphere evolving in time driven by the sum of a Wiener and a compensated Poisson process.}
    \label{fig:sampleSHEmix}
\end{figure}

Having verified the spectral convergence, we now turn to simulating the time discretization using the forward and backward Euler--Maruyama schemes. For this, we focus on the error between $X^{(\kappa)}$ and $X^{(\kappa,h)}$. Simulations are performed on time grids with step size $h=2^{-2m}$ for $m=1,...,10$, coupled with $\kappa=2^m$ to ensure stability for the forward Euler–Maruyama scheme and to satisfy the estimates for the backward scheme. We assume the same setting as in Figure \ref{fig:Spectral}.

The results for the backward Euler–Maruyama scheme in Figure~\ref{fig:SEem}, computed explicitly, with a reference solution using $h=2^{-10}$ and $\kappa=2^5$, confirm the expected convergence rate of $h^{\min\{\alpha/4,1\}}$ from Theorem~\ref{Thm:ConvEM}.
Figure \ref{fig:MEem} shows the simulated convergence of the expectation for $\alpha\in\{1,2,3,4,5\}$. 
Finally, in Figure~\ref{fig:SmEem}, we see the convergence of the second moment for the Euler--Maruyama scheme. The setting is the same as for the strong error. The results show the expected convergence rate $h^{\min\{\alpha/2,1\}}$ as in Theorem~\ref{Thm:ConvEM2}.

We conclude the section showing in Figure~\ref{fig:sampleSHEmix} the evolution of the solution~$X(t)$ in time, with a sum of Brownian and compensated Poisson components, represented with the exponential transformation $\exp(X(t))$ as in Figure~\ref{fig:SHEsample}. For this, we take the same sample of the sum of a Wiener and a Poisson process as in Figure~\ref{fig:SHEsample} and start from a rough initial condition, i.e., in $L^2(\Omega, H^{-1}(\S^2))$.

\end{document}